\documentclass[12,reqno]{amsart}

\usepackage{amsfonts}
\usepackage[active]{srcltx}
\usepackage{amsmath}
\usepackage{amssymb}
\usepackage{amsthm}
\usepackage{mathrsfs}
\usepackage{mathtools}
\usepackage{bm} 
\usepackage{empheq}
\usepackage{mathptmx}

\usepackage{enumitem}
\usepackage{wasysym}
\usepackage{verbatim}
\usepackage{graphicx}
\usepackage[
bookmarks=true,         
bookmarksnumbered=true, 
colorlinks=true, pdfstartview=FitV, linkcolor=blue, citecolor=blue,
urlcolor=blue]{hyperref}
\usepackage{microtype}
\theoremstyle{plain}\newtheorem{theorem}{Theorem}[section]

\newtheorem{lemma}[theorem]{Lemma}
\newtheorem{problem}[theorem]{Problem}
\newtheorem{proposition}[theorem]{Proposition}

\renewenvironment{proof}[1][Proof]{\textbf{#1.} }{\ \rule{0.5em}{0.5em} \par }
\theoremstyle{remark}

\theoremstyle{definition}
\newtheorem{remark}[theorem]{Remark}
\newtheorem{definition}[theorem]{Definition}

\mathchardef\mhyphen="2D 

\def\PP{\mathbb{P}}
\def\RR{\mathbb{R}}
\def\EE{\mathbb{E}}
\def\UU{\mathbb{U}}
\def\FF{\mathbb{F}}

\def\cA{{\mathcal A}}

\def\cC{{\mathcal C}}
 
\def\cG{{\mathcal G}}

\def\cP{{\mathcal P}}

\def\si{{\sigma}}

\def\cF{{\mathcal F}}
\def\cP{{\mathcal  P}}

\def\cL{{\mathcal L}}
\def\Om{{\Omega}}
\def\om{{\omega}}
\def\al{{\alpha}}

\def\Ga{{\Gamma}}

\def\si{{\sigma}}

\def\La{{\Lambda}}
\def\vare{{\varepsilon}}

\def\si{{\sigma}}
\def\al{{\alpha}}

\newcommand{\gga}{\Gamma}

\newcommand{\vp}{\varphi}

\def\XX{\mathbb{X}}
\def\YY{\mathbb{Y}}

\let\Section=\section
\def\section{\setcounter{equation}{0}\Section}

\begin{document}
 
\title[Maximum Principle for  Partially Observed Fractional Stochastic Systems]{The Global Maximum Principle for Optimal Control of Partially Observed Stochastic Systems Driven by Fractional Brownian Motion}
  \author{\sc Yueyang Zheng    }
\address{School of Mathematics, Shandong University, Jinan, Shandong 250100, China, and
School of Mathematical Sciences, Fudan University, Shanghai 200433, China}
  \email{ zhengyueyang0106@163.com}
\author{ \sc   Yaozhong Hu  } 
\thanks{Y. Zheng was supported by the China Scholarship Council.}\thanks{Y.Hu was supported by an  NSERC discovery fund and a centennial  fund of University of Alberta.  
  }
\address{Department of Mathematical and  Statistical
 Sciences, University of Alberta at Edmonton,
Edmonton, Canada, T6G 2G1}
  \email{ yaozhong@ualberta.ca
}
\date{} 
\maketitle

\begin{abstract} In this paper we study the stochastic control problem of partially observed 
(multi-dimensional) stochastic system driven by both Brownian motions and fractional Brownian motions.
In the absence of the powerful tool of Girsanov transformation, we introduce and study 
  new stochastic processes which are used to transform the original problem to a ``classical one".
  The adjoint backward stochastic differential equations and the necessary condition 
  satisfied by the optimal control (maximum principle) are obtained. 
\end{abstract}

\maketitle

\section{Introduction and problem formulation}
Let $(\Om, \cF, \PP)$ be a complete probability space, where the expectation is denoted by $\EE$. On this probability space, we are given  two independent Brownian motions 
$W=(W_t=(W^1_t,W^2_t,\cdots,W^{k_1}_t),t\in[0,T])$ and $\tilde W=(\tilde{W}_t=(\tilde{W}^1_t,\tilde{W}^2_t,\cdots,\tilde{W}^{k_2}_t),t\in[0,T])$ on $\Om_1$
and two independent fractional Brownian motions $B =(B_t=(B^{ 1}_t,B^{ 2}_t,\cdots,B^{ m_1}_t),t\in[0,T])$ and $\tilde B=(\tilde{B}_t=(\tilde{B}^{ 1}_t,\tilde{B}^{ 2}_t,\cdots,\tilde{B}^{ m_2}_t),t\in[0,T])$ on $\Om_2$,  which are independent of the two Brownian motions,   where  $T>0$ is a fixed time horizon and $\Om=\Om_1\times\Om_2$.
For notational simplicity, we assume that the Hurst parameters   for all these fractional Brownian motions are the same, denoted  by  $H\in (1/3, 1)$ (we always assume $H\not=1/2$ since the Brownian motions are  included in the other part).  Since we are fixing $H$ throughout the paper
we omit the explicit notational dependence of $B$ and $\tilde B$ on $H$ to simplify  the notation.   
%
%

We denote  the filtrations $\mathbb{F}^W=\{\mathcal{F}_t^W=\sigma\{W_s,0\leq s\leq t\}\vee\mathcal{N},t\in[0,T]\}$, $\mathbb{F}^{\tilde{W}}=\{\mathcal{F}_t^{\tilde{W}}=\sigma\{\tilde{W}_s,0\leq s\leq t\}\vee\mathcal{N},t\in[0,T]\}$, $\mathbb{F}^B=\{\mathcal{F}_t^B=\sigma\{B_s,0\leq s\leq t\}\vee\mathcal{N},t\in[0,T]\}$, $\mathbb{F}^{\tilde{B}}=\{\mathcal{F}_t^{\tilde{B}}=\sigma\{\tilde{B}_s,0\leq s\leq t\}\vee\mathcal{N},t\in[0,T]\}$, and $\mathbb{F}=\{\mathcal{F}_t=\mathcal{F}^W_t\vee\mathcal{F}^{\tilde{W}}_t\vee\mathcal{F}^B_t\vee\mathcal{F}^{\tilde{B}}_t\vee\mathcal{N},t\in[0,T]\}$. Here, $\mathcal{N}$ denotes the set of all $ \mathbb{P}  $-null sets. 

The state equation we consider  is given  
by the following stochastic differential equation (SDE for short): 
\begin{equation}\label{x}
\left\{
\begin{aligned}
dX^u_t&=\sum_{j =1}^{m_1}A_{j }(t)X^u_tdB^{ j }_t+
\sum_{j =1}^{k_1} \sigma_j(t,X^u_t,u_t)dW_t^j+b(t,X^u_t,u_t)dt\,, \\
X^u_0&=x\in\mathbb{R}^n\,,
\end{aligned}
\right.
\end{equation}
where the stochastic integral with respect to Brownian motion is the usual It\^o integral and the stochastic integral with respect to fractional Brownian motion (fBm for short) is
the Young integral (see e.g. \cite[Section 2]{Hu13}) when $H>1/2$  and is pathwise integral via rough path theory (see  Section 2) when $H<1/2$. 

Suppose that the state process $X$ in \eqref{x} cannot be directly observed. Instead, we can observe a  functional of this process corrupted by some other noises. More precisely, we assume that the observation process
$\xi_t$ is  governed  by the following SDE 
\begin{equation}\label{xi}
\left\{
\begin{aligned}
d\xi^u_t&=h(t,X^u_t,u_t)dt+\sum_{j=1}^{m_2}C_{j }(t)\xi^u_td\tilde{B}^{ j }_t+\sum_{j=1}^{k_2} 
D_j(t)d\tilde{W}_t^j\\
\xi^u_0&=0\,.
\end{aligned}
\right.
\end{equation}  
We want to study the optimal control problem associated with  the cost functional
\begin{equation}\label{e.1.3} 
J(u )= \mathbb{E} \bigg[\Phi(X^u_T)+\int_0^Tf(t,X^u_t,u_t)dt\bigg].
\end{equation}
where $u_t$ is a $\mathcal{F}_t^\xi=\si\{\xi_s^u, 0\le s\le t\}$-adapted control process taking values in a nonempty set $U\subseteq\RR^d$, $t\in[0,T]$ and the coefficients in the above  equations 
are  now explained below by  the following assumptions. 
 
 \begin{enumerate}
 \item[{\bf (H1)}] $A_1, \cdots, A_{m_1} :[0,T]\rightarrow\mathbb{R}^{n\times n};     $   $  C_1, \cdots, C_{m_2}:[0,T]\rightarrow\mathbb{R}^{k_2\times k_2 }$, and $ D_1, \cdots, 
 D_{k_2}:[0,T]\\\rightarrow\mathbb{R}^{k_2 } $  are given deterministic functions. $A_i (t), C_{j}(t), i=1, \cdots, m_1, j=1, \cdots, m_2$ 
are smooth with   bounded derivatives.  We assume also that $A$'s  and $C$'s are 
nilpotent. This means that there is a  positive integer  $N_0$ such that 
for all $n> N_0$
\begin{equation}
\begin{aligned}
&[\dots [A_{i_1}(s_1),A_{i_2}(s_2)]\dots],A_{i_n}(s_n)]=0\,
\quad \forall \  1\le i_1, \cdots, i_n\le m_1,\\
&[\dots [C_{i_1}(s_1),C_{i_2}(s_2)]\dots],C_{i_n}(s_n)]=0\,
\quad \forall \  1\le i_1, \cdots, i_n\le m_2\,, 
\end{aligned}\label{e.3.9} 
\end{equation}
where the symbol $[A, B]=AB-BA$ denotes the commutator of 
the matrices  $A$ and $B$.   Moreover, we assume that  there exists a constant $C$ such that $|D(t)|+|D^{-1}(t)|\leq C$.

\item[{\bf (H2)}]  $  \sigma_1,\sigma_2 ,\cdots,\sigma_{k_1} :[0,T]\times\mathbb{R}^n\times\mathbb{R}^d\rightarrow\mathbb{R}^n,$   $b :[0,T]\times\mathbb{R}^n\times\mathbb{R}^d\rightarrow\mathbb{R}^n$, 
 $\xi^u \in\mathbb{R}^{k_2}, h :[0,T]\times\mathbb{R}^n\times\mathbb{R}^d\rightarrow\mathbb{R}^{k_2} $ ,  $f  :[0,T]\times\mathbb{R}^n\times{\RR}^d\rightarrow\mathbb{R},$ and $\Phi  :\mathbb{R}^n\rightarrow\mathbb{R}$ are given deterministic functions. $b,\sigma,f,\Phi$ and $h$ are twice continuously differentiable with respect to $x$ and  all the  derivatives $ b_X,b_{XX},\sigma_X,\sigma_{XX},f_{X},f_{XX},\Phi_X,\Phi_{XX},h_X,\\ h_{XX}$ are continuous  in $(x,u)$. 
The functions $b, \sigma,f_X,\Phi_X$ are bounded by $C(1+|x|+|u|)$ and $b_X,  b_{XX},\sigma_X,\sigma_{XX},h,h_X, h_{XX},f_{XX}$ and $\Phi_{XX}$ are bounded.

\item[{\bf (H3)}] 
$b,\sigma$  are  bounded   and uniformly Lipschitz continuous in $x$. 
For every $\om_2\in\Om_2$, $\mathbf{B}(\om_2)\in D^\alpha([0,T];\mathbb{R}^{m_1})$ (e.g. Definition \ref{def.6.3} below )  for some  $\al \in (1/3, H)  $. $A_j(\cdot),A^\prime_j(\cdot)$ are deterministic functions with $A^\prime_j$ being the Gubinelli derivative \cite{Gubinelli04} of $A_j$ and $(A_j,A_j^\prime)\in\mathscr{D}_X^{\al,\al}([0,T];\RR^{n\times n})$ (a reduced space defined in \cite[Definition 3.4]{FHL22arxiv}) for $j=1,\dots,m_1$. 
 \end{enumerate}

Notice  that the diffusion terms in both the state and observation equations  \eqref{x} and \eqref{xi} include two parts, one is driven by  Brownian motions $W,\tilde{W}$ and the other  one is driven by  fBms $B,\tilde{B}$. 
 However, the ones driven by fBms are assumed to be linear with respect to the unknowns. 

%
%
%
%
%
The stochastic control problem for completely or partially observed system driven by standard Brownian motions have been studied since long and the theory is rather complete. As for the partially observed optimal control problem involving only Brownian motions, we refer to \cite{LT95},\cite{Tang98},\cite{WWX13} and the references therein. There exist also some results on the stochastic systems driven only by 
fBms of $H>1/2$ and we refer to \cite{HHS13} and \cite{Sun21} and the references therein for further discussion of this topic. 
 
To obtain the maximum principle for the optimal control $\bar u$ for our problem \eqref{x}-\eqref{e.1.3}, one has to perturb     $\bar u$ by its spike variation $u^\vare$
(so that $u^\vare\to  \bar u$ as $\vare\to 0$),
and analyse $J(u^\vare)-J(\bar u)\ge 0$ for all small $\vare$ to obtain 
  necessary  conditions (namely the maximum principle) that the optimal
control $\bar u$ must satisfy. To analyse $J(u^\vare)-J(\bar u)\ge 0$
one needs to analyse $X^{u^\vare}-X^{\bar u}$ for the controlled state 
$X^u$.  
When the state system is driven only by fBm of Hurst parameter $H>1/2$,
this task is done by using the well-developed theory on Young integral and is already difficult  (see \cite[Subsection 5.2]{HHS13}).
It will certainly much more difficult when the systems are driven by both Brownian motion and fBm
and  it is even more   difficult to  handle directly 
in particular when the Hurst parameter of the fBm
is less than $1/2$.  

To get around this difficulty,
to the best of knowledge, we only know   the work of  Buckdahn and Jing \cite{BJ14}, where the Hurst parameter of the fBm 
can be less than $1/2$.   However, there are three  critical assumptions  in the mentioned  work. The first one is that the system is completely observable;
the second one is that the state system contains only one equation;  the third one is that the diffusion coefficient of the fBm in the state equation is linear. 
In particular, these three conditions enables the authors 
to use the  Girsanov transformations   to transform the original system into another ``classical'' one driven only by standard Brownian motion  (but depending on the fBm implicitly through the coefficients),
and then  obtained the maximum principle for the optimal control problem. 

In this work, in addition to allowing the   system to be partially observable, the controlled state  system can also be allowed to contain  several equations. In our new context,  the 
Girsanov transformation method powerful in   single equation case (\cite{BJ14}) is no longer applicable.  Our idea is to introduce  another transformation to transform the problem to a ``classical'' one. 
 To this end, we first give the dyadic piecewise linear approximation of the fBms $B^\vare,\tilde{B}^\vare$ which then gives (random) ordinary differential equations (ODEs for short) driven by the approximating process. Then the 
  systems of approximated state and observation equations  
are transformed to the ``classical''   ones driven by Brownian motions alone,  but with coefficients containing the fBms.  After a usual limiting argument, we transform the original problem to a ``classical'' one. 
Finally, we can apply  the traditional 
technique to obtain maximum principle for our problem  with a slight adoption of random coefficients.

Here is the organization of the paper. In Section 2 we give a preliminary about 
stochastic integral via rough path theory. In Section 3, we approximate the fBms 
by the dyadic approximations and we introduce two processes satisfying two linear equations  which 
will be used to transform our systems   of   state and observation equations 
to classical ones.  Sections 4 and 5 study the two processes when $H>1/2$ and $1/3<H<1/2$, respectively.  
In Section 6, we use the processes obtained in the previous sections to transform 
our original optimal control problem to a classical one.   In Section 7, we obtain the adjoint 
backward stochastic differential equations and
the maximum principle for transformed optimal control problem.  To make the paper more readable, we postpone some complicated  computations  to appendix.

\section{Preliminary} 

\subsection{Some spaces for integrable two-parameter processes}
Let $X=(X^i)$ be a H\"older continuous path in $\mathbb{R}^d$ with exponent  $\al\in(1/3,1)$. Throughout the sequel, we will reserve the constant $N$ as the largest integer such that $N\al\leq1$. Define the simplexes $\Delta(a,b)$ and $\Delta_2(a,b)$ as follows:
\begin{equation}
\begin{aligned}
\Delta(a,b):&=\{(s,t)\in[0,T]^2,\ a\leq s\leq t\leq b\},\\
\Delta_2(a,b):&=\{(s,u,t)\in[0,T]^3,\ a\leq s\leq u\leq t\leq b\},
\end{aligned}
\end{equation}
and we sometimes write $\Delta=\Delta(0,T)$ and $\Delta_2=\Delta_2(0,T)$. Moreover, given a path $X=(X_t):[0,T]\rightarrow\RR^d$ in the rough path's framework, we denote by two-parameter path $(\delta X_{s,t})_{(s,t)\in\Delta}$ the increment of $X$, where $\delta X_{s,t}:=X_t-X_s$ for every $(s,t)\in\Delta$. Additionally, if $A:\Delta\rightarrow\RR^d$ is a two-parameter path, then the increment of $A$ is a map $\delta A:\Delta_2\rightarrow\RR^d$ given by $\delta A_{s,u,t}:=A_{s,t}-A_{s,u}-A_{u,t}$, for every $(s,u,t)\in\Delta_2$.

For any $\alpha\in(0,1)$, we denote by $\cC^{\alpha}([0,T];\RR^d)$ the space of $\RR^d$-valued $\alpha$-H\"older  continuous functions on the interval $[0,T]$ and make use of the notation
\begin{equation}
\Vert X\Vert_{\alpha;[0,T]}=\sup_{0<r<\theta<T}\frac{|X_r-X_{\theta}|}{|r-\theta|^{\alpha}}.
\end{equation}
We need some  spaces and notations introduced in  \cite[Section 2]{FHL22arxiv}. For a sub-$\sigma$-field $\mathcal{G}\subset\mathcal{F}$, define the conditional $L_m$-norm as $\Vert\xi|\mathcal{G}\Vert_m=[\EE(|\xi|^m|\mathcal{G})]^{\frac{1}{m}}$.
The resulting mixed $L_{m,n}$-norm $\Vert\Vert\xi|\mathcal{G}\Vert_m\Vert_n$, with $m,n\geq1$, reduces to the classical $L_m$-norm when $m=n$.

Set $\bar{\Omega}:=(\Omega,\mathcal{F},\PP)$. Given a separable Banach space $\mathcal{X}$ with norm $|\cdot|_{\mathcal{X}}$, the Lebesgue space $L^m(\bar{\Omega};\mathcal{X})$ of $\mathcal{X}$-valued $L_m$-integrable random variables is denoted by $L_m(\mathcal{X})$. Its norm is given by $\xi\mapsto\Vert\xi\Vert_{m;\mathcal{X}}=\Vert|\xi|_{\mathcal{X}}\Vert_m$ or simply $\Vert\xi\Vert_m$ if $\mathcal{X}$ is clear from the context. We will use the abbreviation $\EE_s=\EE(\cdot|\mathcal{G}_s)$ for all $s\in[0,T]$,  where $\mathcal{G}_s=\mathcal{F}_s^W\vee 
\mathcal{F}_s^ {\tilde W}$.

For each  $m\in[1,\infty]$, $C_2L_m([0,T],\bar{\Omega};\mathcal{X})$ denotes the space of $\mathcal{X}$-valued, two-parameter stochastic processes $(s,t)\mapsto A_{s,t}$ such that
\begin{enumerate}
\item[(i)] $A:\Omega\times\Delta\rightarrow\mathcal{X}$ is $\mathcal{F}\otimes\mathcal{B}(\Delta)/\mathcal{B}(\mathcal{X})$-measurable.
\item[(ii)]  $A:\Delta\rightarrow L^m(\Omega;\mathcal{X})$ is continuous. 
\end{enumerate}

Equipped with the norm $\Vert A\Vert_{C_2L_m\mathcal{X}}:=\sup_{(s,t)\in\Delta}\Vert A_{s,t}\Vert_{m,\mathcal{X}}$, the space $C_2L_m([0,T],\bar{\Omega};\\ \mathcal{X})$ is a Banach space. For notational simplicity, we will abbreviate this space by $C_2L_m$. 
For each $m,n\in[1,\infty],m\leq n$, let $C_2L_{m,n}([0,T],\bar{\Omega},\{\mathcal{G}_t\};\mathcal{X})$ be the space of two-parameter processes $A$ such that
\begin{enumerate}
\item[(i)] $A$ belongs to $C_2L_m([0,T],\bar{\Omega};\mathcal{X})$,
\item[(ii)] $\Vert A\Vert_{C_2L_{m,n}\mathcal{X}}:=\sup_{(s,t)\in\Delta}\Vert\Vert A_{s,t}|\mathcal{G}_s\Vert_{m;\mathcal{X}}\Vert_n<\infty$. 
\end{enumerate} 

The linear space $C_2^\alpha L_{m,n}([0,T],\bar{\Omega},\{\mathcal{G}_t\};\mathcal{X})$ consists of two-parameter processes $(A_{s,t})_{(s,t)\in\Delta}$ in $C_2L_{m,n}$ such that
\begin{equation}
\Vert A\Vert_{\alpha;m,n;[0,T]}:=\sup_{s,t\in[0,T]:s<t}\frac{\Vert\Vert A_{s,t}|\mathcal{G}_s\Vert_m\Vert_n}{|t-s|^\alpha}<\infty.
\end{equation}
Similarly, denote by $C_2^\alpha L_n([0,T],\bar{\Omega},\{\mathcal{G}_t\};\mathcal{X})$ the two-parameter processes $(A_{s,t})_{(s,t)\in\Delta}$ in $C_2L_n$ such that
\begin{equation}
\Vert A\Vert_{\alpha;n;[0,T]}:=\sup_{s,t\in[0,T]:s<t}\frac{\Vert A_{s,t}\Vert_n}{|t-s|^\alpha}<\infty.
\end{equation}

\subsection{Young's type stochastic integral}
In this subsection, we construct the stochastic integral  of $Z$ against $\RR^d$-valued $\al$-H\"older continuous path $X$ ($\alpha\in (1/2, 1)$). First, we introduce the following definition, which can be regarded as the reduced version of \cite{FHL22arxiv}.
\begin{definition}\label{Young controlled}
We say that $Z$ is stochastic controlled rough path of $m$-integrability and $\beta$-H\"older regularity with values in $\RR^n$ with respect to $\{\cG_t\}$ if the followings are  satisfied
\begin{enumerate}
\item[(i)] $Z:\Omega\times[0,T]\rightarrow\RR^n$ is $\{\cG_t\}$-adapted;
\item[(ii)] $\delta Z$ belongs to $C_2^\beta L_m([0,T],\bar{\Omega},\{\cG_t\};\RR^n)$. 
\end{enumerate} 
The class of such processes is denoted by $\cL_X^\beta L_m([0,T],\bar{\Omega},\{\cG_t\};\RR^n)$.
\end{definition}
In order to establish the Young's stochastic integral, we need the following  sewing lemma (a reduced version of \cite[Theorem 2.7]{FHL22arxiv},  \cite{le} in the case of $\al\in(1/2,1)$). 
\begin{lemma}\label{sewing lemma}
Let $A:\Omega\times\Delta\rightarrow V$ be a measurable adapted $L^m$-integrable two-parameter process. Define $\delta A_{s,u,t}:=A_{s,t}-A_{s,u}-A_{u,t}$. If there is finite constants $C$ such that
\begin{equation}\label{sewing}
\Vert\delta A_{s,u,t}\Vert_m\leq C|t-u|^\alpha|u-s|^\beta,\ \forall\ (s,u,t)\in\Delta_2,
\end{equation}
for some $\alpha+\beta>1$. Then there exists a unique measurable adapted $L^m$-integrable process $\cA:\Omega\times[0,T]\rightarrow V$ with $\cA_0=0$ such that we have the following local estimate
\begin{equation}
\Vert\delta \cA_{s,t}-A_{s,t}\Vert_m\leq C|t-u|^\alpha|u-s|^\beta\leq o(|t-s|),\ \forall\ (s,t)\in\Delta.
\end{equation}
Moreover, we have
\begin{equation}
\lim_{\pi\in\cP([0,T]),|\pi|\rightarrow0}\sup_{t\in[0,T]}\Big\Vert\cA_t-\sum_{[u,v]\in\pi,u\leq t}A_{u,v\wedge t}\Big\Vert_m=0.
\end{equation}
\end{lemma}

\begin{theorem}\label{p12}
Let $X\in\cC^\alpha([0,T];\RR^d)$, $Z\in\cL_X^\beta L_m([0,T],\bar{\Omega},\{\cG_t\};\RR^n)$ with $\alpha+\beta>1$. Then $A_{s,t}:=Z_s\delta X_{s,t}$ define a two-parameter stochastic process which satisfies the hypotheses of the sewing lemma. We define   $\int_0^{\cdot}ZdX$  to be   the limiting (continuous) process     of $\sum_{[u,v]\in\cP:u\leq t}Z_u\delta X_{u,v\wedge t}$, i.e.,
\begin{equation}\label{e.2.4}
\lim_{\pi\in\cP([0,T]),|\pi|\rightarrow0}\sup_{t\in[0,T]}\Big\Vert\int_0^tZ_sdX-\sum_{[u,v]\in\pi,u\leq t}Z_u\delta X_{u,v\wedge t}\Big\Vert_m=0 
\end{equation}
 and for every $(s,t)\in\Delta$, we have the following bounds
\begin{equation}\label{bounds}
\bigg\Vert\int_s^tZdX-Z_s\delta X_{s,t}\bigg\Vert_m\leq C|t-s|^{\alpha+\beta}.
\end{equation}
We call $\int_0^{\cdot}ZdX$ the stochastic integral of Z against X.
\end{theorem}
\begin{proof}
First, for two-parameter process $A_{s,t}=Z_s\delta X_{s,t}$, we see  that $\delta A_{s,u,t}$ satisfies the conditions in Lemma \ref{sewing lemma}. Indeed, by condition (i) and (ii) in Definition \ref{Young controlled}, we have, for $s<u<t$,
\begin{equation}
\begin{aligned}
\Vert\delta A_{s,u,t}\Vert_m&=\Vert A_{s,t}-A_{s,u}-A_{u,t}\Vert_m\\
&=\Vert (Z_s-Z_u)(X_t-X_u)\Vert_m
\leq(\EE|Z_s-Z_u|^m)^{\frac{1}{m}}|X_t-X_u|\\
&\leq\sup_{s,u}\frac{\Vert Z_s-Z_u\Vert_m}{|s-u|^\beta}|s-u|^\beta|X_t-X_u|\leq C|s-u|^\beta|t-u|^\alpha.
\end{aligned}
\end{equation}
By Lemma \ref{sewing lemma},  we have the convergence \eqref{e.2.4} and local estimate \eqref{bounds}.
\end{proof}

\subsection{Rough stochastic integrals in $\al\in(1/3,1/2)$ H\"older scale}

In this subsection, we give the definition of stochastic integral against the fBm ($H\in(1/3,1/2)$) which is directly from \cite{FHL22arxiv}.

%
%
%
%

\begin{theorem}\label{rough int}{\rm\cite[Theorem 3.7]{FHL22arxiv}}
Let $\al\in(1/3,1/2), \beta,\beta^\prime\in[0,1], \al+\beta>1/2,\al+(\al\wedge\beta)+\beta^\prime>1,m\in[2,\infty),n\in[m,\infty]$ and $\mathbf{X}=(X,\mathbb{X})\in\cC^\alpha([0,T];V)$. Suppose that $(Z,Z^\prime)$ is an extended stochastic controlled rough path in $\bar{D}_X^{\beta,\beta^\prime}L_{m,n}([0,T],\bar{\Omega},\{\mathcal{G}_t\};W)$ (defined in \cite[Definition 3.1]{FHL22arxiv}) with $Z^\prime$ being the Gubinelli derivative of $Z$. Then $\Xi_{s,t}=Z_s\delta X_{s,t}+Z^\prime_s\mathbb{X}_{s,t}$ defines a two-parameter stochastic process which satisfies the hypotheses of the stochastic sewing lemma  \rm{\cite[Theorem 2.7]{FHL22arxiv}}. Then   
\begin{equation}\label{int limit}
 \lim_{\mathcal{P}\rightarrow0}\sum_{(s,t)\in\mathcal{P}}\Xi_{s,t}
\end{equation}
exists in probability, and is denoted by $\int_0^TZd\mathbf{X} =\mathcal{I}(\Xi)_{0,T}$.  Moreover,   for every $(s,t)\in\Delta$, we have the following bounds
\begin{equation}\label{bound1}
\begin{aligned}
&\bigg\Vert\bigg\Vert\int_s^tZ d\mathbf{X}-\Xi_{s,t}\bigg|\mathcal{G}_s\bigg\Vert_m\bigg\Vert_n\lesssim\Big(|\delta X|_\al\Vert\delta Z\Vert_{\beta;m,n}+|||\mathbb{X}|||^2_\al\sup_{r\in[s,t]}\Vert Z^\prime_r\Vert_n\Big)|t-s|^{\al+(\al\wedge\beta)}\\
&\qquad+\Big(|\delta X|_\al\Vert\EE_{\cdot}R^Z\Vert_{\beta+\beta^\prime;n}+|\mathbb{X}|_{2\al}\Vert\EE_{\cdot}\delta Z^\prime\Vert_{\beta^\prime;m,n}\Big)|t-s|^{\al+(\al\wedge\beta)+\beta^\prime}
\end{aligned}
\end{equation}
and
\begin{equation}\label{bound2}
\begin{aligned}
&\bigg\Vert\mathbb{E}_s\bigg(\int_s^tZ d\mathbf{X}-\Xi_{s,t}\bigg)\bigg\Vert_n\lesssim \Big(|\delta X|_\al\Vert\EE_{\cdot}R^Z\Vert_{\beta+\beta^\prime;n}+|\mathbb{X}|_{2\al}\Vert\EE_{\cdot}\delta Z^\prime\Vert_{\beta^\prime;m,n}\Big)|t-s|^{\al+(\al\wedge\beta)+\beta^\prime},
\end{aligned}
\end{equation}
where the above constants are deterministic and depend only on $\al,\beta,\beta^\prime,m,T$.
\end{theorem}

The stochastic integrals $\sum_{j =1}^{m_1}\int_0^tA_{j }(s)X^u_sd \mathbf{B} ^j_s$ and $\sum_{j=1}^{m_2}\int_0^tC_{j }(s)\xi^u_sd \tilde{\mathbf{B}} ^j_s$  are  well-defined pathwisely  by the above theorem (e.g. Equation \eqref{int limit}). To alleviate notation, 
we still use $\sum_{j =1}^{m_1}\int_0^tA_{j }(s)X^u_sd  B ^j_s$ and $\sum_{j=1}^{m_2}\int_0^tC_{j }(s)\xi^u_sd \tilde{B} ^j_s$ to denote them.  
%

\section{Transformation of the equation} 
%
We divide the interval $[0, T]$ into dyadic subintervals:  $0=t_0<t_1<\cdots<t_{2^k} =T$, where $t_\ell =t_\ell ^k =\frac{\ell T}{2^k}$, for $\ell=0, 1,\dots,2^k$. On the subinterval $[t_{\ell -1}, t_\ell ]$, we approximate  the fBms by  
\begin{equation}\label{e.3.a1}
\left\{
\begin{aligned}
&B_t^\vare =B_{t _{\ell -1}}+(t-t _{\ell -1})\frac{\Delta_\ell B}{\vare }
,\ t _{\ell -1}\leq t\leq t_\ell ,\\
&\tilde{B}_t^\vare =\tilde{B}_{t _{\ell-1}}+(t-t _{l-1})\frac{\Delta _\ell \tilde   B}{\vare }  ,\ t _{\ell-1}\leq t\leq t_\ell ,
 \end{aligned}
\right.
\end{equation}
where $\vare=\frac{  T}{2^k}$ and $\Delta _\ell B=B_{t _{\ell}}-B_{t _{\ell-1}}, \Delta_\ell\tilde{B}=\tilde{B}_{t _{\ell}}-\tilde{B}_{t _{\ell-1}}$.
We write the components of $B_t^\vare$  and $\tilde{B}_t^\vare$ as 
 ${B}_t^\vare =( {B}_t^{1,\vare}, \cdots, {B}_t^{m_1,\vare})^T$ and
$\tilde{B}_t^\vare =( \tilde {B}_t^{1,\vare}, \cdots, \tilde{B}_t^{m_2,\vare})^T\,. $

We approximate the state and observation equations as follows:
\begin{equation}
\left\{
\begin{aligned}
& dX^{u,\vare}_t =b(t,X^{u,\vare}_t,u_t) dt +\sum_{j=1}^{k_1}\sigma_j(t,X^{u,\vare}_t,u_t) dW^j_t +\sum^{m_1}_{j=1}A_{j}(t)X^{u,\vare}_t\dot{B}^{j,\vare}_t dt,\\
& d\xi^{u,\vare}_t =h(t,X^{u,\vare}_t,u_t) dt +\sum_{j=1}^{k_2}D_j(t) d\tilde{W}^j_t  +\sum^{m_2}_{j=1}C_{j}(t)\xi^{u,\vare}_t\dot{\tilde{B}}^{j,\vare}_t dt,\\
&X^{u,\vare}_0=x,\quad  \xi^{u,\vare}_0=0,
\end{aligned}
\right.\label{e.2.2} 
\end{equation}
where by \eqref{e.3.a1} $\dot{B}^{j,\vare}_t:=\frac{dB^{j, \vare}_t}{dt}$, $\dot{\tilde{B}}^{j,\vare}_t:=\frac{d\tilde{B}^{j, \vare}_t}{dt}$ exists
except at the division points of the partition, where we can define them as the right derivatives. The different choices of the  values of $\dot{B}^{j,\vare}_t:=\frac{dB^{j, \vare}_t}{dt}$, $\dot{\tilde{B}}^{j,\vare}_t:=\frac{d\tilde{B}^{j, \vare}_t}{dt}$  
at the division points $t_\ell$ will not affect the limits of $X^{u,\vare}$ and $\xi^{u,\vare}$.
To solve the above equation, we write the above equation as   
\begin{equation}\label{X_ap}
\left\{
\begin{aligned}
&dX^{u,\vare}_t-\sum^{m_1}_{j=1}A_{j}(t)\dot{B}^{j,\vare}_tX^{u,\vare}_tdt=b(t,X^{u,\vare}_t,u_t)dt+\sum_{j=1}^{k_1}\sigma_j(t,X^{u,\vare}_t,u_t)dW^j_t,\\
&d\xi^{u,\vare}_t-\sum^{m_2}_{j=1}C_{j}(t)\dot{\tilde{B}}^{j,\vare}_t\xi^{u,\vare}_tdt=h(t,X^{u,\vare}_t,u_t)dt+\sum_{j=1}^{k_2}D_j(t)d\tilde{W}^j_t\,. 
\end{aligned}
\right.
\end{equation}
To get rid of the terms $\dot{\tilde{B}}^{j,\vare}_t$ and  $ \dot{B}^{j,\vare}_t$
appeared in the above two systems, let us consider the ODEs,
which will be solved later  
\begin{equation}\label{approximation ode}
\left\{
\begin{aligned}
&d\Gamma^{\vare}_t=-\sum^{m_1}_{j=1}\Gamma^{\vare}_tA_{j}(t)dB^{j,\vare}_t,\\
&d\Lambda^{\vare}_t=-\sum^{m_2}_{j=1}\Lambda^{\vare}_tC_{j}(t)d\tilde{B}^{j,\vare}_t,
\end{aligned}
\right.
\end{equation}
where $\Gamma^{\vare}_t\in\mathbb{R}^{n\times n}$ and $\Lambda^{\vare}_t\in\mathbb{R}^{k_2\times k_2}$ and $dB^{j, \vare}_t=\dot{B}^{j,\vare}_tdt$, $d\tilde{B}^{j, \vare}_t=\dot{\tilde{B}}^{j,\vare}_tdt$.
Using   the product rules
\begin{equation}
\left\{
\begin{aligned}
&d\big[\Gamma^{\vare}_t X^{u,\vare}_t\big]=\Gamma^{\vare}_t dX^{u,\vare}_t+\Big(d\Gamma^{\vare}_t \Big)X^{u,\vare}_t,\\
&d\big[\Lambda^{\vare}_t \xi^{u,\vare}_t\big]=\Lambda^{\vare}_t d\xi^{u,\vare}_t+\Big(d\Lambda^{\vare}_t\Big)\xi^{u,\vare}_t,
\end{aligned}
\right.
\end{equation}
we can write \eqref{X_ap} as 
\begin{equation}\label{equality_1}
\left\{
\begin{aligned}
 \left(\Gamma^{\vare}_t\right)^{-1}d\big[\Gamma^{\vare}_tX^{u,\vare}_t\big]=&dX^{u,\vare}_t+\left(\Gamma^{\vare}_t\right)^{-1}\Big(d\Gamma^{\vare}_t\Big)X^{u,\vare}_t\\
 =&b(t,X^{u,\vare}_t,u_t) dt +\sum_{j=1}^{k_1}\sigma_j(t,X^{u,\vare}_t,u_t) dW^j_t 
 \\
 \left(\Lambda^{\vare}_t\right)^{-1}d\big[\Lambda^{\vare}_t\xi^{u,\vare}_t\big]=
 &d\xi^{u,\vare}_t+\left(\Lambda^{\vare}_t\right)^{-1}\Big(d\Lambda^{\vare}_t\Big)\xi^{u,\vare}_t \\
 =& h(t,X^{u,\vare}_t,u_t)dt+\sum_{j=1}^{k_2}D_j(t)d\tilde{W}^j_t\,.  
\end{aligned}
\right.
\end{equation}
Denoting  $Y_t^{u,\vare}=\Gamma^{\vare}_tX^{u,\vare}_t$ and $\zeta^{u,\vare}_t=\Lambda^{\vare}_t\xi_t^{u,\vare}$,  we see from   \eqref{equality_1}  that 
$(Y_t^{u,\vare}, \zeta^{u,\vare}_t)$ satisfy
\begin{equation}\label{e.2.7} 
\left\{
\begin{aligned}
&dY^{u,\vare}_t=\Gamma^{\vare}_tb(t,\left(\Gamma^{\vare}_t\right)^{-1}Y^{u,\vare}_t,u_t)dt+\Gamma^{\vare}_t\sum_{j=1}^{k_1}\sigma_j(t,\left(\Gamma^{\vare}_t\right)^{-1}Y^{u,\vare}_t,u_t)dW^j_t,\\
&d\zeta^{u,\vare}_t=\Lambda^{\vare}_th(t,\left(\Gamma^{\vare}_t\right)^{-1}Y^{u,\vare}_t,u_t)dt+\Lambda^{\vare}_t\sum_{j=1}^{k_2}D_j(t)d\tilde{W}^j_t\,. 
\end{aligned}
\right.
\end{equation}
Thus, with the transformations $X \to Y$ and $\xi \to \zeta$, we transform
the system  \eqref{e.2.2} to the above  system \eqref{e.2.7}. 
Notice that both equations in \eqref{e.2.7} do not    
contain the diffusion part driven by 
the fBms $B,\tilde{B}$ and they are classical stochastic differential equations
driven by standard Brownian motion (with $(\om_2)$-random coefficients).    

\begin{remark}
The linearity assumption on the diffusion coefficients of the fBm is used in 
\eqref{approximation ode}  and \eqref{equality_1} which is critical to transform 
\eqref{e.2.2} to  \eqref{e.2.7}, the 
equations without driving fBm.  If the term is more general, such as
$A(X_t) dB_t$, for example, if we consider $dX_t=b(t, X_t) dt +\sum_{j=1}^{k_1}\si_j (t, X_t) dW_t^j+\sum_{j=1}^{m_1} A_j(X_t) dB_t^j$, we may still consider using the transformation 
$d\Ga_t=-\sum_{j=1}^{m_1} \tilde A_j(\Ga_t) dB_t^j$. However, in this case \eqref{equality_1}
becomes 
\begin{equation}  
\begin{aligned}
 \left(\Gamma _t\right)^{-1}d\big[\Gamma _tX _t\big]=&dX _t-\sum_{j=1}^{m_1} \left(\Gamma _t\right)^{-1}\tilde A_j(\Ga_t)X_t dB_t^j \\
 =&b(t,X _t ) dt +\sum_{j=1}^{k_1}\sigma_j(t,X _t ) dW^j_t+\sum_{j=1}^{m_1} \left[A_j(X_t)-\left(\Gamma _t\right)^{-1}\tilde A_j(\Ga_t)X_t\right] dB_t^j \,.   
\end{aligned} 
\end{equation}
The only case we are able to make the coefficient $dB_t^j$ disappear 
is when  $A_j(X_t)=\left(\Gamma _t\right)^{-1}\tilde A_j(\Ga_t)X_t$ or $A_j(X_t)=A_j(t) X_t$ for some $  A_j ( t ) $, and in this case we take $\tilde A_j(\Ga_t)=\Ga_tA_j(t)$. 
However, there may be some other way towards a solution and the problem is very much worthy further investigation. 
\end{remark}

To make the above transformation legitimate,  we need  to solve \eqref{approximation ode}. Now, we give the representation of solution in the nilpotent case (see assumption {\bf (H1)}).

 we shall solve \eqref{approximation ode} by using  the  generalized Campbell-Baker-Hausdorff-Dynkin formula in \cite{Strichartz87} (more related literature, see \cite{BZ12},
 \cite{Huformal}, \cite{Kunita80}), that is,

\begin{equation}
\Gamma^{\vare}_t=\exp \left\{\mathcal{K}^{\vare}_t\right\}\,, 
  \qquad \Lambda^{\vare}_t=\exp \left\{\chi^{\vare}_t\right\}\,.
\end{equation}
Here 
\begin{equation}\label{w^epsilon}
\begin{aligned}
\mathcal{K}^{\vare}_t=&\sum^{\infty}_{n=1}\sum_{(i_1,\dots,i_n)}\sum_{\sigma\in\mathcal{S}_n}\frac{(-1)^{e(\sigma)+n}}{n^2\binom{n-1}{e(\sigma)}}\int_{0< s_1<\dots< s_n< t}[\dots [A_{i_1}(s_1),A_{i_2}(s_2)]\dots],A_{i_n}(s_n)]\\
&\times dB^{\sigma^{-1}(i_1),\vare}_{s_1}dB^{\sigma^{-1}(i_2),\vare}_{s_2}\dots dB^{\sigma^{-1}(i_n),\vare}_{s_n}
\end{aligned}
\end{equation}
and
\begin{equation}\label{w^epsilon2}
\begin{aligned}
\chi^{\vare}_t=&\sum^{\infty}_{n=1}\sum_{(i_1,\dots,i_n)}\sum_{\sigma\in\mathcal{S}_n}\frac{(-1)^{e(\sigma)+n}}{n^2\binom{n-1}{e(\sigma)}}\int_{0< t_1<\dots< t_n< t}[\dots [C_{i_1}(t_1),C_{i_2}(t_2)]\dots],C_{i_n}(t_n)]\\
&\times d\tilde{B}^{\sigma^{-1}(i_1),\vare}_{t_1}d\tilde{B}^{\sigma^{-1}(i_2),\vare}_{t_2}\dots d\tilde{B}^{\sigma^{-1}(i_n),\vare}_{t_n},
\end{aligned}
\end{equation}
where  $\mathcal{S}_n$ is the set of all  permutations of $\{0,\dots,n\}$ and if $\sigma\in\mathcal{S}_n$, $e(\sigma)$ is the cardinality of the set $\{j\in\{1,\dots, n-1\},\sigma(j)>\sigma(j+1)\}$.    Notice that the summation in \eqref{w^epsilon} and \eqref{w^epsilon2} 
are finite due to the nilpotent assumption.  

In the next two sections, we aim to obtain the convergences of solutions $(\Gamma^{\vare},\Lambda^{\vare})$  of  equation \eqref{approximation ode} and multiple integrals in \eqref{w^epsilon}, \eqref{w^epsilon2}, respectively. To this end, we divide our discussion into  the following two cases, $H\in(1/2,1)$ and $H\in(1/3,1/2)$, and  obtain the desired results with the  Young's and rough path's methods, respectively.

\section{The convergence in the Young's framework}
In this section, we assume   $1/2<H<1 $ and    show that the  solutions $\Gamma^{\vare},\Lambda^{\vare}$ to the ODEs  \eqref{approximation ode}
converge  to the solutions of the following  differential equations in the sense of Young:
\begin{equation}
\left\{
\begin{aligned}
&d\Gamma_t=-\sum^{m_1}_{j=1}\Gamma_tA_{j}(t)\circ dB^{j}_t,\\
&d\Lambda_t=-\sum^{m_2}_{j=1}\Lambda_tC_{j}(t)\circ d\tilde{B}^{j}_t\,.  
 \end{aligned}
\right. \label{e.3.1} 
\end{equation}

Now we give the Wong-Zakai type limit theorem in the sense of Young. 
\begin{proposition}\label{p.4.1}
Let $\alpha\in(1/2,1]$ and $f\in C_b^2$. Let $X,\tilde{X}\in\cC^\alpha$ and $y,\tilde{y}\in\RR^n$, and let $Y,\tilde{Y}$ be the unique solutions of 
\begin{equation}
Y_t=y+\int_0^tf(Y_s)dX_s,\ t\in[0,T]
\end{equation} 
with the data $(y,X)$ and $(\tilde{y},\tilde{X})$, respectively. Let $K>0$ be a constant such that $\Vert X\Vert_\alpha,\Vert \tilde{X}\Vert_\alpha\leq K$. Then, for any $\tilde{\alpha}\in(1/2,\alpha)$, there exists a constant $C_K>0$, depending on $\tilde{\alpha},T,\Vert f\Vert_{C_b^2}$ and $K$, such that
\begin{equation}
\Vert Y-\tilde{Y}\Vert_{\tilde{\alpha}}\leq C_K(|y-\tilde{y}|+\Vert X-\tilde{X}\Vert_{\tilde{\alpha}}).
\end{equation} 
\end{proposition} 
\begin{proof} When $\tilde{\alpha}$ on the left hand side is $0$, it is a consequence of 
\cite[Theorem 4]{HN}. For the theorem of present form, we can combine the statement of \cite[Theorem 4]{HN}
with the bound on $\Vert Y-\tilde{Y}\Vert_{\tilde{\alpha}}$ on page 409 in  the proof of that theorem. 
\end{proof}
\begin{theorem}\label{Young WZ}
Let $B=(B^{1},B^{2},\cdots,B^{m_1})\in\mathbb{R}^{m_1}$ and $\tilde{B}=(\tilde{B}^{1},\tilde{B}^{2},\cdots,\tilde{B}^{m_2})\in\mathbb{R}^{m_2}$ be two fBms with $H\in(1/2,1)$, 
and let $B^{\vare},\tilde{B}^{\vare}$ be the dyadic piecewise linear approximations of $B,\tilde{B}$, respectively. Then the solutions $\Gamma^{\vare}$ and $\Lambda^{\vare}$  to the equations 
\eqref{approximation ode}
converge to the Stratonovich solutions  of  \eqref{e.3.1}
\end{theorem}
\begin{proof}
First,  we prove that $B^{\vare}$ converges almost surely to $B$ in $\mathcal{C}^\alpha([0,T];\mathbb{R}^{m_1})$. Indeed, fix $0<s<t<T$. 
We have for any $\beta\in (0, 1)$ satisfying  
$\beta>\beta H>\al$,   
\begin{equation}
\frac{|B_t^{\vare}-B_t-(B_s^{\vare}-B_s)|}{|t-s|^\al}=\frac{|B_t^{\vare}-B_t-(B_s^{\vare}-B_s)|^\beta}{|t-s|^\al} |B_t^{\vare}-B_t-(B_s^{\vare}-B_s)|^{1-\beta}.\label{e.4.4aa} 
\end{equation}
(i) When  $t-s>\frac{T}{2^k}$, we see  $s\in[t_l,t_{l+1}]$ and $t\in[t_m,t_{m+1}]$
 ($l<m$)  and 
\begin{equation*}
\begin{aligned}
&\bigg(\frac{|B_t^{\vare}-B_t-(B_s^{\vare}-B_s)|}{|t-s|^{\al/\beta}}\bigg)^\beta\leq(\frac{2^{k}}{T})^\al[|B_t^{\vare}-B_t|+|B_s^{\vare}-B_s)|]^\beta\\
&\quad\leq(\frac{2^{k}}{T})^\al[|B_{t _{m}}+(t-t _{m})\frac{2^k}{T}(B_{t_{m+1}}-B_{t_{m}})-B_t|+|B_{t _{l}}+(s-t _{l})\frac{2^k}{T}(B_{t_{l+1}}-B_{t_{l}})-B_s|]^\beta\\
&\quad\leq (\frac{2^k}{T})^\al\bigg[\frac{|B_{t _{m}}-B_t|}{|t_m-t|^{\al/\beta}}+\frac{|B_{t _{m+1}}-B_{t_m}|}{|t_{m+1}-t_m|^{\al/\beta}}+\frac{|B_{t _{l}}-B_s|}{|t_l-s|^{\al/\beta}}+\frac{|B_{t _{l+1}}-B_{t_l}|}{|t_{l+1}-t_l|^{\al/\beta}}\bigg]^\beta(\frac{T}{2^k})^\al\leq C_\beta.
\end{aligned}
\end{equation*}
(ii) When  $|t-s|<\frac{T}{2^k}$,   we divide our discussion into  two cases. First, when $s,t\in[t_l,t_{l+1}]$ are in the same subinterval, we have
\begin{equation*}
\begin{aligned}
\bigg(\frac{|B_t^{\vare}-B_t-(B_s^{\vare}-B_s)|}{|t-s|^{\al/\beta}}\bigg)^\beta&\leq\bigg[\frac{|B_t-B_s)|}{|t-s|^{\al/\beta}}+\frac{|B_t^{\vare}-B_s^{\vare})|}{|t-s|^{\al/\beta}}\bigg]^\beta\\
&\leq \bigg[C+\frac{\frac{2^k}{T}|t-s||B_{t_{l+1}}-B_{t_l})|}{|t-s|^{\al/\beta}}\bigg]^\beta\\
&\leq \bigg[C+\frac{2^k}{T}|t-s|^{1-\al/\beta}\frac{|B_{t_{l+1}}-B_{t_l})|}{|t_{l+1}-t_l|^{\al/\beta}}(\frac{T}{2^k})^{\al/\beta}\bigg]^\beta\\
&\leq \bigg[C+C\frac{2^k}{T}(\frac{T}{2^k})^{1-\al/\beta}(\frac{T}{2^k})^{\al/\beta}\bigg]^\beta\leq C_\beta.
\end{aligned}
\end{equation*}
Next,  when $s,t$ are in  two adjacent intervals, i.e. $s\in[t_{l-1},t_l]$ and $t\in[t_l,t_{l+1}]$, we have
\begin{equation*}
\begin{aligned}
&\bigg(\frac{|B_t^{\vare}-B_t-(B_s^{\vare}-B_s)|}{|t-s|^{\al/\beta}}\bigg)^\beta\leq\bigg[\big|B_{t_l}+(t-t_l)\frac{2^k}{T}(B_{t_{l+1}}-B_{t_l})-B_t\\
&\qquad-(B_{t_l}+(s-t_l)\frac{2^k}{T}(B_{t_{l-1}}-B_{t_l})-B_s)\big|/|t-s|^{\al/\beta}\bigg]^\beta\\
&\leq\bigg[\big|B_t-B_s\big|+(t-t_l)\frac{2^k}{T}\big|B_{t_{l+1}}-B_{t_l}\big|+(t_l-s)\frac{2^k}{T}|B_{t_{l-1}}-B_{t_l}|/|t-s|^{\al/\beta}\bigg]^\beta\\
&\leq|t-s|^{-\al}\bigg[\frac{\big|B_t-B_s\big|}{|t-s|^{\al/\beta}}|t-s|^{\al/\beta}+(t-s)\frac{2^k}{T}\bigg(\frac{\big|B_{t_{l+1}}-B_{t_l}\big|}{|t_{l+1}-t_l|^{\al/\beta}}\\
&\qquad\times|t_{l+1}-t_l|^{\al/\beta}+\frac{|B_{t_{l-1}}-B_{t_l}|}{|t_l-t_{l-1}|^{\al/\beta}}|t_{l}-t_{l-1}|^{\al/\beta}\bigg)\bigg]^\beta\\
&\leq|t-s|^{-\al}\bigg[C|t-s|^{\al/\beta}+(t-s)\frac{2^k}{T}\bigg(C|t_{l+1}-t_l|^{\al/\beta}+C|t_{l}-t_{l-1}|^{\al/\beta}\bigg)\bigg]^\beta\\
&\leq C_\beta+C_\beta|t-s|^{\beta-\al}(\frac{2^k}{T})^\beta(\frac{T}{2^k})^{\al}\leq C_\beta.
\end{aligned}
\end{equation*}
  So by \eqref{e.4.4aa} we conclude  
\begin{equation}
\begin{aligned}
&\frac{|B_t^{\vare}-B_t-(B_s^{\vare}-B_s)|}{|t-s|^\al}\le C_\beta (|B_t^{\vare}-B_t|+|B_s^{\vare}-B_s)|)^{1-\beta},
\end{aligned}
\end{equation}
and then
\begin{equation*}
\begin{aligned}
\Vert B^{\vare}(\om_2)-B(\om_2)\Vert_{\alpha;[0,T]}&\le C_\beta\sup_{0<s<t<T} (|B_t^{\vare}-B_t|+|B_s^{\vare}-B_s|)^{1-\beta}\rightarrow0,\ a.s.\  \omega_2\in\Omega_2.
\end{aligned}
\end{equation*}
Now,  it is not hard to prove that $\Gamma^{\vare}\rightarrow\Gamma$ with the fact that the solution map $B^{\vare}\mapsto\Gamma^{\vare}$ is continuous in the sense of Young theory  (Proposition \ref{p.4.1}). Similar convergence result can be obtained for $(\tilde{B}^{\vare},\Lambda^{\vare})$.
\end{proof}
Next,  we show that  the multiple pathwise integral
\begin{equation}
\mathcal{S}_{n,t}(f_n)(B ):=\int_{0<t_1<\cdots<t_n<t}f_n(t_1,\cdots,t_n)dB^{\sigma^{-1}(i_1)}_{t_1}\cdots dB^{\sigma^{-1}(i_n)}_{t_n}
\end{equation}
  is the almost surely convergence limit of 
\begin{equation}
\mathcal{S}_{n,t}(f_n)(B^{\vare}):=\int_{0<t_1<\cdots<t_n<t}f_n(t_1,\cdots,t_n)dB^{\sigma^{-1}(i_1),\vare}_{t_1}\cdots dB^{\sigma^{-1}(i_n),\vare}_{t_n}\,,
\end{equation}
where the integral is in the pathwise sense which contributes to the explicit form of solution and hence guarantee the existence of solution to equation \eqref{e.3.1} and
\begin{equation}
f_n(t_1,t_2,\dots,t_n)=[[\dots [A_{i_1}(t_1),A_{i_2}(t_2)]\dots],A_{i_n}(t_n)],\ t_1< t_2<\dots< t_n.
\end{equation}
We also define
 \begin{equation}
g_n(t_1,t_2,\dots,t_n)=[[\dots [C_{i_1}(t_1),C_{i_2}(t_2)]\dots],C_{i_n}(t_n)],\ t_1< t_2<\dots< t_n.
\end{equation}

%

\begin{theorem}\label{multiconvergelt0.5}
Let {\bf (H1)} hold and $H\in (1/2, 1)$. Let $B^{\vare},\tilde{B}^{\vare}$ be the dyadic piecewise linear smooth approximations of $B,\tilde{B}$, respectively.
Then
\begin{equation}\label{multi converge3}
\begin{aligned}
&\mathcal{S}_{n,t}(f_n)(B^{\vare})\rightarrow \mathcal{S}_{n,t}(f_n)(B),\ a.s.\ \omega_2\in\Omega_2,\\
&\mathcal{S}_{n,t}(g_n)(\tilde{B}^{\vare})\rightarrow \mathcal{S}_{n,t}(g_n)(\tilde{B}),\ a.s.\ \omega_2\in\Omega_2\,. 
\end{aligned}
\end{equation}
\end{theorem}
\begin{proof}
By {\bf (H1)}, it is straightforward to  obtain 
\begin{equation}
\begin{aligned}
&|\mathcal{S}_{n,t}(f_n)(B^{\vare})-\mathcal{S}_{n,t}(f_n)(B)|\leq C_n\sum_{r=1}^n\Vert B^{\sigma^{-1}(i_r),\vare}-B^{\sigma^{-1}(i_r)}\Vert_{\alpha;[0,T]},\ a.s.,\\
&|\mathcal{S}_{n,t}(g_n)(\tilde{B}^{\vare})-\mathcal{S}_{n,t}(g_n)(\tilde{B})|\leq C_n\sum_{r=1}^n\Vert\tilde{B}^{\sigma^{-1}(i_r),\vare}-\tilde{B}^{\sigma^{-1}(i_r)}\Vert_{\alpha;[0,T]},\ a.s.,
\end{aligned}
\end{equation}
according to Theorem 3.4-(iii) in \cite{Hu13}. Then similar to the proof in Theorem \ref{Young WZ}, we also obtain that $B^{\sigma^{-1}(i_r),\vare}, \tilde{B}^{\sigma^{-1}(i_r),\vare}$ converge almost surely to $B^{\sigma^{-1}(i_r)}, \tilde{B}^{\sigma^{-1}(i_r)}$ in $\mathcal{C}^\alpha([0,T];\mathbb{R}^{m_1}(\mathbb{R}^{m_2}))$, respectively. it is not hard to get the desired convergence result \eqref{multi converge3}. Therefore, the proof is complete.
\end{proof}

With the help of the Theorem \ref{multiconvergelt0.5}, we obtain the following theorem.
\begin{theorem}
Under nilpotent case, i.e. assumption {\bf (H1)}, 
the solution to \eqref{e.3.1} can be written as
\begin{equation}
\Gamma _t=\exp \left\{\mathcal{K} _t\right\}
  \qquad \hbox{and}\quad \Lambda _t=\exp \left\{\chi _t\right\}\,, 
\end{equation}
where 
\begin{equation}
\begin{aligned}
\mathcal{K} _t=&\sum^{N_0}_{n=1}\sum_{(i_1,\dots,i_n)}\sum_{\sigma\in\mathcal{S}_n}\frac{(-1)^{e(\sigma)+n}}{n^2\binom{n-1}{e(\sigma)}}\int_{0< s_1<\dots< s_n< t}[\dots [A_{i_1}(s_1),A_{i_2}(s_2)]\dots],A_{i_n}(s_n)]\\
&\times dB^{\sigma^{-1}(i_1) }_{s_1}dB^{\sigma^{-1}(i_2) }_{s_2}\dots dB^{\sigma^{-1}(i_n) }_{s_n}\\
\end{aligned}
\end{equation}
and
\begin{equation}
\begin{aligned}
\chi  _t=&\sum^{N_0}_{n=1}\sum_{(i_1,\dots,i_n)}\sum_{\sigma\in\mathcal{S}_n}\frac{(-1)^{e(\sigma)+n}}{n^2\binom{n-1}{e(\sigma)}}\int_{0< t_1<\dots< t_n< t}[\dots [C_{i_1}(t_1),C_{i_2}(t_2)]\dots],C_{i_n}(t_n)]\\
&\times d\tilde{B}^{\sigma^{-1}(i_1) }_{t_1}d\tilde{B}^{\sigma^{-1}(i_2) }_{t_2}\dots d\tilde{B}^{\sigma^{-1}(i_n) }_{t_n}\,.
\end{aligned}
\end{equation}
\end{theorem} 
\begin{proof}
For any $r_1\leq N_0$, by Theorem \ref{multiconvergelt0.5}, 
\begin{equation}
\int_{0< s_1<\dots< s_{r_1}< t}[[\dots [A_{i_1}(s_1),A_{i_2}(s_2)]\dots],A_{i_{r_1}}(s_{r_1})] dB^{\sigma^{-1}(i_1),\vare }_{s_1}\dots dB^{\sigma^{-1}(i_{r_1}),\vare }_{s_{r_1}}
\end{equation}
converges  to 
\begin{equation}
\int_{0< s_1<\dots< s_{r_1}< t}[[\dots [A_{i_1}(s_1),A_{i_2}(s_2)]\dots],A_{i_{r_1}}(s_{r_1})] dB^{\sigma^{-1}(i_1) }_{s_1}\dots dB^{\sigma^{-1}(i_{r_1}) }_{s_{r_1}}\,. 
\end{equation}
Since only  finite sum is  involved  we see that $\Gamma_t^\vare$ converges to $\Gamma_t$.
Similarly, we see that $\Lambda_t^\vare$ converges to $\Lambda_t$.
\end{proof}

\section{The convergence in the rough path's framework}
In this section, we consider the convergence of $\Gamma^{\vare}$ and $\Lambda^{\vare}$ to
$\Gamma $ and $\Lambda $ defined by \eqref{approximation ode}
and \eqref{e.3.1}, respectively, when $H\in (1/3, 1/2)$.


Let V be a Euclidean space with norm $|\cdot|_V$, and for each $p$, let $V^{\otimes p}$ denote the $p$-th tensor product endowed with a compatible norm $|\cdot|_{V^{\otimes p}}$.  

For any H\"older continuous function $X_t$ of H\"older exponent $\al\in (1/3,
 1/2)$  from $[0, T]$ to a Euclidean space,
we can lift  it to a   multiplicative functional 
$\mathbf{X}_{s,t} =(1,\mathbb{X}^{1 }_{s,t},\mathbb{X}^{2 }_{s,t})$  
in $T^{(2)}(V)=  \mathbb{R}\oplus V\oplus V^{\otimes2}$   on the simplex $\Delta$
satisfying  the Chen's identity:
$\mathbf{X} _{s,t}=\mathbf{X}_{s,u}\otimes\mathbf{X} _{u,t},
 \forall \ s\leq u\leq t.$

Let $2<p<3$. A multiplicative functional $\mathbf{X}$ in $T^{(2)}(V)$ is said to have finite $p$-variation if 
$\sup_D(\sum_l|\XX^i_{t_{l-1},t_l}|^{p/i}_{V^{\otimes i}})^{i/p}<\infty,\ i=1,2,$
where the supremum runs over all finite partitions $D$ on $[0,T]$. And $\mathbf{X}$ satisfies Chen's relation: $\XX^2_{s,t}-\XX^2_{s,u}-\XX^2_{u,t}=\XX^1_{s,u}\otimes\XX^1_{u,t}, \forall\ (s,u,t)\in\Delta_2$.Then we call $\mathbf{X}$ a two-step $p$-rough path with values in $V$, and denoted by $\mathbf{X}\in\cC^{p \mhyphen var}([0,T];V)$.

We shall apply this lifting to the fBms $B$ and $\tilde B$ as well as their approximations with similar notations (more literature related to rough path, see \cite{FH20},\cite{FV10},\cite{Lyos98},\cite{LQ02}).

We denote the control function for any two lifts $\mathbf{X}$ and $\mathbf{Y}$: 
\begin{equation}
\omega (s,t)=\sum_{i=1}^2 \sum_{D_{[s,t]}}\sum_l|\mathbb{X}^i_{t_{l-1},t_l}-\mathbb{Y}^i_{t_{l-1},t_l}|^{p/i}\,,
\end{equation}
and the distance between two lifts   $\mathbf{X}$ and $\mathbf{Y}$:
\begin{equation}
d_p(\mathbf{X},\mathbf{Y})=\sup_{1\leq i\leq\lfloor p\rfloor}(\sup_{D_{[s,t]}}\sum_l|\mathbb{X}^i_{t_{l-1},t_l}-\mathbb{Y}^i_{t_{l-1},t_l}|^{p/i}_{V^{\otimes i}})^{i/p}\,.
\end{equation}

Now we give the  Wong-Zakai theorem in the rough path's framework. 
\begin{theorem}\label{rough WZ}
Let $B=(B^{1},B^{2},\cdots,B^{m_1})\in\mathbb{R}^{m_1}$ and $\tilde{B}=(\tilde{B}^{1},\tilde{B}^{2},\cdots,\tilde{B}^{m_2})\in\mathbb{R}^{m_2}$ be two fBms with $H\in(1/3,1/2)$, and $B^{\vare}, \tilde{B}^{\vare}$ be the dyadic piecewise linear approximation of $B,\tilde{B}$, respectively. Then the smooth solutions $\Gamma^{\vare}$ and $\Lambda^{\vare}$  
of \eqref{approximation ode} almost surely
converge to the solutions of rough differential equation (RDE for short):
\begin{equation}\label{e.4.4}
\left\{
\begin{aligned}
&d\Gamma_t=-\sum^{m_1}_{j=1}\Gamma_tA_{j}(t)d\mathbf{B}^j_t,\\
&d\Lambda_t=-\sum^{m_2}_{j=1}\Lambda_tC_{j}(t)d\mathbf{\tilde{B}}^i_t\,.  
 \end{aligned}
\right.
\end{equation}
\end{theorem}
\begin{proof} This result is an application of the rough path theory. In fact, 
from  Coutin and Qian \cite[Theorem 2]{CQ02},   
 $\mathbf{B}_{s,t}^{\vare}=(1,\mathbb{B}_{s,t}^{1,\vare},\mathbb{B}_{s,t}^{2,\vare})$ and $\tilde{\mathbf{B}}_{s,t}^{\vare}=(1,\tilde{\mathbb{B}}_{s,t}^{1,\vare},\tilde{\mathbb{B}}_{s,t}^{2,\vare})$ converge to the geometric rough path $\mathbf{B}_{s,t}=(1,\mathbb{B}_{s,t}^{1},\mathbb{B}_{s,t}^{2})$ and $\tilde{\mathbf{B}}_{s,t}=(1,\tilde{\mathbb{B}}_{s,t}^{1},\tilde{\mathbb{B}}_{s,t}^{2})$ almost surely
   in $\mathcal{C}^{p \mhyphen var}([0,T];\mathbb{R}^{m_1}(\mathbb{R}^{m_2}))$, respectively, for any $p<3$ such that $Hp>1$. On the other hand, the continuity of  the It\^o-Lyons map (\cite[Theorem 6.3.1, Corollary 6.3.2]{LQ02} for geometric rough path) $\mathbf{B}_{s,t}^{\vare}\mapsto \Gamma^{\vare}$ and $\tilde{\mathbf{B}}_{s,t}^{\vare}\mapsto \Lambda^{\vare}$ in the rough path theory implies that $\Gamma^{\vare}_t,\Lambda^{\vare}_t$ converge to the solutions $\Gamma_t,\Lambda_t$, respectively. 
\end{proof}
\begin{remark}
The stochastic integrals in RDE \eqref{e.4.4} are well-posed in the sense of rough stochastic integral (Theorem \ref{rough int}) where $Z_s=f(\Gamma_s)=-\Gamma_sA_j(s)$ and $(Z_s,Z^\prime_s)$ is controlled rough path with Gubinelli derivative $Z^\prime_s=Df(\Gamma_s)\Gamma^\prime_s+f^\prime(\Gamma_s)$ where $f^\prime$ is the Gubinelli derivative of $f$. Moreover, the solution $(\Gamma,\Gamma^\prime)$ to RDE is controlled rough path with the Gubinelli derivatives $\Gamma_s^\prime=f(\Gamma_s)$.
\end{remark}
 
 In Theorem \ref{rough WZ}, we give the solution to RDE in the sense of Wong-Zakai limit. Then,  to find the explicit representation of the solution and to  show  the existence of solution 
 for the limiting  equation,  it suffices to show the approximated multiple rough integrals converges. Let us focus on $(B^{\vare},\Gamma^{\vare})$, and the case of $(\tilde{B}^{\vare},\Lambda^{\vare})$ are similar. We recall
\begin{equation}\label{K^k}
\begin{aligned}
\mathcal{K}^{\vare}_t&=\sum^{\infty}_{n=1}\sum_{(i_1,\dots,i_n)}\sum_{\sigma\in\mathcal{S}_n}\frac{(-1)^{e(\sigma)+n}}{n^2\binom{n-1}{e(\sigma)}}I_n(f_n)(B^{\vare})\,,  
\end{aligned}
\end{equation}
where
\begin{equation}
\begin{split}
f_n:=&f_n(t_1,t_2, \dots, t_n)=[\dots [A_{i_1}(t_1),A_{i_2}(t_2)]\dots],A_{i_n}(t_n)]\,,\\
I_n(f_n)(B^{\vare}):=&I_n(f_n)(B^{\vare})(t)=\int_{0< t_1<\dots< t_n< t}f_n(t_1,\dots,t_n)dB^{\sigma^{-1}(i_1),\vare}_{t_1}\dots dB^{\sigma^{-1}(i_n),\vare}_{t_n}\,.
\end{split} \label{e.4.5} 
\end{equation} 
%
%
%
We shall repeatedly use the following integration by parts formula to study the above multiple integral $I_n(f_n)$:
\begin{equation}\label{integration by part}
\int_a^bf(t)dB^{i}_t=f(b)B^{i}_b-f(a)B^{i}_a-\int_a^bf^{\prime}(t)B^{i}_tdt,
\end{equation}
which is the special case of It\^o's formula for $f(\cdot)B^i_{\cdot}$ (see \cite[Theorem 7.7]{FH20}). Moreover, the integrand $f(\cdot)$ (in fact, $f_n$ in \eqref{e.4.5}), corresponding to the assumption (H1), should be smooth (differentiable with respect to $t$ up to certain order), which is enough to guarantee the well-posedness of integrals $\int fdB^i$.

Then we can obtain the following lemma.

\begin{lemma} Let multi-index $I=(i_1,\dots,i_n)$  and denote by $f_n$ the Lie commutator defined by $f_n:=[\dots [A_{i_1}(t_1),A_{i_2}(t_2)]\dots],A_{i_n}(t_n)]$.
Then, we have
\begin{equation}\label{general formula}
\begin{aligned}
 I_n(f_n)(B^{\vare})
 = & \sum_{m=0}^{n} 
(-1)^m\sum_{\Delta(m,n)}\int_{0<t_{j_1}<\cdots< t_{j_m}<t} 
 \frac{\partial^m f_n(\overbrace{t_{j_1},\cdots,t_{j_1}}^{n_1},  \cdots,\overbrace{t_{j_m},\cdots,t_{j_m}}^{n_m},\overbrace{t\cdots,t}^{n_{m+1}})}{\partial t_{j_1}\cdots\partial t_{j_m}} \\
&\times B_t^{i_{j_{m}+1},i_{n},\vare} B_{t_{j_m}}^{i_{j_{m-1}+1},i_{j_m},\vare}\cdots B_{t_{j_1}}^{i_1,i_{j_1},\vare}dt_{j_1}dt_{j_2}\cdots dt_{j_m}\,, 
\end{aligned}
\end{equation}
where we need some notations for simplification: $\Delta(m,n):=\{1\leq j_1<\cdots< j_m\leq n\}$ and $S([0,t],n):=\{0<t_1<\cdots<t_n<t\}$, $B_{t}^{i_{r},i_{s}}:=B^{\sigma^{-1}(i_r),\dots,\sigma^{-1}(i_s)}_{t}$, and
\begin{equation}
\begin{aligned}
&B_t^{i_1,i_n,\vare}=\int_{S([0,t],n)}dB^{\sigma^{-1}(i_1),\vare}_{t_1}dB^{\sigma^{-1}(i_2),\vare}_{t_2}\cdots dB^{\sigma^{-1}(i_n),\vare}_{t_n}\,. 
\end{aligned}
\end{equation}
and other iterated forms in \eqref{general formula} have similar representation. Moreover, $n_p$  denotes  the cardinalities of $t_{j_p}$, respectively, for $p=1,2,\dots,m$, and $n_{m+1}$ is the cardinality of $t$, where $n_1=j_1,n_2=j_2-j_1,\dots,n_m=j_m-j_{m-1},n_{m+1}=n-j_m$ such that $n_1+n_2+\cdots+n_m+n_{m+1}=n$.
\end{lemma}
\begin{proof}
In the following, we prove this general formula by induction. First for $n=1$, we have
already proved the result by \eqref{integration by part}. 
Assume that the case of $n$ holds for \eqref{general formula}. Then for  $n+1$ we have 
\[
I_{n+1}(f_{n+1})(B^{\vare})(t)= \int_0^t I_{n}(f_{n+1}(\cdot, t_{n+1}))
(t_{n+1})  d B^{\si^{-1}(i_{n+1}),\vare}_{t_{n+1}}\,.
\]
We substitute the above $ I_{n}(f_{n+1}(\cdot, t_{n+1}))
(t_{n+1}) $ by \eqref{general formula} to obtain  
\begin{equation*}
\begin{aligned}
 &I_{n+1}(f_{n+1})(B^{\vare})\\
=&\sum_{m=0}^{n }(-1)^m\sum_{\Delta(m,n)}\int_{0<t_{j_1}<\cdots<t_{j_m}<t_{n+1}<t}\frac{\partial^m f_{n+1}(t_{j_1},\cdots,t_{j_1},\cdots,t_{j_m},\cdots,t_{j_m},t_{n+1}\cdots,t_{n+1})}{\partial t_{j_1} \cdots\partial t_{j_m}}\\
&\times   B_{t_{j_1}}^{i_1,i_{j_1},\vare} \cdots B_{t_{j_m}}^{i_{j_{m-1}+1},i_{j_m},\vare}
 B_{t_{n+1}}^{i_{j_{m}+1},i_{n},\vare} dt_{j_1} \cdots dt_{j_m}dB_{t_{n+1}}^{\sigma^{-1}(i_{n+1}) ,\vare} \\
=& 
 \sum_{m=0}^{n }(-1)^m\sum_{\Delta(m,n)}\int_{0<t_{j_1}<\cdots<t_{j_m}<t_{n+1}<t}\frac{\partial^m f_{n+1}(t_{j_1},\cdots,t_{j_1},\cdots,t_{j_m},\cdots,t_{j_m},t_{n+1}\cdots,t_{n+1})}{\partial t_{j_1} \cdots\partial t_{j_m}}\\
&\times   B_{t_{j_1}}^{i_1,i_{j_1},\vare} \cdots B_{t_{j_m}}^{i_{j_{m-1}+1},i_{j_m},\vare}
   dt_{j_1} \cdots dt_{j_m}dB_{t_{n+1}}^{i_{j_{m}+1},i_{n+1  },\vare}\,. 
\end{aligned}
\end{equation*}  
Now we can apply the integration by parts 
to $dB_{t_{n+1}}^{i_{j_{m}+1},i_{n+1  },\vare}$ to complete the    induction.
\end{proof}
 
%
%
%
Now we prove that $I_n(f_n)(B^\vare)$ has a limit given by the following expression
\begin{equation}\label{multi limit}
\begin{aligned}
 I_n(f_n)(B) =& \sum_{m=0}^{n}(-1)^m\sum_{\Delta(m,n)}\int_{0<t_{j_1}<\cdots< t_{j_m}<t}\frac{\partial^m f_n(\overbrace{t_{j_1},\cdots,t_{j_1}}^{n_1},\cdots,\overbrace{t_{j_m},\cdots,t_{j_m}}^{n_m},\overbrace{t\cdots,t}^{n_{m+1}})}{\partial t_{j_1}\cdots\partial t_{j_m}}\\
&\times B_t^{i_{j_{m}+1},i_n}B_{t_{j_m}}^{i_{j_{m-1}+1},i_{j_m}}\cdots B_{t_{j_1}}^{i_1,i_{j_1}}dt_{j_1}  \cdots dt_{j_m}\,. 
\end{aligned}
\end{equation} 
\begin{theorem}\label{multi converge}
Let $H\in (1/3, 1/2)$ and let the assumption {\bf (H1)} hold. Let $f_n$ be defined by \eqref{e.4.5}. Then 
\begin{equation}\label{multi converge2}
I_n(f_n)(B^{\vare})\rightarrow I_n(f_n)(B),\ a.s.,\ when\ \vare\rightarrow0,
\end{equation}
where $I_n(f_n)(B^{\vare}),I_n(f_n)(B)$ represent \eqref{general formula} and \eqref{multi limit}, respectively.
\end{theorem}

\begin{proof}
Indeed, by  \cite[Theorem 2]{CQ02} and \cite[Theorem 3.1.2, 3.1.3]{LQ02}, 
we know that for any $ k\le\ell $,  $B_{s}^{k,\ell   ,\vare }$ converges almost surely to 
$B_{s}^{k,\ell  } $ uniformly in $s\in [0, T]$.  On the other hand, by the smoothness of $A_i(t)$ we see that
$\frac{\partial^m f_n(t_{j_1},\cdots,t_{j_1},\cdots,t_{j_m},\cdots,t_{j_m},t\cdots,t)}{\partial t_{j_1} \cdots\partial t_{j_m}}$ are continuous functions. Thus we have 
$I_n(f_n)(B^{\vare})$ converges to $ I_n(f_n)(B)$, proving the theorem. 
\end{proof}
%
\begin{remark}
In the Sections 4-5, we give the Wong-Zakai theorems and meanwhile discuss the convergence of multiple integrals in the case of $H\in(1/2,1)$ and $H\in(1/3,1/2)$, respectively. In some sense, two kinds of convergence ideas have been applied to guarantee the existence of solution to equations \eqref{e.3.1} and \eqref{e.4.4}. one is in the sense of the Wong-Zakai convergence limit, the other is divided into two steps: in step 1, use Campbell-Baker-Hausdorff-Dynkin formula to represent the solution to approximated (random) ODEs pathwisely; in step 2, give the solution to equations \eqref{e.3.1} and \eqref{e.4.4} by use of the convergence of multiple integrals. Therefore, by the Young and rough path's theory, respectively,  the well-posedness of equations \eqref{e.3.1} and \eqref{e.4.4} is obtained, which contributes to the transformation of system and solving the optimal control in the subsequent sections.
\end{remark}

\section{Transformed equivalent system and its optimal control problem}\label{sect5}
%
In this section, we consider the following controlled system driven by both Brownian motion and fBm: 
\begin{equation}\label{e.6.1}
dX^u_t=b(t,X^u_t,u_t)dt+\sum_{j=1}^{k_1}\sigma_j(t,X^u_t,u_t)dW^j_t+\sum_{j=1}^{m_1}A_{j}(t)X^u_tdB_t^{j}.
\end{equation}

Let $\Ga_t$ and $\La_t$ be defined as in previous sections and define 
\[
Y_t^u :=\Ga_t X^u_t=\exp\{{\mathcal{K}_t}\}X^u_t\,.
\]
Now we would like to  apply It\^o  formula to the above product.  In fact, 
when $H\in(1/2,1)$, for every $\om_2\in\Om_2$, $\int A_{j}X^u dB(\om_2)^{j}$(or $\int\Gamma_tA_{j}d\mathbf{B}(\om_2)^j$) is  a Young integral so that  the It\^o's formula (see \cite[Proposition 2.7]{Hu13}) can be applied to $\Gamma_t X_t^u$ pathwisely.

When $H\in(1/3,1/2)$, for every $\om_2\in\Om_2$, we can apply the
 rough It\^o's formula \cite[Theorem 4.13]{FHL22arxiv} to 
 $\Gamma_tX^u_t$ pathwisely  (we shall briefly explain the applicability in the following Remark
 \ref{r.6.1}) and notice that the rough path bracket $[\mathbf{B}]$ is zero since rough path  $\mathbf{B}(\om_2)$  is geometric. Thus for  $H\in (1/3, 1/2)\cup (1/2, 1)$,    we have
\begin{equation}\label{e.6.2}
\begin{aligned}
dY^u_t&=d\Gamma_tX^u_t =\Gamma_tdX^u_t+d\Gamma_tX^u_t\\
&=\Gamma_tb(t,X^u_t,u_t)dt+\sum_{j=1}^{k_1}\Gamma_t\sigma_j(t,X^u_t,u_t)dW^j_t+\sum_{j=1}^{m_1}\Gamma_tA_j(t)X^u_td\mathbf{B}^j_t-\sum_{j=1}^{m_1}\Gamma_tA_j(t)X^u_td\mathbf{B}^j_t\\
&=\Gamma_tb(t,\Gamma^{-1}_tY^u_t,u_t)dt+\sum_{j=1}^{k_1}\Gamma_t\sigma_j(t,\Gamma_t^{-1}Y^u_t,u_t)dW^j_t\,. 
\end{aligned}
\end{equation}
\begin{remark}\label{r.6.1} 
We briefly explain the applicability of the  It\^o's formula \cite[Theorem 4.13]{FHL22arxiv} to our case \eqref{e.4.4}, \eqref{e.6.1}.  In fact, for every $\om_2\in\Om_2$, the state equation \eqref{x} is a  rough SDE in \cite[
Eq (4.1)]{FHL22arxiv}.
Under {\bf (H2)-(H3)}, let $m\in[2,\infty)$, for any $u\in U$, the existence and uniqueness of $L_m$-integrable solution to \eqref{x} can be obtained ($\om_2$-pathwisely) by a  fixed-point argument (a similar  proof to that in \cite[Theorem 4.7]{FHL22arxiv}) under the framework  in Section 2.
The conditions for the validity of the It\^o's formula in \cite[Theorem 4.13]{FHL22arxiv} concerning $X_t$ are easy to verify. Moreover, RDE \eqref{e.4.4} is a special case of rough SDE (without $dW$ integral), whose   solvability is guaranteed by \cite[Theorem 8.3]{FH20}. 
This explains the applicability  of It\^o's formula  to $\vp(V)=V_1V_2$, $V=(V_1,V_2)=(\Gamma,X)$. 
\end{remark} 

Then we transform the original  system of  state   equations  into the following one with 
the term containing the 
differentiation with respect to fBm,  namely $\sum_{j=1}^{m_1}A_{j}(t)X^u_tdB_t^{j}$, 
disappeared: 
\begin{equation}\label{new state}
\left\{
\begin{aligned}
dY^u_t&=\Gamma_tb(t,\Gamma_t^{-1}Y^u_t,u_t)dt+\Gamma_t\sum_{j=1}^{k_1}\sigma_j(t,\Gamma_t^{-1}Y^u_t,u_t)dW^j_t,\\
d\Gamma_t&=-\sum^{m_1}_{j=1}\Gamma_tA_{j}(t) dB^{j}_t,\\
Y^u_0&=x\in\mathbb{R}^n,\ \Gamma_0=I\in\mathbb{R}^{n\times n}.
\end{aligned}
\right.
\end{equation}
Similarly, we can also transform the observation equations  to the following: 
\begin{equation}\label{new observation}
\left\{
\begin{aligned}
d\zeta^u_t&=\Lambda_th(t,\Gamma_t^{-1}Y^u_t,u_t)dt+\Lambda_t\sum_{j=1}^{k_2}D_j(t)d\tilde{W}^j_t,\\
d\Lambda_t&=-\sum^{m_2}_{j=1}\Lambda_tC_{j}(t) d\tilde{B}^{j}_t,\\
\xi^u_0&=0,\ \Lambda_0=I\in\mathbb{R}^{k_2\times k_2}.
\end{aligned}
\right.
\end{equation}
With the relation $X^u_t=\Gamma_t^{-1}Y^u_t$,  the cost functional becomes 
\begin{equation}
J(u(\cdot))=\mathbb{E}\bigg[\Phi(\Gamma_T^{-1}Y^u_T)+\int_0^Tf(t,\Gamma_t^{-1}Y^u_t,u_t)dt\bigg]\,. 
\label{new cost} 
\end{equation}
%
The transformed optimal control problem 
\eqref{new state}, \eqref{new observation} and \eqref{new cost} is 
a classical control problem in the sense that   both the state   and
observation systems do not contain the diffusion terms dictated by fBm,
which are absorbed into the coefficients of the system.   

\begin{remark}
A difficulty posed by the fBm in the control problem is to  
find the limit $\frac{1}{\varepsilon } 
\left[ X^\varepsilon-\bar{X}\right]$
for the state given by \eqref{e.6.1} (in the case of convex 
control domain). In the case 
of only fBm with Hurst parameter $H>1/2$, this can be done 
by using the well-developed theory on Young integral. 
For fBm  
of Hurst parameter less than $1/2$, the situation 
is much more complex, in particular when 
additional  Brownian motions are involved. 
The transformation of \eqref{e.6.1} to \eqref{new state} 
avoids this difficulty.
\end{remark}

The available information to the controller is given by the filtration  
 $\mathcal{F}_t^\zeta=\si(\zeta_s, 0\le s\le t)$ generated by the (transformed) observation process up to time instant $t$.
The admissible control set is defined  by
\begin{equation*}
\begin{split}
\UU_{ad}=&\bigg\{u\bigg|u_t\ \text{is an}\ \mathcal{F}_t^{\zeta}\text{-}\text{adapted process with values in}\ U\ \\
&\qquad \text{such that} \sup_{0\leq t\leq T}\mathbb{E}[|u_t|^p]<\infty, \forall\ p=1,2,\cdots\bigg\}\,.
\end{split} 
\end{equation*}

%

We denote 
\begin{equation*}
\begin{aligned}
\rho^u_t:=\exp\bigg\{-\int_0^t\Big(D^{-1}(s)h(s,\Gamma_s^{-1}Y^u_s,u_s)\Big)^{\top}d\tilde{W}_s-\frac{1}{2}\int_0^t\Big|D^{-1}(s)h(s,\Gamma_s^{-1}Y^u_s,u_s)\Big|^2ds\bigg\} \,. 
\end{aligned}
\end{equation*}
By assumptions {\bf (H1)-(H2)}, we see that $\rho^u_t, t\ge 0$ is a martingale by Novikov's  condition.
We introduce a probability measure
\begin{equation}
\frac{d\bar{\mathbb{P}}}{d\mathbb{P}} = \rho_T^u \quad \hbox{or}\quad 
%
\frac{d\mathbb{P}}{d\bar{\mathbb{P}}}:=\tilde{\rho}^u_T=(\rho_T^u)^{-1}\,. 
\end{equation}
Under this new probability measure   $\bar{\mathbb{P}}$, $\zeta$ is a standard Brownian motion and $W,\zeta, B,\tilde{B}$ are mutually independent standard Brownian motions and fBms. 

We denote the expectation by $\bar{\mathbb{E}}$. Now, the new expectation $\bar{\mathbb{E}}=\bar{\mathbb{E}}^{\omega_1,\omega_2}$ is taken for $W,\zeta, B,\tilde{B}$ on $\Omega_1\times\Omega_2$. 
%
%
It is known  that $\tilde{\rho}^u_t$ satisfies the following equation
\begin{equation}\label{tilderho}
\left\{
\begin{aligned}
&d\tilde{\rho}^u_t=\tilde{\rho}^u_t(D^{-1}(t)h(t,\Gamma^{-1}_tY^u_t,u_t))^{\top}d\zeta_t,\\
&\tilde{\rho}_0^u=1\in\mathbb{R}\,. 
\end{aligned}
\right.
\end{equation}

\subsection{The transformed state equation}

In this subsection, we obtain  some estimates on the  solutions $\gga,\Lambda$ and the existence and uniqueness of solution Y to the above transformed equations
\eqref{new state}-\eqref{new observation}. First, we give the following lemma for the estimates of solution to \eqref{tilderho}. The proof is routine and is referred to \cite{LT95}.
\begin{lemma}
For any $u\in\UU_{ad}$ and $p\geq2$, we have $\bar{\mathbb{E}}[\sup_{t\in[0,T]}|\tilde{\rho}^u_t|^p]<\infty$.

\end{lemma}

\subsubsection{The case of $H\in(1/2,1)$}
First, we obtain some estimates on the solutions $\gga,\Lambda$ when $H$ is greater than $1/2$.
\begin{theorem}\label{solution estimate2}
Let assumption {\bf (H1)} hold and $H\in(1/2,1)$, $B\in\mathcal{C}^\alpha([0,T];\mathbb{R}^{m_1})$ and $\tilde{B}\in\mathcal{C}^\alpha([0,T];\mathbb{R}^{m_2})$ be two sample paths with $\alpha\in(1/2,H)$. Let $\Gamma,\Lambda$ satisfy the Young differential equation in \eqref{new state} and \eqref{new observation}, respectively. Then for all $p\geq 2$, 
\begin{equation}
\begin{aligned}
&\bar{\mathbb{E}}\bigg[\sup_{t\in[0,T]}|\Gamma_t|^p\bigg]<\infty,\quad  \bar{\mathbb{E}}\bigg[\sup_{t\in[0,T]}|\Lambda_t|^p\bigg]<\infty,\\
&\bar{\mathbb{E}}\bigg[\sup_{t\in[0,T]}|\Gamma_t^{-1}|^p\bigg]<\infty,\quad  \bar{\mathbb{E}}\bigg[\sup_{t\in[0,T]}|\Lambda_t^{-1}|^p\bigg]<\infty\,. 
\end{aligned}
\end{equation}
 
\end{theorem}
\begin{proof}
The estimates of solutions in the first line are the direct consequence of \cite[Proposition 8.12]{FH20}, and similarly, the estimates in the second line can also be obtained by first applying It\^{o} formula (\cite[Proposition 2.7]{Hu13}) to $\Gamma^{-1}$ and $\Lambda^{-1}$, respectively.
\end{proof} 

\subsubsection{ The case of $H\in(1/3,1/2)$} 
We first introduce the following definition of our another space of rough paths of the $\alpha$-H\"older regularity, where $\alpha=1/p$ if the control $\omega$
 satisfies that $\omega(s,t)\leq C|t-s|$.
\begin{definition}\label{def.6.3}
For $\alpha\in(1/3,1/2)$, define the space of $\alpha$-H\"older rough paths (over $V$), in symbols $D^{\alpha}([0,T];V)$, as those tuples $\mathbf{X}=(\mathbb{X}^1,\mathbb{X}^2)$ such that
{\footnotesize\begin{equation*}
\begin{aligned}
 \Vert\mathbb{X}^i\Vert_{i\alpha;[0,T]}:=\sup_{0\leq s<t\leq T}\frac{|\mathbb{X}^i_{s,t}|}{|t-s|^{i\alpha}}<\infty,\ \text{for each}\ i=1,2,
\end{aligned}
\end{equation*}}
and such that $\mathbf{X}$ satisfies the so called ``Chen's identity''. Meanwhile, we introduce the associated $\alpha$-H\"older rough path norm: $|||\mathbf{X}|||_{\alpha;[0,T]}:=\sum_{i=1}^2\Vert\mathbb{X}^i\Vert_{i\alpha;[0,T]}^\frac{1}{i}$.
\end{definition}

Then, we can obtain the following estimates for the pathwise solution $\Gamma(\omega_2),\Lambda(\omega_2)$, respectively.
\begin{theorem}\label{solution estimate1}
Let assumption {\bf (H1)} hold and $H\in(1/3,1/2)$, $\mathbf{B}\in D^\alpha([0,T];\mathbb{R}^{m_1})$ and $\tilde{\mathbf{B}}\in D^\alpha([0,T];\mathbb{R}^{m_2})$ be two rough paths with $\alpha\in(1/3,H)$. Let $\Gamma,\Lambda$ satisfy the RDE in \eqref{new state} and \eqref{new observation}, respectively, we then have for $\bar{\mathbb{P}}\text{-}a.s.\ \omega_2\in\Omega_2$ that
\begin{equation}
\begin{aligned}
&\sup_{t\in[0,T]}|\Gamma_t| <\infty,\quad   \sup_{t\in[0,T]}|\Lambda_t|
<\infty,\quad \sup_{t\in[0,T]}|\Gamma_t^{-1}|<\infty,\quad  \sup_{t\in[0,T]}|\Lambda_t^{-1}| <\infty\,. 
\end{aligned}\label{e.5.12} 
\end{equation} 
\end{theorem}
\begin{proof}
For almost surely $\omega_2$, the estimates of pathwise solutions for the first two terms are the  direct consequence of \cite[Proposition 8.13]{FH20}. For the estimates of the last two terms, we can get the equations satisfied by $\Gamma^{-1}$ and $\Lambda^{-1}$ by applying rough It\^{o}'s formula, then the estimates of the last two terms can be obtained similarly.
\end{proof}

\subsubsection{The existence and uniqueness of solution Y to transformed equation}

First, we introduce some solution spaces  as follows
{\footnotesize\begin{equation*}
\mathcal{S}^p([0,T];\mathbb{R}^n)=\Big\{x\big|x\mbox{ is }\mathbb{R}^n\mbox{-valued}\ \FF\mbox{-adapted process such that}\ \bar{\mathbb{E}}\Big[\sup_{0\leq t\leq T}|x_t|^p\Big]<\infty\Big\},
\end{equation*}}
with norm $\Vert x\Vert_p^p:=\bar{\mathbb{E}}\big[\sup_{0\leq t\leq T}|x_t|^p\big]$;
{\footnotesize\begin{equation*}
\mathcal{M}^{2,p}([0,T];\mathbb{R}^n)=\Big\{z^i\big|z^i\mbox{ is }\mathbb{R}^n\mbox{-valued}\ \FF\mbox{-adapted process such that}\ \bar{\mathbb{E}}\bigg[\Big(\int_0^T|z^i_t|^2dt\Big)^{\frac{p}{2}}\bigg]<\infty\Big\},
\end{equation*}}
with norm $\Vert z^i\Vert_{2,p}^p:=\bar{\mathbb{E}}\big[(\int_0^T|z^i_t|^2dt)^{\frac{p}{2}}\big]$, for $i=1,\cdots,k_1(k_2)$.

Now we can state the following theorem.
\begin{theorem} Under the assumptions {\bf (H1)} and {\bf (H2)}, for any $u\in U$, the state equation in \eqref{new state}
has a unique solution    $Y$. 
\end{theorem} 

\begin{proof}  We only need to verify the global Lipschitz condition and integrability condition (due to the $(\om_2)$-randomness of the coefficients). 

First we check the global Lipschitz condition for  $\tilde{b}(t,y,u_t):=\Gamma_tb(t,\Gamma_t^{-1}y,u_t)$ and $\tilde{\sigma}(t,y,u_t):=\Gamma_t\sigma(t,\Gamma_t^{-1}y,u_t)$ with respect to $y$, respectively. Indeed, for any $y_1,y_2\in\mathbb{R}^n$, $(\Gamma_t,u_t)\in\mathbb{R}^{n\times n}\times U$, for $  d\bar{\mathbb{P}}\text{-}a.s., ( \omega_1,\omega_2)\in \Omega$, we have by {\bf (H2)}
{\footnotesize\begin{equation*}
\begin{aligned}
&|\tilde{b}(t,y_1,u_t)-\tilde{b}(t,y_2,u_t)|=|\Gamma_tb(t,\Gamma_t^{-1}y_1,u_t)-\Gamma_tb(t,\Gamma_t^{-1}y_2,u_t)|\\
&=\bigg|\Gamma_t\int_0^1b_X(t,\Gamma_t^{-1}y_2+\theta(\Gamma_t^{-1}y_1-\Gamma_t^{-1}y_2),u_t)d\theta\Gamma_t^{-1}(y_1-
y_2)\bigg|\\
&\leq|\Gamma_t|\int_0^1|b_X(t,\Gamma_t^{-1}y_2+\theta(\Gamma_t^{-1}y_1-\Gamma_t^{-1}y_2),u)|d\theta|\Gamma_t^{-1}||y_1- y_2|\leq C|y_1-y_2|\,. 
\end{aligned}
\end{equation*}}
This verifies 
the global Lipschitz condition for  $\tilde{b}(t,y,u_t)$. 
Similarly, we can verify   
the global Lipschitz condition for $ \tilde{\sigma}(t,y,u_t) $. 

Next, we show the integrability of $\tilde{b}(\cdot,0,0),\tilde{\sigma}(\cdot,0,0)$. Indeed, from {\bf (H2)}, we have
\begin{equation}\label{integrability}
\begin{aligned}
\bar{\mathbb{E}}\bigg|\int_0^T\tilde{b}(t,\om_2,0,0)dt\bigg|^p 
 &\leq\bar{\mathbb{E}}\bigg(\int_0^T|\Gamma_t(\om_2)b(t,0,0)|dt\bigg)^p \leq C\bar{\mathbb{E}}\bigg(\int_0^T|\Gamma_t(\om_2)|dt\bigg)^p\\
 &\leq C\bar{\mathbb{E}}\bigg[\sup_{t\in[0,T]}|\Gamma_t(\om_2)|^p\bigg].
\end{aligned}
\end{equation}
 
 To show the finiteness of the above last expectation, 
 we divide  our discussion into the two cases $H\in (1/2, 1)$ and $H\in (1/3, 1/2)$:
 
\noindent Case  $H\in(1/2,1)$. The  finiteness of  $\bar{\mathbb{E}}\Big[\sup_{t\in[0,T]}|\Gamma_t|^p\Big]<\infty$ is implied by  Theorem \ref{solution estimate2}.
This yields  $\bar{\mathbb{E}}\Big|\int_0^T\tilde{b}(t,\om_2,0,0)dt\Big|^p<\infty$.
 
\noindent Case  $H\in(1/3,1/2)$. For $\bar{\mathbb{P}}\text{-}a.s.\ \omega_2\in\Omega_2$, the solution $\Gamma_t(\omega_2)$ to RDE 
 \eqref{new state} is a sample path which does not contain $\omega_1\in\Omega_1$. So by Theorem \ref{solution estimate1}, we have $\bar{\mathbb{E}}^{\omega_1}\Big[\sup_{t\in[0,T]}|\Gamma_t(\omega_2)|^p\Big]=\sup_{t\in[0,T]}|\Gamma_t(\omega_2)|^p<\infty,$    where $\bar{\mathbb{E}}^{\omega_1}$ denotes the fact that it only take expectation for $\omega_1$ of Brownian motions when  $\omega_2$ of fBms is fixed. This means that    $\bar{\mathbb{P}}\text{-}a.s.\ \omega_2$, $\bar{\mathbb{E}}^{\om_1} \Big|\int_0^T\tilde{b}(t,0,0)dt\Big|^p<\infty$.

Similarly, we can obtain that $\bar{\mathbb{E}} \Big(\int_0^T|\tilde{\sigma}^r(t,\om_2,0,0)|^2dt\Big)^{\frac{p}{2}}<\infty$ as discussed above for two cases. 

The above analysis verifies  the conditions in \cite[Theorem 3.3.1]{Zhang17} which implies  that the first equation in \eqref{new state} admits a unique solution.  This also implies  the existence and uniqueness of \eqref{x} via transformation. 
\end{proof} 

\section{The   maximum principle for partially observed system}

In this section, we   obtain the  maximum principle
for our control problem. We shall use our transformations to transform our original problem to 
the ``classical problem'' \eqref{new state}, \eqref{new observation}, \eqref{new cost}. 
In the case $1/3<H<1/2$, we shall need to fix   $\omega_2$ to the optimization. This is 
because of the difference of the estimate \eqref{e.5.12} from that for the case $H\in (1/2, 1)$. 
However, except  the estimates we mentioned all other procedures  to obtain the maximum principle 
are the same regardless the value of $H$. Thus, 
we shall focus on the case    $H\in(1/2,1)$.
\subsection{ The maximum principle in the case of $H\in(1/2,1)$} 
\subsubsection{Variation}
Now,  since $U$ is not necessarily convex,   we utilize the spike variation technique. 
Suppose $\bar{u}\in\UU_{ad}$ is the optimal control minimizing  the 
cost functional \eqref{new cost}. We want to find the (necessary) condition that 
$\bar u$ must satisfy (find the maximum principle for $\bar u$).
To this end we introduce 
the spike variation of $\bar u$ as follows: 
\begin{equation*}
\begin{aligned}
u^{  \epsilon}_t=u^{\tau, \epsilon}_t=
\begin{cases}
u,&\text{$t\in E_{\epsilon}:=[\tau,\tau+\epsilon]$},\\
\bar{u}_t,&\text{otherwise},
\end{cases}
\end{aligned}
\end{equation*}
where $0<\tau<T$ is arbitrarily fixed, $\epsilon>0$ is arbitrarily 
 chosen such that $[\tau,\tau+\epsilon]\subset[0,T]$, and $u\in  \UU_{ad}  $ is an arbitrary bounded 
 admissible control.  We hope to use the fact that $J(u^{  \epsilon })
 \ge J(\bar u )$ for all $\tau, \vare$ such that $
 \tau, \tau+\vare\in [0, T], \varepsilon>0$
 to obtain a necessary condition satisfied by $\bar u$. 
Usually,  it is hard to obtain the maximum principle 
satisfied by the optimal control $\bar u$ by observing $J(v)\ge J(\bar u)$ 
for all $v\in \UU_{ad}$ directly. The idea is then to choose $v$ as the above small perturbation 
$u^{  \epsilon}$ of $\bar u$  and to  expand 
$J(u^{  \epsilon }) $  in a neighbourhood of $\vare=0$:
$J(u^{  \epsilon }) = J(\bar u) +J_0(\bar u)\varepsilon+o(\varepsilon) $.
The condition $J(u^{  \epsilon })
 \ge J(\bar u )$ will then imply $J_0(\bar u)\geq0$ for all $\tau$. This will produce the maximum principle. 
 
 To expand  $J(u^{  \epsilon }) $  we   let  $(\bar{Y},Y^{\epsilon})$ solve  \eqref{new state} for the corresponding 
 control  $\bar{u},u^{\epsilon}$, respectively. 
For any function $\phi$ defined on  $[0, T]\times \RR^n\times  \RR^d $ denote
\begin{equation}
\begin{aligned}
&\bar{\phi}(s)=\phi(s,\Gamma_s^{-1}\bar{Y}_s,\bar{u}_s),\ \bar{\phi}_X(s)=\phi_X(s,\Gamma_s^{-1}\bar{Y}_s,\bar{u}_s),\ \bar{\phi}_{XX}(s)=\phi_{XX}(s,\Gamma_s^{-1}\bar{Y}_s,\bar{u}_s),\\
&\delta\phi(s)=\phi(s,\Gamma_s^{-1}\bar{Y}_s,u_s)-\bar{\phi}(s),\ \delta\phi_X(s)=\phi_X(s,\Gamma_s^{-1}\bar{Y}_s,u_s)-\bar{\phi}_X(s)\,.  
\end{aligned}
\end{equation}
We shall  apply the above notation with   $\phi$ being replaced by  $b,\sigma^{r_1},h$.

Then we introduce the following first and second order variational equations for $Y,\tilde{\rho}$, respectively:
\begin{equation}
\begin{aligned}
&Y_t^1=\int_0^t\Gamma_s\bar{b}_X(s)\Gamma_s^{-1}Y_s^1ds+\sum_{r_1=1}^{k_1}\int_0^t\Big\{\Gamma_s\bar{\sigma}^{r_1}_X(s)\Gamma_s^{-1}Y^1_s+\Gamma_s\delta\sigma^{r_1}(s)I_{E_{\epsilon}}\Big\}dW_s^{r_1},\\
&Y_t^2=\int_0^t\Big\{\Gamma_s\bar{b}_X(s)\Gamma_s^{-1}Y_s^2+\frac{1}{2}\Gamma_s\bar{b}_{XX}(s)\Gamma_s^{-1}Y_s^1\Gamma_s^{-1}Y_s^1+\Gamma_s\delta b(s)I_{E_{\epsilon}}\Big\}ds\\
&\quad+\sum_{r_1=1}^{k_1}\int_0^t\Big\{\Gamma_s\bar{\sigma}^{r_1}_X(s)\Gamma_s^{-1}Y^2_s+\frac{1}{2}\Gamma_s\bar{\sigma}_{XX}^{r_1}(s)\Gamma_s^{-1}Y^1_s\Gamma_s^{-1}Y^1_s+\Gamma_s\delta\sigma^{r_1}_X(s)\Gamma_s^{-1}Y^1_sI_{E_{\epsilon}}\Big\}dW_s^{r_1}\,,
\end{aligned}
\end{equation}
and
\begin{equation}\label{rho12}
\begin{aligned}
\tilde{\rho}_t^1&=\int_0^t\Big\{\tilde{\rho}_s^1(D^{-1}(s)\bar{h}(s))^{\top}+\bar{\tilde{\rho}}_s[D^{-1}(s)\bar{h}_X(s)\Gamma_s^{-1}Y_s^1]^{\top}+\bar{\tilde{\rho}}_s[D^{-1}(s)\delta h(s)I_{E_{\epsilon}}]^{\top}\Big\}d\zeta_s,\\
\tilde{\rho}_t^2&=\int_0^t\Big\{\tilde{\rho}_s^2(D^{-1}(s)\bar{h}(s))^{\top}+\tilde{\rho}_s^1[D^{-1}(s)\bar{h}_X(s)\Gamma_s^{-1}Y^1_s]^{\top}+\tilde{\rho}_s^1[D^{-1}(s)\delta h(s)I_{E_{\epsilon}}]^{\top}\\
&\quad+\bar{\tilde{\rho}}_s[D^{-1}(s)\bar{h}_X(s)\Gamma_s^{-1}Y_s^2]^{\top}+\frac{1}{2}\bar{\tilde{\rho}}_s[D^{-1}(s)\bar{h}_{XX}(s)\Gamma_s^{-1}Y_s^1\Gamma_s^{-1}Y_s^1]^{\top}\\
&\quad+\bar{\tilde{\rho}}_s[D^{-1}(s)\delta h_X(s)I_{E_{\epsilon}}\Gamma_s^{-1}Y_s^1]^{\top}\Big\}d\zeta_s,
\end{aligned}
\end{equation}
where
\begin{equation}\label{n.7.5}
\Gamma_s\bar{b}_{XX}(s)\Gamma_s^{-1}Y_s^1\Gamma_s^{-1}Y_s^1\triangleq
\begin{pmatrix}
\text{tr}\{\Gamma_s\bar{b}_{XX}^1(s)\Gamma_s^{-1}Y_s^1(\Gamma_s^{-1}Y_s^1)^{\top}\}\\
\vdots\\
\text{tr}\{\Gamma_s\bar{b}_{XX}^n(s)\Gamma_s^{-1}Y_s^1(\Gamma_s^{-1}Y_s^1)^{\top}\}
\end{pmatrix}\,,
\end{equation}
and $\Gamma_s\bar{\sigma}^{r_1}_{XX}(s)\Gamma_s^{-1}Y_s^1\Gamma_s^{-1}Y_s^1, 1\leq r_1\leq k_1$, $\bar{h}_{XX}(s)\Gamma_s^{-1}Y_s^1\Gamma_s^{-1}Y_s^1$ have similar notations.

\begin{lemma}\label{Ytilderhoestimate}
Under the assumptions {\bf (H1)-(H2)},   we have the following estimates for any $p\geq2$:  
\begin{equation}\label{estimate1}
\begin{aligned}
&\bar{\mathbb{E}}\bigg[\sup_{0\leq s\leq T}|Y^1_s|^p\bigg]\leq C\epsilon^{\frac{p}{2}},\ \bar{\mathbb{E}}\bigg[\sup_{0\leq s\leq T}|Y^2_s|^p\bigg]\leq C\epsilon^p,\\
& \bar{\mathbb{E}}\bigg[\sup_{0\leq s\leq T}|\tilde{\rho}^1_s|^p\bigg]\leq C\epsilon^{\frac{p}{2}},\ \bar{\mathbb{E}}\bigg[\sup_{0\leq s\leq T}|\tilde{\rho}^2_s|^p\bigg]\leq C\epsilon^p\,. 
\end{aligned}
\end{equation}
\end{lemma}
\begin{proof}
The proof is postponed to the appendix.
\end{proof}

\begin{lemma}\label{estimate2}
Under the assumptions {\bf (H1)-(H2)},  we have for any $p\geq2$:  
\begin{equation}
\begin{aligned}
&\bar{\mathbb{E}}\bigg[\sup_{0\leq s\leq T}|Y^{\epsilon}_s-\bar{Y}_s-Y_s^1-Y_s^2|^p\bigg]=o(\epsilon^p),\\
& \bar{\mathbb{E}}\bigg[\sup_{0\leq s\leq T}|\tilde{\rho}^{\epsilon}_s-\bar{\tilde{\rho}}_s-\tilde{\rho}_s^1-\tilde{\rho}_s^2|^p\bigg]=o(\epsilon^p)\,. 
\end{aligned}
\end{equation}
\end{lemma}
\begin{proof}
The proof is given  in the appendix.
\end{proof}
Next, we consider our cost functional under new probability $\bar{\mathbb{P}}$
\begin{equation}
\begin{aligned}
J(u(\cdot))&=\mathbb{E}\bigg[\Phi(\Gamma_T^{-1}Y^u_T)+\int_0^Tf(t,\Gamma_t^{-1}Y^u_t,u_t)dt\bigg]\\
&=\bar{\mathbb{E}}\bigg[\tilde{\rho}^u_T\Phi(\Gamma_T^{-1}Y^u_T)+\int_0^T\tilde{\rho}^u_tf(t,\Gamma_t^{-1}Y^u_t,u_t)dt\bigg]\,,
\end{aligned}
\end{equation}
and define
\begin{equation}\label{hatJ}
\begin{aligned}
\hat{J}:=&\bar{\mathbb{E}}\bigg[\int_0^T\bigg\{\bar{\tilde{\rho}}_t\bigg[\langle \bar{f}_X(t),\Gamma_t^{-1}(Y_t^1+Y_t^2)\rangle+\frac{1}{2}\langle\bar{f}_{XX}(t)\Gamma_t^{-1}Y_t^1,\Gamma_t^{-1}Y_t^1\rangle+\delta f(t)I_{E_{\epsilon}}\bigg]\\
&+(\tilde{\rho}_t^1+\tilde{\rho}_t^2)\bar{f}(t)+\tilde{\rho}_t^1\langle\bar{f}_X(t),\Gamma_t^{-1}Y_t^1\rangle\bigg\}dt\bigg]\\
+&\bar{\mathbb{E}}\bigg[\bar{\tilde{\rho}}_T\bigg[\langle\Phi_X(\Gamma_T^{-1}\bar{Y}_T),\Gamma_T^{-1}(Y_T^1+Y^2_T)\rangle+\frac{1}{2}\langle\Phi_{XX}(\Gamma_T^{-1}\bar{Y}_T)\Gamma_T^{-1}Y_T^1,\Gamma_T^{-1}Y_T^1\rangle\bigg]\\
&+(\tilde{\rho}_T^1+\tilde{\rho}_T^2)\Phi(\Gamma_T^{-1}\bar{Y}_T)+\tilde{\rho}_T^1\langle\Phi_X(\Gamma_T^{-1}\bar{Y}_T),\Gamma_T^{-1}Y_T^1\rangle\bigg]\,. 
\end{aligned}
\end{equation}
\begin{lemma}\label{Jestimate}
Let  the assumptions {\bf (H1)-(H2)} hold. We have
\begin{equation}
J(u^{\epsilon})-J(\bar{u})=\hat{J}+o(\epsilon).
\end{equation}
\end{lemma}
\begin{proof}
The detailed proof is given in the appendix.
\end{proof}

\subsubsection{Adjoint equations and maximum principle}
In this subsection, we introduce the first and second order adjoint equations and use them to 
obtain   our   main result of this work  on   the  maximum principle for   partially observed systems 
\eqref{new state}-\eqref{new cost}.

First, applying  the It\^{o}'s formula to $\bar{\tilde{\rho}}_t(Y_t^1+Y_t^2)$ and then to $\tilde{\rho}^1_tY_t^1$  
 yields,  respectively
{\footnotesize\begin{equation*}
\begin{aligned}
&d\bar{\tilde{\rho}}_t(Y_t^1+Y_t^2)=\bar{\tilde{\rho}}_t\bigg[\Gamma_t\bar{b}_X(t)\Gamma_t^{-1}(Y^1_t+Y^2_t)+\frac{1}{2}\Gamma_t\bar{b}_{XX}(t)\Gamma_t^{-1}Y_t^1\Gamma_t^{-1}Y_t^1+\Gamma_t\delta b(t)I_{E_{\epsilon}}\bigg]dt\\
&\quad+\bar{\tilde{\rho}}_t\sum_{r_1=1}^{k_1}\Big\{\Gamma_t\bar{\sigma}_X^{r_1}(t)\Gamma_t^{-1}(Y_t^1+Y_t^2)+\Gamma_t\delta\sigma^{r_1}(t)I_{E_{\epsilon}}+\frac{1}{2}\Gamma_t\bar{\sigma}_{XX}^{r_1}(t)\Gamma_t^{-1}Y_t^1\Gamma_t^{-1}Y_t^1\\
&\quad+\Gamma_t\delta\sigma_X^{r_1}(t)I_{E_{\epsilon}}\Gamma_t^{-1}Y_t^1\Big\}dW_t^{r_1}+\bar{\tilde{\rho}}_t\sum_{r_2=1}^{k_2}(D^{-1}(t)\bar{h}(t))^{r_2}(Y_t^1+Y_t^2)d\zeta_t^{r_2},
\end{aligned}
\end{equation*}}
and 
{\footnotesize\begin{equation*}
\begin{aligned}
&d\tilde{\rho}^1_tY_t^1=\tilde{\rho}_t^1\Gamma_t\bar{b}_X(t)\Gamma_t^{-1}Y_t^1dt+\tilde{\rho}_t^1\sum_{r_1=1}^{k_1}[\Gamma_t\bar{\sigma}_X^{r_1}(t)\Gamma_t^{-1}Y_t^1+\Gamma_t\delta\sigma^{r_1}(t)I_{E_{\epsilon}}]dW_t^{r_1}\\
&+\sum_{r_2=1}^{k_2}[\tilde{\rho}_t^1Y_t^1(D^{-1}(t)\bar{h}(t))^{r_2}+\bar{\tilde{\rho}}_tY_t^1(D^{-1}(t)\bar{h}_X(t)\Gamma_t^{-1}Y_t^1)^{r_2}+\bar{\tilde{\rho}}_tY_t^1(D^{-1}(t)\delta h(t)I_{E_{\epsilon}})^{r_2}]d\zeta_t^{r_2}\,. 
\end{aligned}
\end{equation*}}
By \eqref{rho12}, we have
{\footnotesize\begin{equation*}
\begin{aligned}
&d(\tilde{\rho}_t^1+\tilde{\rho}_t^2)=\Big\{(\tilde{\rho}_t^1+\tilde{\rho}_t^2)(D^{-1}(t)\bar{h}(t))^{\top}+\bar{\tilde{\rho}}_t[D^{-1}(t)\bar{h}_X(t)\Gamma_t^{-1}(Y_t^1+Y_t^2)]^{\top}\\
&\quad+\bar{\tilde{\rho}}_t[D^{-1}(t)\delta h(t)I_{E_{\epsilon}}]^{\top}+\tilde{\rho}_t^1[D^{-1}(t)\bar{h}_X(t)\Gamma_t^{-1}Y^1_t]^{\top}+\tilde{\rho}_t^1[D^{-1}(t)\delta h(t)I_{E_{\epsilon}}]^{\top}\\
&\quad+\frac{1}{2}\bar{\tilde{\rho}}_t[D^{-1}(t)\bar{h}_{XX}(t)\Gamma_t^{-1}Y_t^1\Gamma_t^{-1}Y_t^1]^{\top}+\bar{\tilde{\rho}}_t[D^{-1}(t)\delta h_X(t)I_{E_{\epsilon}}\Gamma_t^{-1}Y_t^1]^{\top}\Big\}d\zeta_t\,. 
\end{aligned}
\end{equation*}}
Setting  $\tilde{\YY}_t=Y_t^1(Y_t^1)^{\top}$,  we  see 
{\footnotesize\begin{equation*}
\begin{aligned}
 d\tilde{\YY}_t =&\bigg\{\tilde{\YY}_t(\Gamma_t^{-1})^{\top}\bar{b}_X^{\top}(t)\Gamma_t^{\top}+\Gamma_t\bar{b}_X(t)\Gamma_t^{-1}\tilde{\YY}_t+\sum_{r_1=1}^{k_1}\Big[\Gamma_t\bar{\sigma}_X^{r_1}(t)\Gamma_t^{-1}\tilde{\YY}_t(\Gamma_t^{-1})^{\top}(\bar{\sigma}_X^{r_1}(t))^{\top}\Gamma_t^{\top}\\
&+\Gamma_t\bar{\sigma}_X^{r_1}(t)\Gamma_t^{-1}Y_t^1(\delta\sigma^{r_1}(t))^{\top}\Gamma_t^{\top}I_{E_{\epsilon}}+\Gamma_t\delta\sigma^{r_1}(t)(Y_t^1)^{\top}(\Gamma^{-1}_t)^{\top}(\bar{\sigma}_X^{r_1}(t))^{\top}\Gamma_t^{\top}I_{E_{\epsilon}}\\  
&+\Gamma_t\delta\sigma^{r_1}(t)(\delta\sigma^{r_1}(t))^{\top}\Gamma_t^{\top}I_{E_{\epsilon}}\Big]\bigg\}dt+\sum_{r_1=1}^{k_1}\Big\{\tilde{\YY}_t(\Gamma_t^{-1})^{\top}(\bar{\sigma}_X^{r_1}(t))^{\top}\Gamma_t^{\top}\\
&+Y_t^1(\delta\sigma^{r_1}(t))^{\top}\Gamma_t^{\top}I_{E_{\epsilon}}+\Gamma_t\bar{\sigma}_X^{r_1}(t)\Gamma_t^{-1}\tilde{\YY}_t+\Gamma_t\delta\sigma^{r_1}(t)(Y_t^1)^{\top}I_{E_{\epsilon}}\Big\}dW_t^{r_1}\,. 
\end{aligned}
\end{equation*}}
Now we introduce the first order adjoint equations: 
\begin{equation}\label{alpha}
\left\{
\begin{aligned}
-d\alpha_t&=\bigg[\bar{f}(t)+\sum_{r_2=1}^{k_2}(D^{-1}(t)\bar{h}(t))^{r_2}\beta_t^{r_2}\bigg]dt-\sum_{r_2=1}^{k_2}\beta_t^{r_2}d\zeta_t^{r_2}-dN_t\,, \\
\alpha_T&=\Phi(\Gamma_T^{-1}\bar{Y}_T)\,,
\end{aligned}
\right.
\end{equation}
\begin{equation}\label{p}
\left\{
\begin{aligned}
-dp_t&=\bigg\{(\Gamma_t^{-1})^{\top}\bar{f}_X(t)+(\Gamma_t^{-1})^{\top}\bar{b}_X^{\top}(t)\Gamma_t^{\top}p_t+\sum_{r_2=1}^{k_2}(D^{-1}(t)\bar{h}(t))^{r_2}\tilde{q}_t^{r_2}\\
&+\sum_{r_1=1}^{k_1}(\Gamma_t^{-1})^{\top}\bar{\sigma}_X^{r_1}(t)\Gamma_t^{\top}q^{r_1}_t+\sum_{r_2=1}^{k_2}[(\Gamma_t^{-1})^{\top}\bar{h}_X^{\top}(t)(D^{-1}(t))^{\top}]^{r_2}\beta_t^{r_2}\bigg\}dt\\
&-\sum_{r_1=1}^{k_1}q_t^{r_1}dW_t^{r_1}-\sum_{r_2=1}^{k_2}\tilde{q}^{r_2}_td\zeta_t^{r_2}-dM_t\\
p_T&=(\Gamma_T^{-1})^{\top}\Phi_X(\Gamma_T^{-1}\bar{Y}_T),
\end{aligned}
\right.
\end{equation}
and  the second order adjoint equation: 
\begin{equation}\label{P}
\left\{
\begin{aligned}
-dP_t&=\bigg\{\bar{\tilde{\rho}}_t(\Gamma_t^{-1})^{\top}\bar{f}_{XX}(t)\Gamma_t^{-1}+(\Gamma_t^{-1})^{\top}\bar{b}_X^{\top}(t)\Gamma_t^{\top}P_t+P_t\Gamma_t\bar{b}_X(t)\Gamma_t^{-1}\\
&+\sum_{r_1=1}^{k_1}(\Gamma_t^{-1})^{\top}(\bar{\sigma}_X^{r_1}(t))^{\top}\Gamma_t^{\top}P_t\Gamma_t\bar{\sigma}_X^{r_1}(t)\Gamma_t^{-1}+\sum_{r_1=1}^{k_1}[(\Gamma_t^{-1})^{\top}(\bar{\sigma}_X^{r_1}(t))^{\top}\Gamma_t^{\top}Q_t^{r_1}\\
&+Q_t^{r_1}\Gamma_t\bar{\sigma}_X^{r_1}(t)\Gamma_t^{-1}]+\sum_{r_2=1}^{k_2}\bar{\tilde{\rho}}_t(\Gamma_t^{-1})^{\top}(D^{-1}(t)\bar{h}_{XX}(t))^{r_2}\Gamma_t^{-1}\beta_t^{r_2}\\
&+\sum_{r_1=1}^{k_1}\langle\bar{\tilde{\rho}}_t(\Gamma_t^{-1})^{\top}\Gamma_t\bar{\sigma}_{XX}^{r_1}(t)\Gamma_t^{-1},q^{r_1}_t\rangle+\sum_{r_2=1}^{k_2}\Big[\bar{\tilde{\rho}}_t\tilde{q}^{r_2}_t([(\Gamma_t^{-1})^{\top}\bar{h}_X^{\top}(t)(D^{-1}(t))^{\top}]^{r_2})^{\top}\\
&+\bar{\tilde{\rho}}_t[(\Gamma_t^{-1})^{\top}\bar{h}_X^{\top}(t)(D^{-1}(t))^{\top}]^{r_2}(\tilde{q}_t^{r_2})^{\top}\Big]+\langle\bar{\tilde{\rho}}_t(\Gamma_t^{-1})^{\top}\Gamma_t\bar{b}_{XX}(t)\Gamma_t^{-1},p_t\rangle\bigg\}dt\\
&-\sum_{r_1=1}^{k_1}Q_t^{r_1}dW_t^{r_1}-dR_t\\
P_T&=\bar{\tilde{\rho}}_T(\Gamma_T^{-1})^{\top}\Phi_{XX}(\Gamma_T^{-1}\bar{Y}_T)\Gamma_T^{-1},
\end{aligned}
\right.
\end{equation}
where $N,M,R$ are  square-integrable $\FF$-martingales  starting $0$  at time $0$, which are orthogonal to the Brownian motions $W,\zeta$. 
Let us point out that similar to  \cite{BJ14}, in our  backward stochastic differential equations
\eqref{alpha}, \eqref{p}, \eqref{P}, we need the
three additional martingales $(N_t, M_t, R_t)$.  

Applying classical It\^{o}'s formula to the It\^o processes $\langle\bar{\tilde{\rho}}_t(Y_t^1+Y_t^2),p_t\rangle$, $\langle \tilde{\rho}_t^1Y_t^1,p_t\rangle$, $(\tilde{\rho}_t^1+\tilde{\rho}_t^2)\alpha_t$ and $\langle P_t,\tilde{\YY}_t\rangle$ on 
the interval $[0,T]$,  and  then substituting  them into the \eqref{hatJ},  and noticing 
\begin{equation}
\begin{aligned}
&\bar{\mathbb{E}}\int_0^T\bigg\{\sum_{r_1=1}^{k_1}\langle\bar{\tilde{\rho}}_t\Gamma_t\delta\sigma_X^{r_1}(t)\Gamma_t^{-1}Y_t^1I_{E_{\epsilon}},q_t^{r_1}\rangle+\frac{1}{2}\text{tr}\bigg[P_t\bigg[\sum_{r_1=1}^{k_1}[\Gamma_t\bar{\sigma}_X^{r_1}(t)\Gamma_t^{-1}Y_t^1(\delta\sigma^{r_1}(t))^{\top}\Gamma_t^{\top}I_{E_{\epsilon}}\\
&+\Gamma_t\delta\sigma^{r_1}(t)(Y_t^1)^{\top}(\Gamma_t^{-1})^{\top}(\bar{\sigma}_X^{r_1}(t))^{\top}\Gamma_t^{\top}I_{E_{\epsilon}}]\bigg]+\sum_{r_1=1}^{k_1}Q_t^{r_1}[Y_t^1(\delta\sigma^{r_1}(t))^{\top}\Gamma_t^{\top}I_{E_{\epsilon}}\\
&+\Gamma_t\delta\sigma^{r_1}(t)(Y_t^1)^{\top}I_{E_{\epsilon}}]\bigg]+\sum_{r_2=1}^{k_2}[\tilde{\rho}_t^1(D^{-1}(t)\delta h(t)I_{E_{\epsilon}})^{r_2}+\bar{\tilde{\rho}}_t(D^{-1}(t)\delta h_X(t)\Gamma_t^{-1}Y_t^1I_{E_{\epsilon}})^{r_2}]\beta_t^{r_2}\\
&+\sum_{r_1=1}^{k_1}\langle\tilde{\rho}_t^1\Gamma_t\delta\sigma^{r_1}(t)I_{E_{\epsilon}},q_t^{r_1}\rangle+\sum_{r_2=1}^{k_2}\langle\bar{\tilde{\rho}}_tY_t^1(D^{-1}(t)\delta h(t)I_{E_{\epsilon}})^{r_2},\tilde{q}_t^{r_2}\rangle\bigg\}dt=o(\epsilon)\,,  
\end{aligned}\label{e.6.21} 
\end{equation}
we  obtain 
\begin{equation}
\begin{aligned}
\hat{J}&=\bar{\mathbb{E}}\int_0^T\bigg\{\bar{\tilde{\rho}}_t\delta f(t)+\langle\bar{\tilde{\rho}}_t\Gamma_t\delta b(t),p_t\rangle+\sum_{r_1=1}^{k_1}\langle\bar{\tilde{\rho}}_t[\Gamma_t\delta\sigma^{r_1}(t)+\Gamma_t\delta\sigma_X^{r_1}(t)\Gamma_t^{-1}Y_t^1],q_t^{r_1}\rangle\\
&+\frac{1}{2}\text{tr}\bigg[P_t\bigg[\sum_{r_1=1}^{k_1}[\Gamma_t\bar{\sigma}_X^{r_1}(t)\Gamma_t^{-1}Y_t^1(\delta\sigma^{r_1}(t))^{\top}\Gamma_t^{\top}+\Gamma_t\delta\sigma^{r_1}(t)(Y_t^1)^{\top}(\Gamma_t^{-1})^{\top}(\bar{\sigma}_X^{r_1}(t))^{\top}\Gamma_t^{\top}\\
&+\Gamma_t\delta\sigma^{r_1}(t)(\delta\sigma^{r_1}(t))^{\top}\Gamma_t^{\top}]\bigg]+\sum_{r_1=1}^{k_1}Q_t^{r_1}[Y_t^1(\delta\sigma^{r_1}(t))^{\top}\Gamma_t^{\top}+\Gamma_t\delta\sigma^{r_1}(t)(Y_t^1)^{\top}]\bigg]\\
&+\sum_{r_2=1}^{k_2}[\bar{\tilde{\rho}}_t(D^{-1}(t)\delta h(t))^{r_2}+\tilde{\rho}_t^1(D^{-1}(t)\delta h(t))^{r_2}+\bar{\tilde{\rho}}_t(D^{-1}(t)\delta h_X(t)\Gamma_t^{-1}Y_t^1)^{r_2}]\beta_t^{r_2}\\
&+\sum_{r_1=1}^{k_1}\langle\tilde{\rho}_t^1\Gamma_t\delta\sigma^{r_1}(t),q_t^{r_1}\rangle+\sum_{r_2=1}^{k_2}\langle\bar{\tilde{\rho}}_tY_t^1(D^{-1}(t)\delta h(t))^{r_2},\tilde{q}_t^{r_2}\rangle\bigg\}I_{E_{\epsilon}}dt\,. 
\end{aligned}\label{e.6.22} 
\end{equation}
Denote  
\begin{equation}
\begin{aligned}
H(t,Y,u;p,q,\beta)&=\text{tr}[q_t^{\top}\Gamma_t\sigma(t,\Gamma^{-1}_tY_t,u_t)]+\langle\beta_t,D^{-1}(t)h(t,\Gamma_t^{-1}Y_t,u_t)\rangle\\
&\quad+\langle p_t,\Gamma_tb(t,\Gamma_t^{-1}Y_t,u_t)\rangle+f(t,\Gamma_t^{-1}Y_t,u_t),
\end{aligned}
\end{equation}
and
\begin{equation}
\begin{aligned}
\delta H(t)=H(t,\bar{Y},u;p,q,\beta)-H(t,\bar{Y},\bar{u};p,q,\beta).
\end{aligned}
\end{equation}
Then we  can write \eqref{e.6.22} as 
\begin{equation}
\begin{aligned}
\hat{J}&=\bar{\mathbb{E}}\int_0^T\bigg\{\bar{\tilde{\rho}}_t\delta H(t)+\frac{1}{2}\text{tr}\{(\delta\sigma(t))^{\top}\Gamma_t^{\top}P_t\Gamma_t\delta\sigma(t)\}\bigg\}I_{E_{\epsilon}}dt+o(\epsilon)\,. 
\end{aligned}
\end{equation} 
Now we can summarize the above argument as   the main theorem in our paper. 
\begin{theorem}
Let {\bf (H1)-(H2)}  hold and let  $\bar{u}$ be  the optimal control.   Assume that $\bar{Y}$ and $\bar{\tilde{\rho}}$ are the 
solutions to \eqref{new state} and \eqref{tilderho} corresponding to  $\bar{u}$. 
Suppose that $(\alpha,\beta)\in\mathcal{S}^{p}([0,T];\mathbb{R})\times (\mathcal{M}^{2,p}([0,T];\mathbb{R}))^{k_2}$, $(p,q,\tilde{q})\in\mathcal{S}^{p}([0,T];\mathbb{R}^n)\times (\mathcal{M}^{2,p}([0,T];\mathbb{R}^n))^{k_1}\times(\mathcal{M}^{2,p}([0,T];\mathbb{R}^n))^{k_2}$ and $(P,Q)\in\mathcal{S}^{p}([0,T];\mathbb{R}^{n\times n})\times (\mathcal{M}^{2,p}([0,T];\mathbb{R}^{n\times n}))^{k_1}$ satisfy \eqref{alpha}, \eqref{p} and \eqref{P}, respectively. Then for any $u\in U$, we have the following condition for   the optimal  control.  
\begin{equation}
\begin{aligned}
&\bar{\mathbb{E}}\bigg[\bar{\tilde{\rho}}_t\delta H(t)+\frac{1}{2}\text{tr}\{(\delta\sigma(t))^{\top}\Gamma_t^{\top}P_t\Gamma_t\delta\sigma(t)\}\bigg|\mathcal{F}_t^{\zeta}\bigg]\geq0,\ a.e.\ t,\ \bar{\mathbb{P}}\text{-}a.s. 
\end{aligned}
\end{equation} 
\end{theorem}
%

\subsection{The maximum principle in the case of $H\in(1/3,1/2)$}
 In this subsection, we follow the same idea 
in the  previous section to obtain the maximum principle for the case  $H\in (1/3, 1/2)$.
We need some results analogous to Lemmas \ref{Ytilderhoestimate}-\ref{Jestimate}. However, we can no longer use Theorem  \ref{solution estimate2}. 
Instead we have to use Theorem  \ref{solution estimate1}. 
We shall state these corresponding lemmas, whose proofs are similar and are omitted. 

\begin{lemma}\label{Ytilderhoestimate2}
Under the assumptions  {\bf (H1)-(H2)}, for any $p\geq2$   and  for $\bar{\mathbb{P}}\text{-}a.s.\ \omega_2\in\Omega_2$, we have the following estimates:
\begin{equation}
\begin{aligned}
&\bar{\mathbb{E}}\bigg[\sup_{0\leq s\leq T}|Y^1_s|^p\bigg]\leq C\epsilon^{\frac{p}{2}},\ \bar{\mathbb{E}}\bigg[\sup_{0\leq s\leq T}|Y^2_s|^p\bigg]\leq C\epsilon^p,\\
& \bar{\mathbb{E}}\bigg[\sup_{0\leq s\leq T}|\tilde{\rho}^1_s|^p\bigg]\leq C\epsilon^{\frac{p}{2}},\ \bar{\mathbb{E}}\bigg[\sup_{0\leq s\leq T}|\tilde{\rho}^2_s|^p\bigg]\leq C\epsilon^p,
\end{aligned}
\end{equation}
where expectation $\bar{\mathbb{E}}=\bar{\mathbb{E}}^{\omega_1}$, which means that it only takes expectation for $\omega_1$ of Brownian motions, for any fixed $\omega_2$.
\end{lemma}

\begin{lemma}\label{estimate22}
Under the assumptions  {\bf (H1)-(H2)}, for $p\geq2$ and $\bar{\mathbb{P}}\text{-}a.s.\ \omega_2\in\Omega_2$, we have
\begin{equation}
\begin{aligned}
\bar{\mathbb{E}}\bigg[\sup_{0\leq s\leq T}|Y^{\epsilon}_s-\bar{Y}_s-Y_s^1-Y_s^2|^p\bigg]=o(\epsilon^p),\\
\bar{\mathbb{E}}\bigg[\sup_{0\leq s\leq T}|\tilde{\rho}^{\epsilon}_s-\bar{\tilde{\rho}}_s-\tilde{\rho}_s^1-\tilde{\rho}_s^2|^p\bigg]=o(\epsilon^p)\,. 
\end{aligned}
\end{equation}
\end{lemma}

\begin{lemma}\label{Jestimate2}
Let the assumptions  {\bf (H1)-(H2)} hold. Then for $\bar{\mathbb{P}}\text{-}a.s.\ \omega_2\in\Omega_2$, we have
\begin{equation}
J(u^{\epsilon})-J(\bar{u})=\hat{J}+o(\epsilon).
\end{equation}
\end{lemma}

%

\begin{theorem}
Let {\bf (H1)-(H2)} hold and let $\bar{u}$ be the optimal control. Assume that $\bar{Y}$ and $\bar{\tilde{\rho}}$ are the solutions to \eqref{new state} and \eqref{tilderho} corresponding to $\bar{u}$. Suppose that $(\alpha,\beta)\in\mathcal{S}^{p}([0,T];\mathbb{R})\times (\mathcal{M}^{2,p}([0,T];\mathbb{R}))^{k_2}$, $(p,q,\tilde{q})\in\mathcal{S}^{p}([0,T];\mathbb{R}^n)\times (\mathcal{M}^{2,p}([0,T];\mathbb{R}^n))^{k_1}\times(\mathcal{M}^{2,p}([0,T];\mathbb{R}^n))^{k_2}$ and $(P,Q)\in\mathcal{S}^{p}([0,T];\mathbb{R}^{n\times n})\times (\mathcal{M}^{2,p}([0,T];\mathbb{R}^{n\times n}))^{k_1}$ satisfy \eqref{alpha}, \eqref{p} and \eqref{P}, respectively. Then for any $u\in U$, and for $\bar{\mathbb{P}}\text{-}a.s.\ \omega_2$, we have
\begin{equation}
\begin{aligned}
&\bar{\mathbb{E}}\bigg[\bar{\tilde{\rho}}_t\delta H(t)+\frac{1}{2}\text{tr}\{(\delta\sigma(t))^{\top}\Gamma_t^{\top}P_t\Gamma_t\delta\sigma(t)\}\bigg|\mathcal{F}_t^{\zeta}\bigg]\geq0,\ a.e.\ t,\ \bar{\mathbb{P}}\text{-}a.s.\ \omega_1.
\end{aligned}
\end{equation}
\end{theorem}

\section{Appendix} 
In the appendix, we give  proofs of the results in Section 7.

\emph{Proof of Lemma \ref{Ytilderhoestimate}}.
From  the classical $L^p$-estimate of SDE, H\"older's inequality and assumption {\bf (H2)}, it follows 
\begin{equation*}
\begin{aligned}
&\bar{\mathbb{E}}\bigg[\sup_{0\leq s\leq T}|Y_s^1|^p\bigg]\leq C\bar{\mathbb{E}}\bigg[\sum_{r_1=1}^{k_1}\bigg(\int_0^T\Big|\Gamma_s\delta\sigma^{r_1}(s)I_{E_{\epsilon}}\Big|^2ds\bigg)^{\frac{p}{2}}\bigg]\\
&\leq C\bar{\mathbb{E}}\bigg[\sum_{r_1=1}^{k_1}\sup_{0\leq s\leq T}|\Gamma_s|^p\bigg(\int_0^T\Big|\delta\sigma^{r_1}(s)\Big|^2I_{E_{\epsilon}}ds\bigg)^{\frac{p}{2}}\bigg]\\
&\leq C\bigg(\bar{\mathbb{E}}\sup_{0\leq s\leq T}|\Gamma_s|^{2p}\bigg)^{\frac{1}{2}}\bigg(\bar{\mathbb{E}}\bigg[\sum_{r_1=1}^{k_1}\bigg(\int_0^T\Big|\delta\sigma^{r_1}(s)\Big|^2I_{E_{\epsilon}}ds\bigg)^{\frac{p}{2}}\bigg]^2\bigg)^{\frac{1}{2}}\\
&\leq C\bigg(\bar{\mathbb{E}}\sum_{r_1=1}^{k_1}\bigg(\int_0^T\Big(1+|\Gamma_s^{-1}\bar{Y}_s|^2+|u_s|^2+|\bar{u}_s|^2\Big)I_{E_{\epsilon}}ds\bigg)^p\bigg)^{\frac{1}{2}}\\
&\leq C\bigg(\bar{\mathbb{E}}\bigg[\epsilon^p+\bigg(\sup_{0\leq s\leq T}|\Gamma_s^{-1}\bar{Y}_s|^2\int_0^TI_{E_{\epsilon}}ds\bigg)^p+\bigg(\int_0^T(|u_s|^2+|\bar{u}_s|^2)I_{E_{\epsilon}}ds\bigg)^p\bigg]\bigg)^{\frac{1}{2}}\\
&\leq C\bigg(\epsilon^p+\epsilon^{p-1}\bar{\mathbb{E}}\bigg(\int_0^T(|u_s|^{2p}+|\bar{u}_s|^{2p})I_{E_{\epsilon}}ds\bigg)\bigg)^{\frac{1}{2}}\leq O(\epsilon^{\frac{p}{2}})\\
\end{aligned}
\end{equation*}
Similarly, we have by H\"older's inequality and assumption {\bf (H2)}
\begin{equation*}
\begin{aligned}
&\bar{\mathbb{E}}\bigg[\sup_{0\leq s\leq T}|Y_s^2|^p\bigg]\leq C\bar{\mathbb{E}}\bigg[\bigg(\int_0^T\bigg|\frac{1}{2}\Gamma_s\bar{b}_{XX}(s)\Gamma_s^{-1}Y_s^1\Gamma_s^{-1}Y_s^1+\Gamma_s\delta b(s)I_{E_{\epsilon}}\bigg|ds\bigg)^p\\
&+\sum_{r_1=1}^{k_1}\bigg(\int_0^T\bigg|\frac{1}{2}\Gamma_s\bar{\sigma}^{r_1}_{XX}(s)\Gamma_s^{-1}Y_s^1\Gamma_s^{-1}Y_s^1+\Gamma_s\delta\sigma_X^{r_1}(s)I_{E_{\epsilon}}\Gamma_s^{-1}Y_s^1\bigg|^2ds\bigg)^{\frac{p}{2}}\bigg]\\
&\leq C\bar{\mathbb{E}}\bigg[\bigg(\frac{1}{2}\int_0^T\bigg|\Gamma_s\bar{b}_{XX}(s)\Gamma_s^{-1}Y_s^1\Gamma_s^{-1}Y_s^1\bigg|ds\bigg)^p+\bigg(\int_0^T\bigg|\Gamma_s\delta b(s)I_{E_{\epsilon}}\bigg|ds\bigg)^p\\
&+\sum_{r_1=1}^{k_1}\bigg(\int_0^T\bigg|\frac{1}{2}\Gamma_s\bar{\sigma}^{r_1}_{XX}(s)\Gamma_s^{-1}Y_s^1\Gamma_s^{-1}Y_s^1\bigg|^2ds\bigg)^{\frac{p}{2}}+\sum_{r_1=1}^{k_1}\bigg(\int_0^T\bigg|\Gamma_s\delta\sigma_X^{r_1}(s)I_{E_{\epsilon}}\Gamma_s^{-1}Y_s^1\bigg|^2ds\bigg)^{\frac{p}{2}}\bigg]\\
&\leq C\bar{\mathbb{E}}\bigg[\sup_{0\leq s\leq T}|\Gamma_s|^p\sup_{0\leq s\leq T}|\Gamma_s^{-1}Y_s^1|^{2p}\bigg]+C\bar{\mathbb{E}}\bigg[\bigg(\int_0^T|\Gamma_s\delta b(s)|I_{E_{\epsilon}}ds\bigg)^p\bigg]\\
&+C\sum_{r_1=1}^{k_1}\bigg[\bar{\mathbb{E}}\bigg(\int_0^T\bigg|\Gamma_s\bar{\sigma}^{r_1}_{XX}(s)\Gamma_s^{-1}Y_s^1\Gamma_s^{-1}Y_s^1\bigg|^2ds\bigg)^{\frac{p}{2}}+\bar{\mathbb{E}}\bigg(\int_0^T\bigg|\Gamma_s\delta\sigma_X^{r_1}(s)\Gamma_s^{-1}Y_s^1I_{E_{\epsilon}}\bigg|^2ds\bigg)^{\frac{p}{2}}\bigg]\\
&\leq C\bigg(\bar{\mathbb{E}}\sup_{0\leq s\leq T}|\Gamma_s^{-1}|^{4p}|Y_s^1|^{4p}\bigg)^{\frac{1}{2}}+C\bar{\mathbb{E}}\bigg[\bigg(\int_0^T|\Gamma_s|(1+|\Gamma_s^{-1}\bar{Y}_s|+|u_s|+|\bar{u}_s|)I_{E_{\epsilon}}ds\bigg)^p\bigg]\\
&+C\sum_{r_1=1}^{k_1}\bigg[\bar{\mathbb{E}}\bigg(\sup_{0\leq s\leq T}|\Gamma_s\bar{\sigma}^{r_1}_{XX}(s)|^p\sup_{0\leq s\leq T}|\Gamma_s^{-1}Y_s^1|^{2p}\bigg)+\bar{\mathbb{E}}\bigg(\int_0^T|\Gamma_s|^2|\Gamma_s^{-1}Y_s^1|^2I_{E_{\epsilon}}ds\bigg)^{\frac{p}{2}}\bigg]\\
&\leq C\epsilon^p+C\bar{\mathbb{E}}\bigg[\sup_{0\leq s\leq T}|\Gamma_s|^p\bigg(\int_0^T(1+|\Gamma_s^{-1}\bar{Y}_s|+|u_s|+|\bar{u}_s|)I_{E_{\epsilon}}ds\bigg)^p\bigg]\\
&+C\bigg(\bar{\mathbb{E}}\sup_{0\leq s\leq T}|\Gamma_s|^{2p}\bigg)^{\frac{1}{2}}\bigg(\bar{\mathbb{E}}\sup_{0\leq s\leq T}|\Gamma_s^{-1}Y_s^1|^{4p}\bigg)^{\frac{1}{2}}+C\epsilon^{\frac{p}{2}}\bar{\mathbb{E}}\bigg[\sup_{0\leq s\leq T}|\Gamma_s|^p\sup_{0\leq s\leq T}|\Gamma_s^{-1}Y_s^1|^p\bigg]\\
\end{aligned}
\end{equation*}
\begin{equation*}
\begin{aligned}
&\leq C\epsilon^p+C\bigg(\bar{\mathbb{E}}\sup_{0\leq s\leq T}|\Gamma_s|^{2p}\bigg)^{\frac{1}{2}}\bigg(\bar{\mathbb{E}}\bigg(\int_0^T(1+|\Gamma_s^{-1}\bar{Y}_s|+|u_s|+|\bar{u}_s|)I_{E_{\epsilon}}ds\bigg)^{2p}\bigg)^{\frac{1}{2}}\\
&\leq C\epsilon^p+C\bigg(\bar{\mathbb{E}}\bigg(\int_0^T(1+|\Gamma_s^{-1}\bar{Y}_s|^{2p}+|u_s|^{2p}+|\bar{u}_s|^{2p})I_{E_{\epsilon}}ds\bigg)\epsilon^{2p-1}\bigg)^{\frac{1}{2}}\\
&\leq C\epsilon^p+C\bigg(\epsilon^{2p-1}\bigg(\epsilon+\epsilon\bar{\mathbb{E}}\sup_{0\leq s\leq T}|\Gamma_s^{-1}\bar{Y}_s|^{2p}+\epsilon\sup_{0\leq s\leq T}\bar{\mathbb{E}}(|u_s|^{2p}+|\bar{u}_s|^{2p})\bigg)\bigg)^{\frac{1}{2}}\leq O(\epsilon^p).
\end{aligned}
\end{equation*}
Now we give   the last two estimates in \eqref{estimate1}. First, from the classical $L^p$-estimate of solution to SDE, we have
\begin{equation*}
\begin{aligned}
\bar{\mathbb{E}}\bigg[\sup_{0\leq s\leq T}|\tilde{\rho}_s^1|^p\bigg]&\leq C\bar{\mathbb{E}}\bigg[\bigg(\int_0^T|\bar{\tilde{\rho}}_s(D^{-1}(s)\bar{h}_X(s)\Gamma_s^{-1}Y_s^1)^{\top}+\bar{\tilde{\rho}}_s(D^{-1}(s)\delta h(s)I_{E_{\epsilon}})^{\top}|^2ds\bigg)^{\frac{p}{2}}\bigg]\\
&\leq C\bar{\mathbb{E}}\bigg[\bigg(\int_0^T|\bar{\tilde{\rho}}_s(D^{-1}(s)\bar{h}_X(s)\Gamma_s^{-1}Y_s^1)^{\top}|^2ds\bigg)^{\frac{p}{2}}\bigg]\\
&\quad+C\bar{\mathbb{E}}\bigg[\bigg(\int_0^T|\bar{\tilde{\rho}}_s(D^{-1}(s)\delta h(s)I_{E_{\epsilon}})^{\top}|^2ds\bigg)^{\frac{p}{2}}\bigg]=\mathbb{I}+\mathbb{II},
\end{aligned}
\end{equation*}
where, by H\"older's inequality and assumption {\bf (H2)}

\begin{equation*}
\begin{aligned}
\mathbb{I}&\leq \bar{\mathbb{E}}\bigg[\bigg(\sup_{0\leq s\leq T}|\bar{\tilde{\rho}}_t|^2\int_0^T|D^{-1}(s)|^2|\bar{h}_X(s)|^2|\Gamma_s^{-1}Y_s^1|^2ds\bigg)^{\frac{p}{2}}\bigg]\\ 
&\leq \bar{\mathbb{E}}\bigg[\sup_{0\leq s\leq T}|\bar{\tilde{\rho}}_t|^p\sup_{0\leq s\leq T}|\Gamma_s^{-1}|^p|Y_s^1|^p\bigg(\int_0^T|D^{-1}(s)|^2|\bar{h}_X(s)|^2ds\bigg)^{\frac{p}{2}}\bigg]\\
&\leq C\bigg(\bar{\mathbb{E}}\sup_{0\leq s\leq T}|\bar{\tilde{\rho}}_t|^{2p}\bigg)^{\frac{1}{2}}\bigg(\bar{\mathbb{E}}\sup_{0\leq s\leq T}|\Gamma_s^{-1}|^{2p}|Y_s^1|^{2p}\bigg)^{\frac{1}{2}}\leq C\epsilon^{\frac{p}{2}},
\end{aligned}
\end{equation*}
and 
\begin{equation*}
\begin{aligned}
\mathbb{II}&\leq \bar{\mathbb{E}}\bigg[\bigg(\int_0^T|\bar{\tilde{\rho}}_s|^4|D^{-1}(s)|^4I_{E_{\epsilon}}ds\bigg)^{\frac{p}{4}}\bigg(\int_0^T|\delta h(s)|^4I_{E_{\epsilon}}ds\bigg)^{\frac{p}{4}}\bigg]\\
&\leq C\bar{\mathbb{E}}\bigg[\sup_{0\leq s\leq T}|\bar{\tilde{\rho}}_s|^p\bigg(\int_0^T|D^{-1}(s)|^4I_{E_{\epsilon}}ds\bigg)^{\frac{p}{4}}\bigg(\int_0^TI_{E_{\epsilon}}ds\bigg)^{\frac{p}{4}}\bigg]\\
&\leq C\epsilon^{\frac{p}{2}}\bar{\mathbb{E}}\bigg[\sup_{0\leq s\leq T}|\bar{\tilde{\rho}}_s|^p\bigg]\leq C\epsilon^{\frac{p}{2}}\,. 
\end{aligned}
\end{equation*}
It is also  easy to see 
\begin{equation*}
\begin{aligned}
&\bar{\mathbb{E}}\bigg[\sup_{0\leq s\leq T}|\tilde{\rho}_s^2|^p\bigg]\leq C\bar{\mathbb{E}}\bigg[\bigg(\int_0^T|\tilde{\rho}_s^1(D^{-1}(s)\bar{h}_X(s)\Gamma^{-1}_sY^1_s)^{\top}+\tilde{\rho}_s^1(D^{-1}(s)\delta h(s)I_{E_{\epsilon}})^{\top}\\
&\quad+\bar{\tilde{\rho}}_s(D^{-1}(s)\bar{h}_X(s)\Gamma^{-1}_sY^2_s)^{\top}+\frac{1}{2}\bar{\tilde{\rho}}_s(D^{-1}(s)\bar{h}_{XX}(s)\Gamma_s^{-1}Y_s^1\Gamma_s^{-1}Y_s^1)^{\top}\\
&\quad+\bar{\tilde{\rho}}_s(D^{-1}(s)\delta h_X(s)\Gamma_s^{-1}Y^1_sI_{E_{\epsilon}})^{\top}|^2ds\bigg)^{\frac{p}{2}}\bigg]\leq \text{I}+\text{II}+\text{III}+\text{IV}+\text{V}\,. 
\end{aligned}
\end{equation*}
In the following we are going to bound $\text{I},\text{II},\text{III},\text{IV}$ and $\text{V}$ by H\"older's inequality and assumption {\bf (H2)}. 
We can bound $\text{I}$ as follows. 
\begin{equation*}
\begin{aligned}
\text{I}&=\bar{\mathbb{E}}\bigg[\bigg(\int_0^T|\tilde{\rho}_s^1(D^{-1}(s)\bar{h}_X(s)\Gamma_s^{-1}Y_s^1)^{\top}|^2ds\bigg)^{\frac{p}{2}}\bigg]\\
&\leq\bar{\mathbb{E}}\bigg[\sup_{0\leq s\leq T}|\tilde{\rho}_s^1|^p\sup_{0\leq s\leq T}|\Gamma_s^{-1}Y_s^1|^p\bigg(\int_0^T|D^{-1}(s)\bar{h}_X(s)|^2ds\bigg)^{\frac{p}{2}}\bigg]\\
&\leq C\bigg(\bar{\mathbb{E}}\sup_{0\leq s\leq T}|\tilde{\rho}_s^1|^{2p}\bigg)^{\frac{1}{2}}\bigg(\bar{\mathbb{E}}\sup_{0\leq s\leq T}|\Gamma_s^{-1}Y_s^1|^{2p}\bigg)^{\frac{1}{2}}\leq C\epsilon^p\,. 
\end{aligned}
\end{equation*}
The term $\text{II}$ can be bounded similarly. 
\begin{equation*} 
\begin{aligned}
\text{II}&=\bar{\mathbb{E}}\bigg[\bigg(\int_0^T|\tilde{\rho}_s^1(D^{-1}(s)\delta h(s)I_{E_{\epsilon}})^{\top}|^2ds\bigg)^{\frac{p}{2}}\bigg]\\
&\leq\bar{\mathbb{E}}\bigg[\bigg(\int_0^T|\tilde{\rho}_s^1|^4|D^{-1}(s)|^4I_{E_{\epsilon}}ds\bigg)^{\frac{p}{4}}\bigg(\int_0^T|\delta h(s)|^4I_{E_{\epsilon}}ds\bigg)^{\frac{p}{4}}\bigg]\\
&\leq C\epsilon^{\frac{p}{2}}\bar{\mathbb{E}}\bigg[\sup_{0\leq s\leq T}|\tilde{\rho}_s^1|^p\bigg]\leq C\epsilon^p\,. 
\end{aligned}
\end{equation*}
We can compute $\text{III}$ and $\text{IV}$ as follows. 
\begin{equation*}
\begin{aligned}
\text{III}&=\bar{\mathbb{E}}\bigg[\bigg(\int_0^T|\bar{\tilde{\rho}}_s(D^{-1}(s)\bar{h}_X(s)\Gamma_s^{-1}Y^2_s)^{\top}|^2ds\bigg)^{\frac{p}{2}}\bigg]\\
&\leq C\bigg(\bar{\mathbb{E}}\sup_{0\leq s\leq T}|\bar{\tilde{\rho}}_s|^{2p}\bigg)^{\frac{1}{2}}\bigg(\bar{\mathbb{E}}\sup_{0\leq s\leq T}|\Gamma_s^{-1}|^{4p}\bigg)^{\frac{1}{4}}\bigg(\bar{\mathbb{E}}\sup_{0\leq s\leq T}|Y^2_s|^{4p}\bigg)^{\frac{1}{4}}\leq C\epsilon^p,
\end{aligned}
\end{equation*}
and 
\begin{equation*}
\begin{aligned}
&\text{IV}=\bar{\mathbb{E}}\bigg[\bigg(\int_0^T|\bar{\tilde{\rho}}_s(D^{-1}(s)\bar{h}_{XX}(s)\Gamma_s^{-1}Y^1_s\Gamma_s^{-1}Y^1_s)^{\top}|^2ds\bigg)^{\frac{p}{2}}\bigg]\\
&\leq\bar{\mathbb{E}}\bigg[\sup_{0\leq s\leq T}|\bar{\tilde{\rho}}_s|^p\sup_{0\leq s\leq T}|\Gamma_s^{-1}Y^1_s|^p\sup_{0\leq s\leq T}|\Gamma_s^{-1}Y^1_s|^p\bigg(\int_0^T|D^{-1}(s)|^2|\bar{h}_{XX}(s)|^2ds\bigg)^{\frac{p}{2}}\bigg]\\
&\leq C\bigg(\bar{\mathbb{E}}\sup_{0\leq s\leq T}|\bar{\tilde{\rho}}_s|^{2p}\bigg)^{\frac{1}{2}}\bigg(\bar{\mathbb{E}}\sup_{0\leq s\leq T}|\Gamma_s^{-1}Y^1_s|^{4p}\bigg)^{\frac{1}{2}}\leq C\epsilon^p\,.
\end{aligned}
\end{equation*}
Finally, we have 
\begin{equation*}
\begin{aligned}
\text{V}&=\bar{\mathbb{E}}\bigg[\bigg(\int_0^T|\bar{\tilde{\rho}}_s(D^{-1}(s)\delta h_{X}(s)\Gamma_s^{-1}Y^1_sI_{E_{\epsilon}})^{\top}|^2ds\bigg)^{\frac{p}{2}}\bigg]\\
&\leq C\bar{\mathbb{E}}\bigg[\bigg(\int_0^T|\bar{\tilde{\rho}}_s|^4|\Gamma_s^{-1}Y^1_s|^4I_{E_{\epsilon}}ds\bigg)^{\frac{p}{4}}\bigg(\int_0^T|\delta h_{X}(s)|^4I_{E_{\epsilon}}ds\bigg)^{\frac{p}{4}}\bigg]\\
&\leq C\epsilon^{\frac{p}{2}}\bar{\mathbb{E}}\bigg[\sup_{0\leq s\leq T}|\bar{\tilde{\rho}}_s|^p\sup_{0\leq s\leq T}|\Gamma_s^{-1}Y^1_s|^p\bigg]\\
&\leq C\epsilon^{\frac{p}{2}}\bigg(\bar{\mathbb{E}}\sup_{0\leq s\leq T}|\bar{\tilde{\rho}}_s|^{2p}\bigg)^{\frac{1}{2}}\bigg(\bar{\mathbb{E}}\sup_{0\leq s\leq T}|\Gamma_s^{-1}Y^1_s|^{2p}\bigg)^{\frac{1}{2}}\leq C\epsilon^p.
\end{aligned}
\end{equation*}
Therefore, the proof of these four estimates are complete.

\emph{Proof of Lemma \ref{estimate2}}.
We only give  the second estimate and the first estimate can be obtained similarly.    First, we have
\begin{equation*}
\begin{aligned}
&\bar{\tilde{\rho}}_t+\tilde{\rho}^1_t+\tilde{\rho}_t^2=1+\int_0^t\bigg\{(\bar{\tilde{\rho}}_s+\tilde{\rho}_s^1+\tilde{\rho}_s^2)(D^{-1}(s)\bar{h}(s))^{\top}+\bar{\tilde{\rho}}_s(D^{-1}(s)\delta h_{X}(s)\Gamma_s^{-1}Y^1_sI_{E_{\epsilon}})^{\top}\\
&+\bar{\tilde{\rho}}_s[D^{-1}(s)\bar{h}_X(s)\Gamma_s^{-1}(Y^1_s+Y^2_s)]^{\top}+\bar{\tilde{\rho}}_s(D^{-1}(s)\delta h(s)I_{E_{\epsilon}})^{\top}+\tilde{\rho}^1_s(D^{-1}(s)\bar{h}_X(s)\Gamma_s^{-1}Y_s^1)^{\top}\\
&+\tilde{\rho}^1_s(D^{-1}(s)\delta h(s)I_{E_{\epsilon}})^{\top}+\frac{1}{2}\bar{\tilde{\rho}}_s(D^{-1}(s)\bar{h}_{XX}(s)\Gamma_s^{-1}Y^1_s\Gamma_s^{-1}Y^1_s)^{\top}\bigg\}d\zeta_s\,. 
\end{aligned}
\end{equation*}
Then
\begin{equation*}
\begin{aligned}
&\tilde{\rho}_t^{\epsilon}-\bar{\tilde{\rho}}_t-\tilde{\rho}^1_t-\tilde{\rho}_t^2=\int_0^t\bigg\{\tilde{\rho}_s^{\epsilon}(D^{-1}(s)h(s,\Gamma_s^{-1}Y_s^{\epsilon},u^{\epsilon}_s))^{\top}-(\bar{\tilde{\rho}}_s+\tilde{\rho}_s^1+\tilde{\rho}_s^2)(D^{-1}(s)\bar{h}(s))^{\top}\\
&-\bar{\tilde{\rho}}_s[D^{-1}(s)\bar{h}_X(s)\Gamma_s^{-1}(Y^1_s+Y^2_s)]^{\top}-\bar{\tilde{\rho}}_s(D^{-1}(s)\delta h(s)I_{E_{\epsilon}})^{\top}-\tilde{\rho}^1_s(D^{-1}(s)\bar{h}_X(s)\Gamma_s^{-1}Y_s^1)^{\top}\\
&-\tilde{\rho}^1_s(D^{-1}(s)\delta h(s)I_{E_{\epsilon}})^{\top}-\frac{1}{2}\bar{\tilde{\rho}}_s(D^{-1}(s)\bar{h}_{XX}(s)\Gamma_s^{-1}Y^1_s\Gamma_s^{-1}Y^1_s)^{\top}\\
&-\bar{\tilde{\rho}}_s(D^{-1}(s)\delta h_{X}(s)\Gamma_s^{-1}Y^1_sI_{E_{\epsilon}})^{\top}\bigg\}d\zeta_s.
\end{aligned}
\end{equation*}
It is easy to see 
\begin{equation*}
\begin{aligned}
&D^{-1}(s)h(s,\Gamma^{-1}_s(\bar{Y}_s+Y^1_s+Y^2_s),u^{\epsilon})-D^{-1}(s)\bar{h}(s)\\
=&D^{-1}(s)h(s,\Gamma^{-1}_s(\bar{Y}_s+Y^1_s+Y^2_s),u^{\epsilon})-D^{-1}(s)h(s,\Gamma^{-1}_s\bar{Y}_s,u^{\epsilon})\\
&+D^{-1}(s)h(s,\Gamma^{-1}_s\bar{Y}_s,u^{\epsilon})-D^{-1}(s)\bar{h}(s)\\
=&D^{-1}(s)\int_0^1h_X(s,\Gamma_s^{-1}\bar{Y}_s+\alpha\Gamma^{-1}_s(Y^1_s+Y^2_s),u^{\epsilon})d\alpha\Gamma_s^{-1}(Y_s^1+Y_s^2)+D^{-1}(s)\delta h(s)I_{E_{\epsilon}}\\
=&D^{-1}(s)\int_0^1[h_X(s,\Gamma_s^{-1}\bar{Y}_s+\alpha\Gamma^{-1}_s(Y^1_s+Y^2_s),u^{\epsilon})-h_X(s,\Gamma_s^{-1}\bar{Y}_s,u^{\epsilon}_s)]d\alpha\Gamma_s^{-1}(Y_s^1+Y_s^2)\\
&+D^{-1}(s)h_X(s,\Gamma_s^{-1}\bar{Y}_s,u^{\epsilon}_s)\Gamma_s^{-1}(Y_s^1+Y_s^2)+D^{-1}(s)\delta h(s)I_{E_{\epsilon}}\\
=&D^{-1}(s)\int_0^1\int_0^1\alpha h_{XX}(s,\Gamma_s^{-1}\bar{Y}_s+\beta\alpha\Gamma^{-1}_s(Y^1_s+Y^2_s),u^{\epsilon})d\beta\Gamma_s^{-1}(Y_s^1+Y_s^2)\\
&\times d\alpha\Gamma_s^{-1}(Y_s^1+Y_s^2)+D^{-1}(s)h_X(s,\Gamma_s^{-1}\bar{Y}_s,u^{\epsilon}_s)\Gamma_s^{-1}(Y_s^1+Y_s^2)+D^{-1}(s)\delta h(s)I_{E_{\epsilon}}\,,  
\end{aligned}
\end{equation*}
where $\int_0^1\int_0^1\alpha h_{XX}d\beta d\alpha\Gamma_s^{-1}(Y_s^1+Y_s^2)\Gamma_s^{-1}(Y_s^1+Y_s^2)$ has similar notation as \eqref{n.7.5}
and
\begin{equation*}
\begin{aligned}
&(\bar{\tilde{\rho}}_s+\tilde{\rho}_s^1+\tilde{\rho}_s^2)[D^{-1}(s)h(s,\Gamma^{-1}_s(\bar{Y}_s+Y^1_s+Y^2_s),u^{\epsilon}_s)]^{\top}-\bar{\tilde{\rho}}_s[D^{-1}(s)\bar{h}(s)]^{\top}\\
&=(\bar{\tilde{\rho}}_s+\tilde{\rho}_s^1+\tilde{\rho}_s^2)[D^{-1}(s)h(s,\Gamma^{-1}_s(\bar{Y}_s+Y^1_s+Y^2_s),u^{\epsilon}_s)-D^{-1}(s)h(s,\Gamma^{-1}_s\bar{Y}_s,u^{\epsilon}_s)]^{\top}\\
&\quad+(\bar{\tilde{\rho}}_s+\tilde{\rho}_s^1+\tilde{\rho}_s^2)[D^{-1}(s)h(s,\Gamma^{-1}_s\bar{Y}_s,u^{\epsilon}_s)]^{\top}-\bar{\tilde{\rho}}_s[D^{-1}(s)\bar{h}(s)]^{\top}\\
&=(\bar{\tilde{\rho}}_s+\tilde{\rho}_s^1+\tilde{\rho}_s^2)\bigg[D^{-1}(s)\int_0^1h_X(s,\Gamma^{-1}_s\bar{Y}_s+\alpha\Gamma^{-1}_s(Y^1_s+Y^2_s),u^{\epsilon}_s)d\alpha\Gamma_s^{-1}(Y_s^1+Y_s^2)\bigg]^{\top}\\
&\quad+(\bar{\tilde{\rho}}_s+\tilde{\rho}_s^1+\tilde{\rho}_s^2)[D^{-1}(s)(h(s,\Gamma^{-1}_s\bar{Y}_s,u^{\epsilon}_s)-\bar{h}(s))]^{\top}+(\tilde{\rho}_s^1+\tilde{\rho}_s^2)[D^{-1}(s)\bar{h}(s)]^{\top}\\
&=(\bar{\tilde{\rho}}_s+\tilde{\rho}_s^1+\tilde{\rho}_s^2)\bigg[D^{-1}(s)\int_0^1[h_X(s,\Gamma^{-1}_s\bar{Y}_s+\alpha\Gamma^{-1}_s(Y^1_s+Y^2_s),u^{\epsilon}_s)\\
&\quad-h_X(s,\Gamma_s^{-1}\bar{Y}_s,u_s^{\epsilon})]d\alpha\Gamma_s^{-1}(Y_s^1+Y_s^2)+D^{-1}(s)h_X(s,\Gamma_s^{-1}\bar{Y}_s,u_s^{\epsilon})\Gamma_s^{-1}(Y_s^1+Y_s^2)\bigg]^{\top}\\
&\quad+(\bar{\tilde{\rho}}_s+\tilde{\rho}_s^1+\tilde{\rho}_s^2)[D^{-1}(s)\delta h(s)I_{E_{\epsilon}}]^{\top}+(\tilde{\rho}_s^1+\tilde{\rho}_s^2)[D^{-1}(s)\bar{h}(s)]^{\top}\\
\end{aligned}
\end{equation*}
\begin{equation*}
\begin{aligned}
&=(\bar{\tilde{\rho}}_s+\tilde{\rho}_s^1+\tilde{\rho}_s^2)\bigg[D^{-1}(s)\int_0^1\int_0^1\alpha h_{XX}(s,\Gamma^{-1}_s\bar{Y}_s+\beta\alpha\Gamma^{-1}_s(Y^1_s+Y^2_s),u^{\epsilon}_s)d\beta d\alpha\\
&\times\Gamma_s^{-1}(Y_s^1+Y_s^2)\Gamma_s^{-1}(Y_s^1+Y_s^2)\bigg]^{\top}+(\bar{\tilde{\rho}}_s+\tilde{\rho}_s^1+\tilde{\rho}_s^2)\bigg[D^{-1}(s)h_X(s,\Gamma_s^{-1}\bar{Y}_s,u_s^{\epsilon})\\
&\times\Gamma_s^{-1}(Y_s^1+Y_s^2)\bigg]^{\top}+(\bar{\tilde{\rho}}_s+\tilde{\rho}_s^1+\tilde{\rho}_s^2)[D^{-1}(s)\delta h(s)I_{E_{\epsilon}}]^{\top}+(\tilde{\rho}_s^1+\tilde{\rho}_s^2)[D^{-1}(s)\bar{h}(s)]^{\top}\,. 
\end{aligned}
\end{equation*}
By the above analysis, we   obtain
\begin{equation*}
\begin{aligned}
&\tilde{\rho}_t^{\epsilon}-\bar{\tilde{\rho}}_t-\tilde{\rho}^1_t-\tilde{\rho}_t^2=\int_0^t\bigg\{(\tilde{\rho}^{\epsilon}_s-\bar{\tilde{\rho}}_s-\tilde{\rho}^1_s-\tilde{\rho}_s^2)(D^{-1}(s)h(s,\Gamma^{-1}_sY^{\epsilon},u^{\epsilon}_s))^{\top}\\
&+(\bar{\tilde{\rho}}_s+\tilde{\rho}^1_s+\tilde{\rho}_s^2)\bigg[D^{-1}(s)\int_0^1h_X(s,\Gamma_s^{-1}(\bar{Y}_s+Y_s^1+Y_s^2)+\theta\Gamma_s^{-1}(Y^{\epsilon}_s-\bar{Y}_s-Y_s^1-Y_s^2),u^{\epsilon}_s)d\theta\\
&\times\Gamma_s^{-1}(Y^{\epsilon}_s-\bar{Y}_s-Y_s^1-Y_s^2)\bigg]^{\top}+(\bar{\tilde{\rho}}_s+\tilde{\rho}_s^1+\tilde{\rho}_s^2)\bigg[D^{-1}(s)\int_0^1\int_0^1\alpha [h_{XX}(s,\Gamma^{-1}_s\bar{Y}_s\\
&+\beta\alpha\Gamma^{-1}_s(Y^1_s+Y^2_s),u^{\epsilon}_s)-\bar{h}_{XX}]d\beta d\alpha\Gamma_s^{-1}(Y_s^1+Y_s^2)\Gamma_s^{-1}(Y_s^1+Y_s^2)\bigg]^{\top}\\
&+(\bar{\tilde{\rho}}_s+\tilde{\rho}_s^1+\tilde{\rho}_s^2)[D^{-1}(s)\delta h_X(s)\Gamma_s^{-1}(Y_s^1+Y_s^2)I_{E_{\epsilon}}]^{\top}+\tilde{\rho}^2_s[D^{-1}(s)\bar{h}_X(s)\Gamma_s^{-1}(Y_s^1+Y_s^2)]^{\top}\\
&+\tilde{\rho}^2_s[D^{-1}(s)\delta h(s)I_{E_{\epsilon}}]^{\top}+\tilde{\rho}^1_s[D^{-1}(s)\bar{h}_X(s)\Gamma_s^{-1}Y_s^2]^{\top}-\bar{\tilde{\rho}}_s[D^{-1}(s)\delta h_X(s)\Gamma_s^{-1}Y_s^1I_{E_{\epsilon}}]^{\top}\\
&+\frac{1}{2}
(\tilde{\rho}_s^1+\tilde{\rho}_s^2)[D^{-1}(s)\bar{h}_{XX}(s)\Gamma_s^{-1}(Y_s^1+Y_s^2)\Gamma_s^{-1}(Y_s^1+Y_s^2)]^{\top}\\
&+\frac{1}{2}\bar{\tilde{\rho}}_s[D^{-1}(s)\bar{h}_{XX}(s)(\Gamma_s^{-1}Y_s^2\Gamma_s^{-1}Y_s^2+2\Gamma_s^{-1}Y_s^1\Gamma_s^{-1}Y_s^2)]^{\top}\bigg\}d\zeta_s.
\end{aligned}
\end{equation*}
Using classical $L^p$-estimate of SDE, we have
\begin{equation*}
\begin{aligned}
&\bar{\mathbb{E}}\bigg[\sup_{0\leq s\leq T}|\tilde{\rho}^{\epsilon}_s-\bar{\tilde{\rho}}_s-\tilde{\rho}_s^1-\tilde{\rho}_s^2|^p\bigg]\\
&\leq C\bar{\mathbb{E}}\bigg[\bigg(\int_0^T\bigg|(\bar{\tilde{\rho}}_s+\tilde{\rho}^1_s+\tilde{\rho}_s^2)\bigg[D^{-1}(s)\int_0^1h_X(s,\Gamma_s^{-1}(\bar{Y}_s+Y_s^1+Y_s^2)\\
&+\theta\Gamma_s^{-1}(Y^{\epsilon}_s-\bar{Y}_s-Y_s^1-Y_s^2),u^{\epsilon}_s)d\theta\Gamma_s^{-1}(Y^{\epsilon}_s-\bar{Y}_s-Y_s^1-Y_s^2)\bigg]^{\top}\\
&+(\bar{\tilde{\rho}}_s+\tilde{\rho}_s^1+\tilde{\rho}_s^2)\bigg[D^{-1}(s)\int_0^1\int_0^1\alpha [h_{XX}(s,\Gamma^{-1}_s\bar{Y}_s+\beta\alpha\Gamma^{-1}_s(Y^1_s+Y^2_s),u^{\epsilon}_s)\\
&-\bar{h}_{XX}]d\beta d\alpha\Gamma_s^{-1}(Y_s^1+Y_s^2)\Gamma_s^{-1}(Y_s^1+Y_s^2)\bigg]^{\top}+\tilde{\rho}^2_s[D^{-1}(s)\bar{h}_X(s)\Gamma_s^{-1}(Y_s^1+Y_s^2)]^{\top}\\
&+(\tilde{\rho}_s^1+\tilde{\rho}_s^2)[D^{-1}(s)\delta h_X(s)\Gamma_s^{-1}(Y_s^1+Y_s^2)I_{E_{\epsilon}}]^{\top}+\tilde{\rho}^2_s[D^{-1}(s)\delta h(s)I_{E_{\epsilon}}]^{\top}\\
&+\tilde{\rho}^1_s[D^{-1}(s)\bar{h}_X(s)\Gamma_s^{-1}Y_s^2]^{\top}+\bar{\tilde{\rho}}_s[D^{-1}(s)\delta h_X(s)\Gamma_s^{-1}Y_s^2I_{E_{\epsilon}}]^{\top}\\
&+\frac{1}{2}
(\tilde{\rho}_s^1+\tilde{\rho}_s^2)[D^{-1}(s)\bar{h}_{XX}(s)\Gamma_s^{-1}(Y_s^1+Y_s^2)\Gamma_s^{-1}(Y_s^1+Y_s^2)]^{\top}\\
&+\frac{1}{2}\bar{\tilde{\rho}}_s[D^{-1}(s)\bar{h}_{XX}(s)(\Gamma_s^{-1}Y_s^2\Gamma_s^{-1}Y_s^2+2\Gamma_s^{-1}Y_s^1\Gamma_s^{-1}Y_s^2)]^{\top}\bigg|^2ds\bigg)^{\frac{p}{2}}\bigg]\\
&=\mathbb{A}_1+\mathbb{A}_2+\mathbb{A}_3+\mathbb{A}_4+\mathbb{A}_5+\mathbb{A}_6+\mathbb{A}_7+\mathbb{A}_8+\mathbb{A}_9.
\end{aligned}
\end{equation*}
Now by H\"older's inequality and the assumption {\bf (H2)} we can obtain  the  estimates  
for each of them, respectively.  In fact,  we have 
\begin{equation*}
\begin{aligned}
&\mathbb{A}_1=\bar{\mathbb{E}}\bigg[\bigg(\int_0^T\bigg|(\bar{\tilde{\rho}}_s+\tilde{\rho}^1_s+\tilde{\rho}_s^2)\bigg[D^{-1}(s)\int_0^1h_X(s,\Gamma_s^{-1}(\bar{Y}_s+Y_s^1+Y_s^2)\\
&\quad+\theta\Gamma_s^{-1}(Y^{\epsilon}_s-\bar{Y}_s-Y_s^1-Y_s^2),u^{\epsilon}_s)d\theta\Gamma_s^{-1}(Y^{\epsilon}_s-\bar{Y}_s-Y_s^1-Y_s^2)\bigg]^{\top}\bigg|^2ds\bigg)^{\frac{p}{2}}\bigg]\\
&\leq \bar{\mathbb{E}}\bigg[\sup_{0\leq s\leq T}|\bar{\tilde{\rho}}_s+\tilde{\rho}^1_s+\tilde{\rho}_s^2|^p\sup_{0\leq s\leq T}|\Gamma_s^{-1}(Y^{\epsilon}_s-\bar{Y}_s-Y_s^1-Y_s^2)|^p\bigg(\int_0^T|D^{-1}(s)|^2\\
&\quad\times\bigg(\int_0^1|h_X(s,\Gamma_s^{-1}(\bar{Y}_s+Y_s^1+Y_s^2)+\theta\Gamma_s^{-1}(Y^{\epsilon}_s-\bar{Y}_s-Y_s^1-Y_s^2),u^{\epsilon}_s)|d\theta\bigg)^2ds\bigg)^{\frac{p}{2}}\bigg]\\
&\leq C\bigg(\bar{\mathbb{E}}\sup_{0\leq s\leq T}|\bar{\tilde{\rho}}_s+\tilde{\rho}^1_s+\tilde{\rho}_s^2|^{2p}\bigg)^{\frac{1}{2}}\bigg(\bar{\mathbb{E}}\sup_{0\leq s\leq T}|\Gamma_s^{-1}(Y^{\epsilon}_s-\bar{Y}_s-Y_s^1-Y_s^2)|^{2p}\bigg)^{\frac{1}{2}}\\
&\leq C(C+\epsilon^{\frac{p}{2}}+\epsilon^p)\bigg(\bar{\mathbb{E}}\sup_{0\leq s\leq T}|\Gamma_s^{-1}|^{4p}\bigg)^{\frac{1}{4}}\bigg(\bar{\mathbb{E}}\sup_{0\leq s\leq T}|Y^{\epsilon}_s-\bar{Y}_s-Y_s^1-Y_s^2|^{4p}\bigg)^{\frac{1}{4}}=o(\epsilon^p)\,. 
\end{aligned}
\end{equation*}
Next, we have 
\begin{equation*}
\begin{aligned}
\mathbb{A}_2&=\bar{\mathbb{E}}\bigg[\bigg(\int_0^T\bigg|(\bar{\tilde{\rho}}_s+\tilde{\rho}_s^1+\tilde{\rho}_s^2)\bigg[D^{-1}(s)\int_0^1\int_0^1\alpha [h_{XX}(s,\Gamma^{-1}_s\bar{Y}_s+\beta\alpha\Gamma^{-1}_s(Y^1_s+Y^2_s),u^{\epsilon}_s)\\
&\quad-\bar{h}_{XX}]d\beta d\alpha\Gamma_s^{-1}(Y_s^1+Y_s^2)\Gamma_s^{-1}(Y_s^1+Y_s^2)\bigg]^{\top}\bigg|^2ds\bigg)^{\frac{p}{2}}\bigg]\\
&\leq C\epsilon^{\frac{p}{2}}\bar{\mathbb{E}}\bigg[\sup_{0\leq s\leq T}|\bar{\tilde{\rho}}_s+\tilde{\rho}_s^1+\tilde{\rho}_s^2|^p\sup_{0\leq s\leq T}|\Gamma_s^{-1}(Y_s^1+Y_s^2)|^p\sup_{0\leq s\leq T}|\Gamma_s^{-1}(Y_s^1+Y_s^2)|^p\bigg]\\
&\leq C\epsilon^{\frac{p}{2}}\bigg(\bar{\mathbb{E}}\sup_{0\leq s\leq T}|\bar{\tilde{\rho}}_s+\tilde{\rho}_s^1+\tilde{\rho}_s^2|^{2p}\bigg)^{\frac{1}{2}}\bigg(\bar{\mathbb{E}}\sup_{0\leq s\leq T}|\Gamma_s^{-1}(Y_s^1+Y_s^2)|^{4p}\bigg)^{\frac{1}{2}}\leq C\epsilon^{\frac{3p}{2}}=o(\epsilon^p)\,. \\ 
\end{aligned}
\end{equation*}
We also have 

\begin{equation*}
\begin{aligned}
\mathbb{A}_3&=\bar{\mathbb{E}}\bigg[\bigg(\int_0^T\bigg|\tilde{\rho}^2_s[D^{-1}(s)\bar{h}_X(s)\Gamma_s^{-1}(Y_s^1+Y_s^2)]^{\top}\bigg|^2ds\bigg)^{\frac{p}{2}}\bigg]\\
&\leq \bar{\mathbb{E}}\bigg[\sup_{0\leq s\leq T}|\tilde{\rho}^2_s|^p\sup_{0\leq s \leq T}|\Gamma_s^{-1}(Y_s^1+Y_s^2)|^p\bigg(\int_0^T|D^{-1}(s)|^2|\bar{h}_X(s)|^2ds\bigg)^{\frac{p}{2}}\bigg]\\
&\leq C\bigg(\bar{\mathbb{E}}\sup_{0\leq s\leq T}|\tilde{\rho}^2_s|^{2p}\bigg)^{\frac{1}{2}}\bigg(\bar{\mathbb{E}}\sup_{0\leq s \leq T}|\Gamma_s^{-1}(Y_s^1+Y_s^2)|^{2p}\bigg)^{\frac{1}{2}}\leq C\epsilon^{\frac{3p}{2}}=o(\epsilon^p)\,. 
\end{aligned}
\end{equation*}
\begin{equation*}
\begin{aligned}
\mathbb{A}_4&=\bar{\mathbb{E}}\bigg[\bigg(\int_0^T\bigg|(\tilde{\rho}_s^1+\tilde{\rho}_s^2)[D^{-1}(s)\delta h_X(s)\Gamma_s^{-1}(Y_s^1+Y_s^2)I_{E_{\epsilon}}]^{\top}\bigg|^2ds\bigg)^{\frac{p}{2}}\bigg]\\
&\leq \bar{\mathbb{E}}\bigg[\bigg(\int_0^T|\tilde{\rho}_s^1+\tilde{\rho}_s^2|^4|D^{-1}(s)|^4|\Gamma_s^{-1}(Y_s^1+Y_s^2)|^4I_{E_{\epsilon}}ds\bigg)^{\frac{p}{4}}\bigg(\int_0^T|\delta h_X(s)|^4I_{E_{\epsilon}}ds\bigg)^{\frac{p}{4}}\bigg]\\
&\leq C\epsilon^{\frac{p}{2}}\bar{\mathbb{E}}\bigg[\sup_{0\leq s\leq T}|\tilde{\rho}_s^1+\tilde{\rho}_s^2|^p\sup_{0\leq s\leq T}|\Gamma_s^{-1}(Y_s^1+Y_s^2)|^p\bigg]\\
\end{aligned}
\end{equation*}
\begin{equation*}
\begin{aligned}
&\leq C\epsilon^{\frac{p}{2}}\bigg(\bar{\mathbb{E}}\sup_{0\leq s\leq T}|\tilde{\rho}_s^1+\tilde{\rho}_s^2|^{2p}\bigg)^{\frac{1}{2}}\bigg(\bar{\mathbb{E}}\sup_{0\leq s\leq T}|\Gamma_s^{-1}(Y_s^1+Y_s^2)|^{2p}\bigg)^{\frac{1}{2}}\leq C\epsilon^{\frac{3p}{2}}=o(\epsilon^p)\,. 
\end{aligned}
\end{equation*}

We continue to consider 
\begin{equation*}
\begin{aligned}
\mathbb{A}_5&=\bar{\mathbb{E}}\bigg[\bigg(\int_0^T\bigg|\tilde{\rho}^2_s[D^{-1}(s)\delta h(s)I_{E_{\epsilon}}]^{\top}\bigg|^2ds\bigg)^{\frac{p}{2}}\bigg]\\
&\leq \bar{\mathbb{E}}\bigg[\bigg(\int_0^T|\tilde{\rho}^2_s|^4|D^{-1}(s)|^4I_{E_{\epsilon}}ds\bigg)^{\frac{p}{4}}\bigg(\int_0^T|\delta h(s)|^4I_{E_{\epsilon}}ds\bigg)^{\frac{p}{4}}\bigg]\\
&\leq C\epsilon^{\frac{p}{2}}\bar{\mathbb{E}}\bigg[\sup_{0\leq s\leq T}|\tilde{\rho}^2_s|^p\bigg]\leq C\epsilon^{\frac{3p}{2}}=o(\epsilon^p),
\end{aligned}
\end{equation*}
and 
\begin{equation*}
\begin{aligned}
\mathbb{A}_6&=\bar{\mathbb{E}}\bigg[\bigg(\int_0^T\bigg|\tilde{\rho}^1_s[D^{-1}(s)\bar{h}_X(s)\Gamma_s^{-1}Y_s^2]^{\top}\bigg|^2ds\bigg)^{\frac{p}{2}}\bigg]\\
&\leq \bar{\mathbb{E}}\bigg[\sup_{0\leq s\leq T}|\tilde{\rho}^1_s|^p\sup_{0\leq s\leq T}|\Gamma_s^{-1}Y_s^2|^p\bigg(\int_0^T|D^{-1}(s)|^2|\bar{h}_X(s)|^2ds\bigg)^{\frac{p}{2}}\bigg]\\
&\leq C\bigg(\bar{\mathbb{E}}\sup_{0\leq s\leq T}|\tilde{\rho}^1_s|^{2p}\bigg)^{\frac{1}{2}}\bigg(\bar{\mathbb{E}}\sup_{0\leq s\leq T}|\Gamma_s^{-1}Y_s^2|^{2p}\bigg)^{\frac{1}{2}}\leq C\epsilon^{\frac{3p}{2}}=o(\epsilon^p),
\end{aligned}
\end{equation*}
\begin{equation*}
\begin{aligned}
\mathbb{A}_7&=\bar{\mathbb{E}}\bigg[\bigg(\int_0^T\bigg|\bar{\tilde{\rho}}_s[D^{-1}(s)\delta h_X(s)\Gamma_s^{-1}Y_s^2I_{E_{\epsilon}}]^{\top}\bigg|^2ds\bigg)^{\frac{p}{2}}\bigg]\\
&\leq \bar{\mathbb{E}}\bigg[\bigg(\int_0^T|\bar{\tilde{\rho}}_s|^4|D^{-1}(s)|^4|\Gamma_s^{-1}Y_s^2|^4I_{E_{\epsilon}}ds\bigg)^{\frac{p}{4}}\bigg(\int_0^T|\delta h_X(s)|^4I_{E_{\epsilon}}ds\bigg)^{\frac{p}{4}}\bigg]\\
&\leq C\epsilon^{\frac{p}{4}}\bar{\mathbb{E}}\bigg[\sup_{0\leq s\leq T}|\bar{\tilde{\rho}}_s|^p\sup_{0\leq s\leq T}|\Gamma_s^{-1}Y_s^2|^p\bigg(\int_0^T|D^{-1}(s)|^4I_{E_{\epsilon}}ds\bigg)^{\frac{p}{4}}\bigg]\\
&\leq C\epsilon^{\frac{p}{2}}\bigg(\bar{\mathbb{E}}\sup_{0\leq s\leq T}|\bar{\tilde{\rho}}_s|^{2p}\bigg)^{\frac{1}{2}}\bigg(\bar{\mathbb{E}}\sup_{0\leq s\leq T}|\Gamma_s^{-1}Y_s^2|^{2p}\bigg)^{\frac{1}{2}}\leq C\epsilon^{\frac{3p}{2}}=o(\epsilon^p).
\end{aligned}
\end{equation*}
As for the last two terms, we have 
\begin{equation*}
\begin{aligned}
\mathbb{A}_8&=\bar{\mathbb{E}}\bigg[\bigg(\int_0^T\bigg|(\tilde{\rho}_s^1+\tilde{\rho}_s^2)[D^{-1}(s)\bar{h}_{XX}(s)\Gamma_s^{-1}(Y_s^1+Y_s^2)\Gamma_s^{-1}(Y_s^1+Y_s^2)]^{\top}\bigg|^2ds\bigg)^{\frac{p}{2}}\bigg]\\
&\leq C\bar{\mathbb{E}}\bigg[\sup_{0\leq s\leq T}|\tilde{\rho}_s^1+\tilde{\rho}_s^2|^p\sup_{0\leq s\leq T}|\Gamma_s^{-1}(Y_s^1+Y_s^2)|^{2p}\bigg]\\
&\leq C\bigg(\bar{\mathbb{E}}\sup_{0\leq s\leq T}|\tilde{\rho}_s^1+\tilde{\rho}_s^2|^{2p}\bigg)^{\frac{1}{2}}\bigg(\bar{\mathbb{E}}\sup_{0\leq s\leq T}|\Gamma_s^{-1}(Y_s^1+Y_s^2)|^{4p}\bigg)^{\frac{1}{2}}\leq C\epsilon^{\frac{3p}{2}}=o(\epsilon^p),\\
\end{aligned}
\end{equation*}
and 
\begin{equation*}
\begin{aligned}
&\mathbb{A}_9=\bar{\mathbb{E}}\bigg[\bigg(\int_0^T\bigg|\bar{\tilde{\rho}}_s[D^{-1}(s)\bar{h}_{XX}(s)\Gamma_s^{-1}Y_s^1\Gamma_s^{-1}Y_s^2]^{\top}\bigg|^2ds\bigg)^{\frac{p}{2}}\bigg]\\
&\leq \bar{\mathbb{E}}\bigg[\sup_{0\leq s\leq T}|\bar{\tilde{\rho}}_s|^p\sup_{0\leq s\leq T}|\Gamma_s^{-1}Y_s^1|^p\sup_{0\leq s\leq T}|\Gamma_s^{-1}Y_s^2|^p\bigg(\int_0^T|D^{-1}(s)|^2|\bar{h}_{XX}(s)|^2ds\bigg)^{\frac{p}{2}}\bigg]\\
&\leq C\bigg(\bar{\mathbb{E}}\sup_{0\leq s\leq T}|\bar{\tilde{\rho}}_s|^{2p}\bigg)^{\frac{1}{2}}\bigg(\bar{\mathbb{E}}\sup_{0\leq s\leq T}|\Gamma_s^{-1}Y_s^1|^{4p}\bigg)^{\frac{1}{4}}\bigg(\bar{\mathbb{E}}\sup_{0\leq s\leq T}|\Gamma_s^{-1}Y_s^2|^{4p}\bigg)^{\frac{1}{4}}\leq C\epsilon^{\frac{3p}{2}}=o(\epsilon^p)\,. 
\end{aligned}
\end{equation*}
Similarly, we have 
\begin{equation*}
\begin{aligned}
&\bar{\mathbb{E}}\bigg[\bigg(\int_0^T\bigg|\bar{\tilde{\rho}}_s[D^{-1}(s)\bar{h}_{XX}(s)\Gamma_s^{-1}Y_s^2\Gamma_s^{-1}Y_s^2]^{\top}\bigg|^2ds\bigg)^{\frac{p}{2}}\bigg]\leq C\epsilon^{2p}=o(\epsilon^p),
\end{aligned}
\end{equation*}

Therefore, we have $\bar{\mathbb{E}}[\sup_{0\leq s\leq T}|\tilde{\rho}^{\epsilon}_s-\bar{\tilde{\rho}}_s-\tilde{\rho}_s^1-\tilde{\rho}_s^2|^p]=o(\epsilon^p)$.

\emph{Proof of Lemma \ref{Jestimate}}.  We can write 
\begin{equation*}
\begin{aligned}
&J(\bar{u})+\hat{J}:=\bar{\mathbb{E}}\bigg[\int_0^T\bigg\{(\bar{\tilde{\rho}}_t+\tilde{\rho}_t^1+\tilde{\rho}_t^2)\bar{f}(t)+\bar{\tilde{\rho}}_t\bigg[\langle \bar{f}_X(t),\Gamma_t^{-1}(Y_t^1+Y_t^2)\rangle\\
&+\frac{1}{2}\langle\bar{f}_{XX}(t)\Gamma_t^{-1}Y_t^1,\Gamma_t^{-1}Y_t^1\rangle+\delta f(t)I_{E_{\epsilon}}\bigg]+\tilde{\rho}_t^1\langle\bar{f}_X(t),\Gamma_t^{-1}Y_t^1\rangle\bigg\}dt\bigg]\\
&+\bar{\mathbb{E}}\bigg[\bar{\tilde{\rho}}_T\bigg[\langle\Phi_X(\Gamma_T^{-1}\bar{Y}_T),\Gamma_T^{-1}(Y_T^1+Y^2_T)\rangle+\frac{1}{2}\langle\Phi_{XX}(\Gamma_T^{-1}\bar{Y}_T)\Gamma_T^{-1}Y_T^1,\Gamma_T^{-1}Y_T^1\rangle\bigg]\\
&+(\bar{\tilde{\rho}}_T+\tilde{\rho}_T^1+\tilde{\rho}_T^2)\Phi(\Gamma_T^{-1}\bar{Y}_T)+\tilde{\rho}_T^1\langle\Phi_X(\Gamma_T^{-1}\bar{Y}_T),\Gamma_T^{-1}Y_T^1\rangle\bigg]\,. 
\end{aligned}
\end{equation*}
Then by the similar argument as in the proof of Lemma \ref{estimate2}, we obtain
\begin{equation*}
\begin{aligned}
&(\bar{\tilde{\rho}}_t+\tilde{\rho}_t^1+\tilde{\rho}_t^2)f(t,\Gamma_t^{-1}(\bar{Y}_t+Y_t^1+Y_t^2),u^{\epsilon}_t)-\bar{\tilde{\rho}}_t\bar{f}(t)\\
=&(\bar{\tilde{\rho}}_t+\tilde{\rho}_t^1+\tilde{\rho}_t^2)\bigg\langle\int_0^1\int_0^1\alpha f_{XX}(t,\Gamma_t^{-1}\bar{Y}_t+\beta\alpha\Gamma_t^{-1}(Y_t^1+Y_t^2),u^{\epsilon}_t)d\beta d\alpha\\
&\times\Gamma_t^{-1}(Y_t^1+Y_t^2),\Gamma_t^{-1}(Y_t^1+Y_t^2)\bigg\rangle+(\bar{\tilde{\rho}}_t+\tilde{\rho}_t^1+\tilde{\rho}_t^2)\\
&\times\langle f_X(t,\Gamma_t^{-1}\bar{Y}_t,u^{\epsilon}_t),\Gamma_t^{-1}(Y_t^1+Y_t^2)\rangle+(\bar{\tilde{\rho}}_t+\tilde{\rho}_t^1+\tilde{\rho}_t^2)\delta f(t)I_{E_{\epsilon}}+(\tilde{\rho}_t^1+\tilde{\rho}_t^2)\bar{f}(t),
\end{aligned}
\end{equation*}
and
\begin{equation*}
\begin{aligned}
&(\bar{\tilde{\rho}}_T+\tilde{\rho}_T^1+\tilde{\rho}_T^2)\Phi(\Gamma_T^{-1}(\bar{Y}_T+Y_T^1+Y_T^2))-\bar{\tilde{\rho}}_T\Phi(\Gamma_T^{-1}\bar{Y}_T)\\
=&(\bar{\tilde{\rho}}_T+\tilde{\rho}_T^1+\tilde{\rho}_T^2)\bigg\langle\int_0^1\int_0^1\alpha\Phi_{XX}(\Gamma_T^{-1}\bar{Y}_T+\beta\alpha\Gamma_T^{-1}(Y_T^1+Y_T^2))d\beta d\alpha\\
&\times\Gamma_T^{-1}(Y_T^1+Y_T^2),\Gamma_T^{-1}(Y_T^1+Y_T^2)\bigg\rangle+(\bar{\tilde{\rho}}_T+\tilde{\rho}_T^1+\tilde{\rho}_T^2)\\
&\times\langle\Phi_X(\Gamma_T^{-1}\bar{Y}_T),\Gamma_T^{-1}(Y_T^1+Y_T^2)\rangle+(\tilde{\rho}_T^1+\tilde{\rho}_T^2)\Phi(\Gamma_T^{-1}\bar{Y}_T).
\end{aligned}
\end{equation*}
Therefore, we have
\begin{equation*}
\begin{aligned}
&J(u^{\epsilon})-J(\bar{u})-\hat{J}=\bar{\mathbb{E}}\bigg[\int_0^T\bigg\{(\tilde{\rho}_t^{\epsilon}-\bar{\tilde{\rho}}_t-\tilde{\rho}_t^1-\tilde{\rho}_t^2)f(t,\Gamma_t^{-1}Y^{\epsilon}_t,u^{\epsilon}_t)+(\tilde{\rho}_t^1+\tilde{\rho}_t^2)\langle\bar{f}_X(t),\Gamma_t^{-1}Y^2_t\rangle\\
&+(\bar{\tilde{\rho}}_t+\tilde{\rho}^1_t+\tilde{\rho}^2_t)\langle\delta f_X(t),\Gamma_t^{-1}(Y^1_t+Y^2_t)\rangle I_{E_{\epsilon}}+(\bar{\tilde{\rho}}_t+\tilde{\rho}^1_t+\tilde{\rho}^2_t)\bigg\langle\int_0^1f_X(t,\Gamma_t^{-1}(\bar{Y}_t+Y^1_t+Y^2_t)\\
&+\theta\Gamma_t^{-1}(Y^{\epsilon}_t-\bar{Y}_t-Y^1_t-Y^2_t),u^{\epsilon}_t)d\theta,\Gamma_t^{-1}(Y^{\epsilon}_t-\bar{Y}_t-Y^1_t-Y^2_t)\bigg\rangle+(\tilde{\rho}_t^1+\tilde{\rho}_t^2)\delta f(t)I_{E_{\epsilon}}\\
&+(\bar{\tilde{\rho}}_t+\tilde{\rho}^1_t+\tilde{\rho}^2_t)\bigg\langle\int_0^1\int_0^1\alpha[f_{XX}(t,\Gamma_t^{-1}\bar{Y}_t+\beta\alpha\Gamma_t^{-1}(Y_t^1+Y_t^2),u^{\epsilon}_t)-\bar{f}_{XX}(t)]d\beta d\alpha\\
&\times\Gamma_t^{-1}(Y_t^1+Y_t^2),\Gamma_t^{-1}(Y_t^1+Y_t^2)\bigg\rangle+\frac{1}{2}(\bar{\tilde{\rho}}_t+\tilde{\rho}_t^1+\tilde{\rho}_t^2)\langle\bar{f}_{XX}(t)\Gamma_t^{-1}(Y_t^1+Y_t^2),\Gamma_t^{-1}Y_t^2\rangle\\
&+\tilde{\rho}_t^2\langle\bar{f}_X(t),\Gamma_t^{-1}Y^1_t\rangle+\frac{1}{2}(\tilde{\rho}_t^1+\tilde{\rho}_t^2)\langle\bar{f}_{XX}(t)\Gamma_t^{-1}Y^1_t,\Gamma_t^{-1}Y_t^1\rangle \bigg\}dt\bigg]\\
\end{aligned}
\end{equation*}
\begin{equation*}
\begin{aligned}
&+\bar{\mathbb{E}}\bigg[(\tilde{\rho}_T^{\epsilon}-\bar{\tilde{\rho}}_T-\tilde{\rho}_T^1-\tilde{\rho}_T^2)\Phi(\Gamma_T^{-1}Y_T^{\epsilon})+(\bar{\tilde{\rho}}_T+\tilde{\rho}_T^1+\tilde{\rho}_T^2)\bigg\langle\int_0^1\Phi_X(\Gamma_T^{-1}(\bar{Y}_T+Y^1_T+Y^2_T)\\
&+\theta\Gamma_T^{-1}(Y_T^{\epsilon}-\bar{Y}_T-Y^1_T-Y^2_T))d\theta,\Gamma_T^{-1}(Y_T^{\epsilon}-\bar{Y}_T-Y^1_T-Y^2_T)\bigg\rangle+\tilde{\rho}_T^1\langle\Phi_X(\Gamma_T^{-1}\bar{Y}_T),\Gamma_T^{-1}Y_T^2\rangle\\
&+\tilde{\rho}^2_T\langle\Phi_X(\Gamma_T^{-1}\bar{Y}_T),\Gamma_T^{-1}(Y_T^1+Y^2_T)\rangle+\frac{1}{2}(\tilde{\rho}_T^1+\tilde{\rho}_T^2)\langle\Phi_{XX}(\Gamma_T^{-1}\bar{Y}_T)\Gamma_T^{-1}Y_T^1,\Gamma_T^{-1}Y_T^1\rangle\\
&+(\bar{\tilde{\rho}}_T+\tilde{\rho}_T^1+\tilde{\rho}_T^2)\bigg\langle\int_0^1\int_0^1\alpha[\Phi_{XX}(\Gamma_T^{-1}\bar{Y}_T+\beta\alpha\Gamma_T^{-1}(Y_T^1+Y_T^2))-\Phi_{XX}(\Gamma_T^{-1}\bar{Y}_T)]d\beta d\alpha\\
&\times\Gamma_T^{-1}(Y_T^1+Y_T^2),\Gamma_T^{-1}(Y_T^1+Y_T^2)\bigg\rangle+\frac{1}{2}(\bar{\tilde{\rho}}_T+\tilde{\rho}_T^1+\tilde{\rho}_T^2)\langle\Phi_{XX}(\Gamma_T^{-1}\bar{Y}_T)\Gamma_T^{-1}(Y_T^1+Y_T^2),\Gamma_T^{-1}Y_T^2\rangle\bigg]\,. 
\end{aligned}
\end{equation*}
We can write the above right hand side as 
\begin{equation*}
\begin{aligned}
|J(u^{\epsilon})-J(\bar{u})-\hat{J}|^2\leq&\mathbb{D}_{1}+\mathbb{D}_{2}+\mathbb{D}_{3}+\mathbb{D}_{4}+\mathbb{D}_{5}+\mathbb{D}_{6}+\mathbb{D}_{7}+\mathbb{D}_{8}+\mathbb{D}_{9}+\mathbb{D}_{10}+\mathbb{D}_{11}\\
&\quad+\mathbb{D}_{12}+\mathbb{D}_{13}+\mathbb{D}_{14}+\mathbb{D}_{15}+\mathbb{D}_{16}\,, 
\end{aligned}
\end{equation*}
where, by H\"older's inequality and the assumption {\bf (H2)} 
\begin{equation*}
\begin{aligned}
\mathbb{D}_1&=\bar{\mathbb{E}}\bigg[\bigg|\int_0^T(\tilde{\rho}_t^{\epsilon}-\bar{\tilde{\rho}}_t-\tilde{\rho}_t^1-\tilde{\rho}_t^2)f(t,\Gamma_t^{-1}Y^{\epsilon}_t,u^{\epsilon}_t)dt\bigg|^2\bigg]\\
&\leq \bar{\mathbb{E}}\bigg[\sup_{0\leq t\leq T}|\tilde{\rho}_t^{\epsilon}-\bar{\tilde{\rho}}_t-\tilde{\rho}_t^1-\tilde{\rho}_t^2|^2\bigg|\int_0^Tf(t,\Gamma_t^{-1}Y^{\epsilon}_t,u^{\epsilon}_t)dt\bigg|^2\bigg]\\
&\leq \bigg(\bar{\mathbb{E}}\sup_{0\leq t\leq T}|\tilde{\rho}_t^{\epsilon}-\bar{\tilde{\rho}}_t-\tilde{\rho}_t^1-\tilde{\rho}_t^2|^4\bigg)^{\frac{1}{2}}\bigg(\bar{\mathbb{E}}\bigg|\int_0^Tf(t,\Gamma_t^{-1}Y^{\epsilon}_t,u^{\epsilon}_t)dt\bigg|^4\bigg)^{\frac{1}{2}}\\
&\leq o(\epsilon^2)\bigg(\bar{\mathbb{E}}\bigg(\int_0^TC(1+|\Gamma_t^{-1}Y^{\epsilon}_t|^2+|u^{\epsilon}_t|^2)dt\bigg)^4\bigg)^{\frac{1}{2}}=o(\epsilon^2)\,,\\
\end{aligned}
\end{equation*}
\begin{equation*}
\begin{aligned}
\mathbb{D}_2&=\bar{\mathbb{E}}\bigg[\bigg|\int_0^T(\tilde{\rho}_t^1+\tilde{\rho}_t^2)\langle\bar{f}_X(t),\Gamma_t^{-1}Y^2_t\rangle dt\bigg|^2\bigg]\\
&\leq\bar{\mathbb{E}}\bigg[\sup_{0\leq t\leq T}|\tilde{\rho}_t^1+\tilde{\rho}_t^2|^2\sup_{0\leq t\leq T}|\Gamma_t^{-1}Y^2_t|^2\bigg(\int_0^TC(1+|\Gamma_t^{-1}\bar{Y}_t|+|\bar{u}_t|)dt\bigg)^2\bigg]\\
&\leq \bigg(\bar{\mathbb{E}}\sup_{0\leq t\leq T}|\tilde{\rho}_t^1+\tilde{\rho}_t^2|^8\bigg)^{\frac{1}{4}}\bigg(\bar{\mathbb{E}}\sup_{0\leq t\leq T}|\Gamma_t^{-1}Y^2_t|^8\bigg)^{\frac{1}{4}}\bigg(\bar{\mathbb{E}}\bigg(\int_0^TC(1+|\Gamma_t^{-1}\bar{Y}_t|+|\bar{u}_t|)dt\bigg)^4\bigg)^{\frac{1}{2}}\\
&\leq C\epsilon\bigg(\bar{\mathbb{E}}\sup_{0\leq t\leq T}|\Gamma_t^{-1}|^{16}\bigg)^{\frac{1}{8}}\bigg(\bar{\mathbb{E}}\sup_{0\leq t\leq T}|Y^2_t|^{16}\bigg)^{\frac{1}{8}}\leq C\epsilon^3=o(\epsilon^2)\,. 
\end{aligned}
\end{equation*}
\begin{equation*}
\begin{aligned}
\mathbb{D}_3&=\bar{\mathbb{E}}\bigg[\bigg|\int_0^T(\bar{\tilde{\rho}}_t+\tilde{\rho}^1_t+\tilde{\rho}^2_t)\langle\delta f_X(t),\Gamma_t^{-1}(Y^1_t+Y^2_t)\rangle I_{E_{\epsilon}} dt\bigg|^2\bigg]\\
&\leq \bar{\mathbb{E}}\bigg[\bigg(\int_0^T|\bar{\tilde{\rho}}_t+\tilde{\rho}^1_t+\tilde{\rho}^2_t|^2|\Gamma_t^{-1}(Y^1_t+Y^2_t)|^2I_{E_{\epsilon}}dt\bigg)\bigg(\int_0^T|\delta f_X(t)|^2 I_{E_{\epsilon}} dt\bigg)\bigg]\\
&\leq \epsilon\bar{\mathbb{E}}\bigg[\sup_{0\leq t\leq T}|\bar{\tilde{\rho}}_t+\tilde{\rho}^1_t+\tilde{\rho}^2_t|^2\sup_{0\leq t\leq T}|\Gamma_t^{-1}(Y^1_t+Y^2_t)|^2\\
&\qquad\times\bigg(\int_0^TC(1+|\Gamma_t^{-1}\bar{Y}_t|^2+|\bar{u}_t|^2+|u_t|^2) I_{E_{\epsilon}} dt\bigg)\bigg]\\
\end{aligned}
\end{equation*}
\begin{equation*}
\begin{aligned}
&\leq \epsilon\bigg(\bar{\mathbb{E}}\sup_{0\leq t\leq T}|\bar{\tilde{\rho}}_t+\tilde{\rho}^1_t+\tilde{\rho}^2_t|^8\bigg)^{\frac{1}{4}}\bigg(\bar{\mathbb{E}}\sup_{0\leq t\leq T}|\Gamma_t^{-1}(Y^1_t+Y^2_t)|^8\bigg)^{\frac{1}{4}}\\
&\qquad\times\bigg(\bar{\mathbb{E}}\bigg(\int_0^TC(1+|\Gamma_t^{-1}\bar{Y}_t|^2+|\bar{u}_t|^2+|u_t|^2) I_{E_{\epsilon}} dt\bigg)^2\bigg)^{\frac{1}{2}}\\
&\leq C\epsilon^2\bigg(\bar{\mathbb{E}}\bigg(\int_0^TC(1+|\Gamma_t^{-1}\bar{Y}_t|^4+|\bar{u}_t|^4+|u_t|^4) I_{E_{\epsilon}} dt\bigg)\bigg(\int_0^TI_{E_{\epsilon}}dt\bigg)\bigg)^{\frac{1}{2}}\\
&\leq C\epsilon^{\frac{5}{2}}\bigg(\epsilon+\epsilon\bar{\mathbb{E}}\sup_{0\leq t\leq T}|\Gamma_t^{-1}\bar{Y}_t|^4+\epsilon\sup_{0\leq t\leq T}\bar{\mathbb{E}}(|\bar{u}_t|^4+|u_t|^4)\bigg)^{\frac{1}{2}}\leq C\epsilon^3=o(\epsilon^2)\,. 
\end{aligned}
\end{equation*}
\begin{equation*}
\begin{aligned}
&\mathbb{D}_4=\bar{\mathbb{E}}\bigg[\bigg|\int_0^T(\bar{\tilde{\rho}}_t+\tilde{\rho}^1_t+\tilde{\rho}^2_t)\bigg\langle\int_0^1f_X(t,\Gamma_t^{-1}(\bar{Y}_t+Y^1_t+Y^2_t)\\
&\quad+\theta\Gamma_t^{-1}(Y^{\epsilon}_t-\bar{Y}_t-Y^1_t-Y^2_t),u^{\epsilon}_t)d\theta,\Gamma_t^{-1}(Y^{\epsilon}_t-\bar{Y}_t-Y^1_t-Y^2_t)\bigg\rangle dt\bigg|^2\bigg]\\
&\leq \bar{\mathbb{E}}\bigg[\sup_{0\leq t\leq T}|\Gamma_t^{-1}(Y^{\epsilon}_t-\bar{Y}_t-Y^1_t-Y^2_t)|^2\bigg(\int_0^T|\bar{\tilde{\rho}}_t+\tilde{\rho}^1_t+\tilde{\rho}^2_t|\\
&\qquad\times\int_0^1|f_X(t,\Gamma_t^{-1}(\bar{Y}_t+Y^1_t+Y^2_t)+\theta\Gamma_t^{-1}(Y^{\epsilon}_t-\bar{Y}_t-Y^1_t-Y^2_t),u^{\epsilon}_t)|d\theta dt\bigg)^2\bigg]\\
&\leq \bigg(\bar{\mathbb{E}}\sup_{0\leq t\leq T}|\Gamma_t^{-1}(Y^{\epsilon}_t-\bar{Y}_t-Y^1_t-Y^2_t)|^4\bigg)^{\frac{1}{2}}\bigg(\bar{\mathbb{E}}\bigg(\int_0^T|\bar{\tilde{\rho}}_t+\tilde{\rho}^1_t+\tilde{\rho}^2_t|\\
&\qquad\times\int_0^1|f_X(t,\Gamma_t^{-1}(\bar{Y}_t+Y^1_t+Y^2_t)+\theta\Gamma_t^{-1}(Y^{\epsilon}_t-\bar{Y}_t-Y^1_t-Y^2_t),u^{\epsilon}_t)|d\theta dt\bigg)^4\bigg)^{\frac{1}{2}}\\
&\leq \bigg(\bar{\mathbb{E}}\sup_{0\leq t\leq T}|\Gamma_t^{-1}|^8\bigg)^{\frac{1}{4}}\bigg(\bar{\mathbb{E}}\sup_{0\leq t\leq T}|Y^{\epsilon}_t-\bar{Y}_t-Y^1_t-Y^2_t|^8\bigg)^{\frac{1}{4}}\bigg(\bar{\mathbb{E}}\bigg(\int_0^T|\bar{\tilde{\rho}}_t+\tilde{\rho}^1_t+\tilde{\rho}^2_t|\\
&\times\int_0^1|f_X(t,\Gamma_t^{-1}(\bar{Y}_t+Y^1_t+Y^2_t)+\theta\Gamma_t^{-1}(Y^{\epsilon}_t-\bar{Y}_t-Y^1_t-Y^2_t),u^{\epsilon}_t)|d\theta dt\bigg)^4\bigg)^{\frac{1}{2}}=o(\epsilon^2)\,. 
\end{aligned}
\end{equation*}
\begin{equation*}
\begin{aligned}
\mathbb{D}_5&=\bar{\mathbb{E}}\bigg[\bigg|\int_0^T(\bar{\tilde{\rho}}_t+\tilde{\rho}^1_t+\tilde{\rho}^2_t)\bigg\langle\int_0^1\int_0^1\alpha[f_{XX}(t,\Gamma_t^{-1}\bar{Y}_t+\beta\alpha\Gamma_t^{-1}(Y_t^1+Y_t^2),u^{\epsilon}_t)\\
&\qquad-\bar{f}_{XX}(t)] d\beta d\alpha\Gamma_t^{-1}(Y_t^1+Y_t^2),\Gamma_t^{-1}(Y_t^1+Y_t^2)\bigg\rangle dt\bigg|^2\bigg]\\
&\leq C\epsilon\bar{\mathbb{E}}\bigg[\int_0^T|\bar{\tilde{\rho}}_t+\tilde{\rho}^1_t+\tilde{\rho}^2_t|^2|\Gamma_t^{-1}(Y_t^1+Y_t^2)|^2|\Gamma_t^{-1}(Y_t^1+Y_t^2)|^2 dt\bigg]\\
&\leq C\epsilon\bar{\mathbb{E}}\bigg[\sup_{0\leq t\leq T}|\bar{\tilde{\rho}}_t+\tilde{\rho}^1_t+\tilde{\rho}^2_t|^2\sup_{0\leq t\leq T}|\Gamma_t^{-1}(Y_t^1+Y_t^2)|^4\bigg]\\
&\leq C\epsilon\bigg(\bar{\mathbb{E}}\sup_{0\leq t\leq T}|\bar{\tilde{\rho}}_t+\tilde{\rho}^1_t+\tilde{\rho}^2_t|^4\bigg)^{\frac{1}{2}}\bigg(\bar{\mathbb{E}}\sup_{0\leq t\leq T}|\Gamma_t^{-1}(Y_t^1+Y_t^2)|^8\bigg)^{\frac{1}{2}}\leq C\epsilon^3=o(\epsilon^2)\,.\\
\end{aligned}
\end{equation*}
\begin{equation*}
\begin{aligned}
\mathbb{D}_6&=\bar{\mathbb{E}}\bigg[\bigg|\int_0^T(\tilde{\rho}_t^1+\tilde{\rho}_t^2)\delta f(t)I_{E_{\epsilon}} dt\bigg|^2\bigg]\\
&\leq \bar{\mathbb{E}}\bigg[\bigg(\int_0^T|\tilde{\rho}_t^1+\tilde{\rho}_t^2|^2I_{E_{\epsilon}}dt\bigg)\bigg(\int_0^T|\delta f(t)|^2I_{E_{\epsilon}} dt\bigg)\bigg]\\
\end{aligned}
\end{equation*}
\begin{equation*}
\begin{aligned}
&\leq \epsilon\bar{\mathbb{E}}\bigg[\sup_{0\leq t\leq T}|\tilde{\rho}_t^1+\tilde{\rho}_t^2|^2\bigg(\int_0^TC(1+|\Gamma_t^{-1}\bar{Y}_t|^4+|\bar{u}_t|^4+|u_t|^4)I_{E_{\epsilon}} dt\bigg)\bigg] \\
&\leq \epsilon\bigg(\bar{\mathbb{E}}\sup_{0\leq t\leq T}|\tilde{\rho}_t^1+\tilde{\rho}_t^2|^4\bigg)^{\frac{1}{2}}\bigg(\bar{\mathbb{E}}\bigg(\int_0^TC(1+|\Gamma_t^{-1}\bar{Y}_t|^4+|\bar{u}_t|^4+|u_t|^4)I_{E_{\epsilon}} dt\bigg)^2\bigg)^{\frac{1}{2}}\\
&\leq C\epsilon^2\bigg(\bar{\mathbb{E}}\bigg(\int_0^T(1+|\Gamma_t^{-1}\bar{Y}_t|^8+|\bar{u}_t|^8+|u_t|^8)I_{E_{\epsilon}} dt\bigg)\bigg(\int_0^TI_{E_{\epsilon}}dt\bigg)\bigg)^{\frac{1}{2}}\\
&\leq C\epsilon^{\frac{5}{2}}\bigg(\epsilon+\epsilon\bar{\mathbb{E}}\sup_{0\leq t\leq T}|\Gamma_t^{-1}\bar{Y}_t|^8+\epsilon\sup_{0\leq t\leq T}\bar{\mathbb{E}}(|\bar{u}_t|^8+|u_t|^8)\bigg)^{\frac{1}{2}}\leq C\epsilon^3=o(\epsilon^2)\,. 
\end{aligned}
\end{equation*}
\begin{equation*}
\begin{aligned}
&\mathbb{D}_7=\bar{\mathbb{E}}\bigg[\bigg|\int_0^T(\bar{\tilde{\rho}}_t+\tilde{\rho}_t^1+\tilde{\rho}_t^2)\langle\bar{f}_{XX}(t)\Gamma_t^{-1}(Y_t^1+Y_t^2),\Gamma_t^{-1}Y_t^2\rangle dt\bigg|^2\bigg]\\
&\leq \bar{\mathbb{E}}\bigg[\sup_{0\leq t\leq T}|\bar{\tilde{\rho}}_t+\tilde{\rho}_t^1+\tilde{\rho}_t^2|^2\sup_{0\leq t\leq T}|\Gamma_t^{-1}(Y_t^1+Y_t^2)|^2\sup_{0\leq t\leq T}|\Gamma_t^{-1}Y_t^2|^2\bigg(\int_0^T|\bar{f}_{XX}(t)|dt\bigg)^2\bigg]\\
&\leq C\bigg(\bar{\mathbb{E}}\sup_{0\leq t\leq T}|\bar{\tilde{\rho}}_t+\tilde{\rho}_t^1+\tilde{\rho}_t^2|^4\bigg)^{\frac{1}{2}}\bigg(\bar{\mathbb{E}}\sup_{0\leq t\leq T}|\Gamma_t^{-1}(Y_t^1+Y_t^2)|^8\bigg)^{\frac{1}{4}}\bigg(\bar{\mathbb{E}}\sup_{0\leq t\leq T}|\Gamma_t^{-1}Y_t^2|^8\bigg)^{\frac{1}{4}}\\
&\leq C(C+\epsilon+\epsilon^2)\bigg(\bar{\mathbb{E}}\sup_{0\leq t\leq T}|Y_t^1+Y_t^2|^{16}\bigg)^{\frac{1}{8}}\bigg(\bar{\mathbb{E}}\sup_{0\leq t\leq T}|Y_t^2|^{16}\bigg)^{\frac{1}{8}}\leq C\epsilon^3=o(\epsilon^2)\,. 
\end{aligned}
\end{equation*}
\begin{equation*}
\begin{aligned}
\mathbb{D}_8&=\bar{\mathbb{E}}\bigg[\bigg|\int_0^T\tilde{\rho}_t^2\langle\bar{f}_X(t),\Gamma_t^{-1}Y^1_t\rangle dt\bigg|^2\bigg]\leq \bar{\mathbb{E}}\bigg[\sup_{0\leq t\leq T}|\tilde{\rho}_t^2|^2\sup_{0\leq t\leq T}|\Gamma_t^{-1}Y^1_t|^2\bigg(\int_0^T|\bar{f}_X(t)| dt\bigg)^2\bigg]\\
&\leq \bigg(\bar{\mathbb{E}}\sup_{0\leq t\leq T}|\tilde{\rho}_t^2|^8\bigg)^{\frac{1}{4}}\bigg(\bar{\mathbb{E}}\sup_{0\leq t\leq T}|\Gamma_t^{-1}Y^1_t|^8\bigg)^{\frac{1}{4}}\bigg(\bar{\mathbb{E}}\bigg(\int_0^T|\bar{f}_X(t)| dt\bigg)^4\bigg)^{\frac{1}{2}}\\
&\leq C\epsilon^2\bigg(\bar{\mathbb{E}}\sup_{0\leq t\leq T}|\Gamma_t^{-1}|^{16}\bigg)^{\frac{1}{8}}\bigg(\bar{\mathbb{E}}\sup_{0\leq t\leq T}|Y^1_t|^{16}\bigg)^{\frac{1}{8}}\leq C\epsilon^3=o(\epsilon^2)\,. 
\end{aligned}
\end{equation*}
\begin{equation*}
\begin{aligned}
\mathbb{D}_9&=\bar{\mathbb{E}}\bigg[\bigg|\int_0^T(\tilde{\rho}_t^1+\tilde{\rho}_t^2)\langle\bar{f}_{XX}(t)\Gamma_t^{-1}Y^1_t,\Gamma_t^{-1}Y_t^1\rangle dt\bigg|^2\bigg]\\
&\leq\bar{\mathbb{E}}\bigg[\sup_{0\leq t\leq T}|\tilde{\rho}_t^1+\tilde{\rho}_t^2|^2\sup_{0\leq t\leq T}|\Gamma_t^{-1}Y^1_t|^2\sup_{0\leq t\leq T}|\Gamma_t^{-1}Y^1_t|^2\bigg(\int_0^T|\bar{f}_{XX}(t)| dt\bigg)^2\bigg]\\
&\leq C\bigg(\bar{\mathbb{E}}\sup_{0\leq t\leq T}|\tilde{\rho}_t^1+\tilde{\rho}_t^2|^4\bigg)^{\frac{1}{2}}\bigg(\bar{\mathbb{E}}\sup_{0\leq t\leq T}|\Gamma_t^{-1}Y^1_t|^8\bigg)^{\frac{1}{2}}\leq C\epsilon^3=o(\epsilon^2)\,.\\
\end{aligned}
\end{equation*}
\begin{equation*}
\begin{aligned}
\mathbb{D}_{10}&=\bar{\mathbb{E}}[|(\tilde{\rho}_T^{\epsilon}-\bar{\tilde{\rho}}_T-\tilde{\rho}_T^1-\tilde{\rho}_T^2)\Phi(\Gamma_T^{-1}Y_T^{\epsilon})|^2]\\
&\leq \bar{\mathbb{E}}\bigg[\sup_{0\leq t\leq T}|\tilde{\rho}_t^{\epsilon}-\bar{\tilde{\rho}}_t-\tilde{\rho}_t^1-\tilde{\rho}_t^2|^2C(1+|\Gamma_T^{-1}Y_T^{\epsilon}|^4)\bigg]\\
&\leq C\bigg(\bar{\mathbb{E}}\sup_{0\leq t\leq T}|\tilde{\rho}_t^{\epsilon}-\bar{\tilde{\rho}}_t-\tilde{\rho}_t^1-\tilde{\rho}_t^2|^4\bigg)^{\frac{1}{2}}(\bar{\mathbb{E}}(1+|\Gamma_T^{-1}Y_T^{\epsilon}|^8))^{\frac{1}{2}}=o(\epsilon^2)\,. 
\end{aligned}
\end{equation*}
Similarly,
\begin{equation*}
\begin{aligned}
\mathbb{D}_{11}&=\bar{\mathbb{E}}\bigg[\bigg|(\bar{\tilde{\rho}}_T+\tilde{\rho}_T^1+\tilde{\rho}_T^2)\bigg\langle\int_0^1\Phi_X(\Gamma_T^{-1}(\bar{Y}_T+Y^1_T+Y^2_T)\\
&+\theta\Gamma_T^{-1}(Y_T^{\epsilon}-\bar{Y}_T-Y^1_T-Y^2_T))d\theta,\Gamma_T^{-1}(Y_T^{\epsilon}-\bar{Y}_T-Y^1_T-Y^2_T)\bigg\rangle\bigg|^2\bigg]=o(\epsilon^2)\,. 
\end{aligned}
\end{equation*}
\begin{equation*}
\begin{aligned}
&\mathbb{D}_{12}=\bar{\mathbb{E}}[|\tilde{\rho}^2_T\langle\Phi_X(\Gamma_T^{-1}\bar{Y}_T),\Gamma_T^{-1}(Y_T^1+Y^2_T)\rangle|^2]\\
&\leq\bar{\mathbb{E}}\bigg[\sup_{0\leq t\leq T}|\tilde{\rho}^2_t|^2|\Phi_X(\Gamma_T^{-1}\bar{Y}_T)|^2\sup_{0\leq t\leq T}|\Gamma_t^{-1}(Y_t^1+Y^2_t)|^2\bigg]\\
&\leq \bigg(\bar{\mathbb{E}}\sup_{0\leq t\leq T}|\tilde{\rho}^2_t|^8\bigg)^{\frac{1}{4}}(\bar{\mathbb{E}}|\Phi_X(\Gamma_T^{-1}\bar{Y}_T)|^4)^{\frac{1}{2}}\bigg(\bar{\mathbb{E}}\sup_{0\leq t\leq T}|\Gamma_t^{-1}(Y_t^1+Y^2_t)|^8\bigg)^{\frac{1}{4}}\leq C\epsilon^3=o(\epsilon^2)\,.\\
\end{aligned}
\end{equation*}
\begin{equation*}
\begin{aligned}
&\mathbb{D}_{13}=\bar{\mathbb{E}}\bigg[\bigg|(\bar{\tilde{\rho}}_T+\tilde{\rho}_T^1+\tilde{\rho}_T^2)\bigg\langle\int_0^1\int_0^1\alpha[\Phi_{XX}(\Gamma_T^{-1}\bar{Y}_T+\beta\alpha\Gamma_T^{-1}(Y_T^1+Y_T^2))\\
&\qquad-\Phi_{XX}(\Gamma_T^{-1}\bar{Y}_T)]d\beta d\alpha\Gamma_T^{-1}(Y_T^1+Y_T^2),\Gamma_T^{-1}(Y_T^1+Y_T^2)\bigg\rangle\bigg|^2\bigg]\\
&\leq \epsilon\bar{\mathbb{E}}\bigg[\sup_{0\leq t\leq T}|\bar{\tilde{\rho}}_t+\tilde{\rho}_t^1+\tilde{\rho}_t^2|^2\sup_{0\leq t\leq T}|\Gamma_t^{-1}(Y_t^1+Y_t^2)|^4\bigg]\\
&\leq \epsilon\bigg(\bar{\mathbb{E}}\sup_{0\leq t\leq T}|\bar{\tilde{\rho}}_t+\tilde{\rho}_t^1+\tilde{\rho}_t^2|^4\bigg)^{\frac{1}{2}}\bigg(\bar{\mathbb{E}}\sup_{0\leq t\leq T}|\Gamma_t^{-1}(Y_t^1+Y_t^2)|^8\bigg)^{\frac{1}{2}}\leq C\epsilon^3=o(\epsilon^2)\,.\\
\end{aligned}
\end{equation*}
\begin{equation*}
\begin{aligned}
\mathbb{D}_{14}&=\bar{\mathbb{E}}\bigg[\bigg|(\bar{\tilde{\rho}}_T+\tilde{\rho}_T^1+\tilde{\rho}_T^2)\langle\Phi_{XX}(\Gamma_T^{-1}\bar{Y}_T)\Gamma_T^{-1}(Y_T^1+Y_T^2),\Gamma_T^{-1}Y_T^2\rangle\bigg|^2\bigg]\\
&\leq C\bar{\mathbb{E}}\bigg[\sup_{0\leq t\leq T}|\bar{\tilde{\rho}}_t+\tilde{\rho}_t^1+\tilde{\rho}_t^2|^2\sup_{0\leq t\leq T}|\Gamma_t^{-1}(Y_t^1+Y_t^2)|^2\sup_{0\leq t\leq T}|\Gamma_t^{-1}Y_t^2|^2\bigg]\\
&\leq C\bigg(\bar{\mathbb{E}}\sup_{0\leq t\leq T}|\Gamma_t^{-1}(Y_t^1+Y_t^2)|^8\bigg)^{\frac{1}{4}}\bigg(\bar{\mathbb{E}}\sup_{0\leq t\leq T}|\Gamma_t^{-1}Y_t^2|^8\bigg)^{\frac{1}{4}}\leq C\epsilon^3=o(\epsilon^2)\,. 
\end{aligned}
\end{equation*}
\begin{equation*}
\begin{aligned}
\mathbb{D}_{15}&=\bar{\mathbb{E}}\bigg[\bigg|(\tilde{\rho}_T^1+\tilde{\rho}_T^2)\langle\Phi_{XX}(\Gamma_T^{-1}\bar{Y}_T)\Gamma_T^{-1}Y_T^1,\Gamma_T^{-1}Y_T^1\rangle\bigg|^2\bigg]\\
&\leq C\bar{\mathbb{E}}\bigg[\sup_{0\leq t\leq T}|\tilde{\rho}_t^1+\tilde{\rho}_t^2|^2\sup_{0\leq t\leq T}|\Gamma_t^{-1}Y_t^1|^2\sup_{0\leq t\leq T}|\Gamma_t^{-1}Y_t^1|^2\bigg]\\
&\leq C\bigg(\bar{\mathbb{E}}\sup_{0\leq t\leq T}|\tilde{\rho}_t^1+\tilde{\rho}_t^2|^4\bigg)^{\frac{1}{2}}\bigg(\bar{\mathbb{E}}\sup_{0\leq t\leq T}|\Gamma_t^{-1}Y_t^1|^8\bigg)^{\frac{1}{2}}\leq C\epsilon^3=o(\epsilon^2)\,. 
\end{aligned}
\end{equation*}
\begin{equation*}
\begin{aligned}
\mathbb{D}_{16}&=\bar{\mathbb{E}}\bigg[\bigg|\tilde{\rho}_T^1\langle\Phi_X(\Gamma_T^{-1}\bar{Y}_T),\Gamma_T^{-1}Y_T^2\rangle\bigg|^2\bigg]\\
&\leq C\bar{\mathbb{E}}\bigg[\sup_{0\leq t\leq T}|\tilde{\rho}_t^1|^2(1+|\Gamma_T^{-1}\bar{Y}_T|^2)\sup_{0\leq t\leq T}|\Gamma_t^{-1}Y_t^2|^2\bigg]\\
&\leq C\bigg(\bar{\mathbb{E}}\sup_{0\leq t\leq T}|\tilde{\rho}_t^1|^8\bigg)^{\frac{1}{4}}\bigg(\bar{\mathbb{E}}(1+|\Gamma_T^{-1}\bar{Y}_T|^4)\bigg)^{\frac{1}{2}}\bigg(\bar{\mathbb{E}}\sup_{0\leq t\leq T}|\Gamma_t^{-1}Y_t^2|^8\bigg)^{\frac{1}{4}}\\
&\leq C\epsilon^3=o(\epsilon^2).
\end{aligned}
\end{equation*}
The proof is complete.

\end{document}